\documentclass[11pt, a4paper]{amsart}
\usepackage{geometry}    
\usepackage{amsmath, amscd} 
\usepackage{amssymb} 
\usepackage{xcolor} 
\usepackage[breaklinks]{hyperref}
\hypersetup{
 colorlinks,
 linkcolor=blue
}
\usepackage{enumitem}
\usepackage{pst-node}
\usepackage{tikz-cd} 
\usepackage{graphicx} 
\usepackage{comment}
\usepackage{xcolor}
\usepackage{mathtools} 
\usepackage[textsize=small]{todonotes}
\usepackage{cleveref}
\usepackage{ stmaryrd } 
\usepackage[percent]{overpic} 
\usepackage{url}

\usepackage{breakurl}
\usepackage[breaklinks]{hyperref} 
\DeclareRobustCommand{\SkipTocEntry}[5]{}

\newtheorem{defnIntro}{Definition}
\newtheorem{thmIntro}[defnIntro]{Theorem}

\newtheorem{prop}{Proposition}[section]
\newtheorem{thm}[prop]{Theorem}
\newtheorem{lemma}[prop]{Lemma}
\newtheorem{cor}[prop]{Corollary}
\theoremstyle{definition}
\newtheorem{defn}[prop]{Definition}

\newtheorem{question}[prop]{Question}
\newtheorem{example}[prop]{Example}
\newtheorem{remark}[prop]{Remark}

\newcommand{\rr}{\mathbb{R}}
\newcommand{\nn}{\mathbb{N}}
\newcommand{\zz}{\mathbb{Z}}
\newcommand{\qq}{\mathbb{Q}}

\newcommand{\ff}{\mathbb{F}}

\newcommand{\kk}{\mathbb{K}}

\renewcommand{\aa}{\mathbb{A}}

\def \slnf {\ensuremath{\mathrm{SL}_n(\mathbb{F})}}

\def \glnf {\ensuremath{\mathrm{GL}_n(\mathbb{F})}}

\newcommand{\frakg}{\mathfrak{g}}
\newcommand{\fraka}{\mathfrak{a}}

\newcommand{\frakk}{\mathfrak{k}}
\newcommand{\frakp}{\mathfrak{p}}

\newcommand{\com}[1]{}
\DeclareMathOperator{\Id}{Id}

\newcommand{\build}{B}
\newcommand{\atlas}{\mathcal{A}}
\newcommand{\affwg}{W_{a}}
\newcommand{\sphwg}{W_{s}}

\tikzset{%
  symbol/.style={
    draw=none,
    every to/.append style={
      edge node={node [sloped, allow upside down, auto=false]{$#1$}}
    },
  },
}

\newcommand{\BT}{\textnormal{BT}}
\newcommand{\pr}{\textnormal{pr}}

\title{Morphisms of generalized affine buildings}
\author[Appenzeller, Flamm, Jaeck]{Raphael Appenzeller, Xenia Flamm, Victor Jaeck}
\date{\today}      
\address{Department of Mathematics, Heidelberg University, Germany}
\email{rappenzeller@mathi.uni-heidelberg.de}

\address{Institut des Hautes \'Etudes Scientifiques, France\newline
\indent Max-Planck Institute for Mathematics in the Sciences, Germany}
\email{flamm@ihes.fr, xenia.flamm@mis.mpg.de}

\address{Department of Mathematics, ETH Z\"{u}rich, Switzerland}
\email{victor.jaeck@math.ethz.ch}

\def\subjclassname{\textup{2020} Mathematics Subject Classification}
\expandafter\let\csname subjclassname@1991\endcsname=\subjclassname

\begin{document}

\begin{abstract}
We define a notion of morphism for generalized affine buildings, also known as affine $\Lambda$-buildings, extending existing definitions and giving rise to a category of generalized affine buildings.
For affine $\Lambda$-buildings equipped with a transitive group action, we provide sufficient conditions for the existence of morphisms between them.
As an application, we investigate under which conditions morphisms or isomorphisms between various generalized affine buildings from the literature (defined via lattices, norms, non-standard symmetric spaces, or \`a la Bruhat--Tits) can be defined.
For generalized affine buildings coming from non-standard symmetric spaces we further show functoriality for subgroups and under change of valued field.
\end{abstract}

\maketitle

\section{Introduction}

Generalized affine buildings, such as (possibly non-discrete) Euclidean buildings, appear in a wide range of mathematical contexts, including the asymptotic geometry of symmetric spaces \cite{KleinerLeebRigidityofquasiisometries}, the structure theory of Kac--Moody groups over valued fields \cite{TitsuniquenessandpresentationofKacMoody, Remyconstructiondereseauxentheoriedekacmoody}, and compactifications of character varieties \cite{BIPPtherealspectrumnewarticle}.
The theory of affine buildings has a long history, beginning with the foundational work of Bruhat and Tits \cite{BruhatTits,BruhatTits84}. It was later extended from discrete to non-discrete settings, leading to Euclidean buildings \cite{KleinerLeebRigidityofquasiisometries}, $\mathbb{R}$-buildings \cite{BruhatTits, Parreau_ImmeublesAffines, RousseauBook23}, and finally affine $\Lambda$-buildings, introduced by Bennett \cite{BennettlambdabuildingsI} and further developed by many authors \cite{KramerTent, Schwer_GeneralizedAffineBuildings, BennettlambdabuildingsII, SchwerStruyve_LambdaBuildingsBaseChangeFunctors, BennettSchwerStruyve_OnAxiomaticDefNonDiscreteAffineBuildings,  HebertIzquierdoLoisel_LambdaBuildingsAssocToQuasiSplitGroups, appenzeller2024semialgebraic}.
However, there is still no general theory of morphisms that applies uniformly across these settings. The main difficulty is that different contexts give rise to distinct models of generalized affine buildings, often of different types, making it hard to define morphisms or subbuildings in a systematic way.

In this article, we introduce a new notion of morphisms between generalized affine buildings that addresses this challenge.
These morphisms preserve the apartment structure and use a new notion of morphisms of apartments that allows for different types of root systems, while still preserving the action of the affine Weyl group. 
Composing morphisms yields another morphism, thus defining two categories, whose objects are apartments and generalized affine buildings respectively. 
Our theory is particularly well-suited for buildings equipped with a natural group action on their atlases.
In this setting, we construct morphisms under mild hypotheses and identify additional conditions ensuring that they are injective, surjective, or isomorphisms; see \Cref{intro:thm:baby_morphismofGbuildings}.
The main advantage of our approach is that these conditions are easy to verify in explicit examples. 
This result applies to several well-known constructions of generalized affine buildings, defined via lattices, norms, non-standard symmetric spaces, or \`a la Bruhat--Tits, and permits to relate them.
This brings some clarity to the subject as there is a fair bit of confusion about terminology.
Furthermore the new notion of morphisms allows to prove functoriality properties under valued field extensions (\Cref{intro:corol_FieldExt}) and under injective group morphisms (\Cref{intro:thm:InjMorphBuildings}) for the generalized affine buildings defined via non-standard symmetric spaces.
Although our results are formulated in the general setting of affine $\Lambda$-buildings, the most relevant case is $\Lambda = \mathbb{R}$, where the results are new even for non-discrete affine buildings.
We now describe the main results in more detail.

\subsection{A new notion of morphisms of apartments and buildings}
We begin by defining the new notions of morphisms of apartments and generalized affine buildings.
Recall that a \emph{generalized affine building}, or just a building, is a set $\build$ together with an atlas of maps $\atlas$ from the model apartment $\mathbb{A}$ to $\build$ satisfying certain compatibility axioms; see \Cref{dfn:Building} for the precise conditions.
The \emph{model apartment} is given by $\mathbb{A}=\operatorname{Span}_{\mathbb{Q}}(\Phi) \otimes_\mathbb{Q} \Lambda$, where $\Phi$ is a root system and $\Lambda$ an ordered abelian group, together with an action of the \emph{affine Weyl group} defined by $W_{T,\Phi} \coloneqq T \rtimes \sphwg(\Phi)$.
Here $\sphwg(\Phi)$ is the (spherical) Weyl group associated to $\Phi$, and $T$ is a subgroup of the abelian group $\aa$ which acts by translation on $\aa$.
When $T$ and $\Phi$ can be deduced from the context, we write $\affwg$ for the affine Weyl group.
In this case the apartment $\mathbb{A}$ (and the building $(\build,\atlas)$ as well) is said to be \emph{of type $\mathbb{A}(\Phi,\Lambda,T)$}.

The main motivation for our definition is to be able to account for buildings of different types, especially when the root systems are different.
For example, for buildings associated to algebraic groups, a subgroup generally has a root system that does not admit an inclusion to the root system of the ambient group, see e.g.\ the inclusion $\operatorname{Sp}_2 < \operatorname{SL}_4$ where the root systems are of type $B_2$ and $A_3$. 

\begin{figure}[h]
    \centering
    \includegraphics[width=0.4\linewidth]{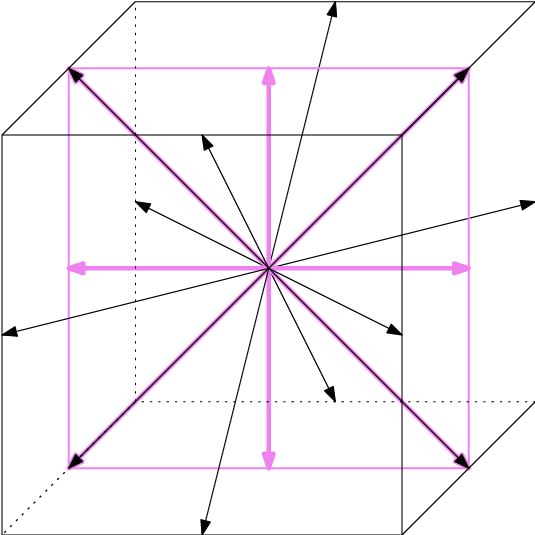}
    \caption{The root systems of type $B_2$ (in pink) and $A_3$ (in black).}
    \label{fig:intro:Sp4inSL4}
\end{figure}

Let now $\Phi$ and $\Phi'$ be two crystallographic root systems and $\Lambda$ and $\Lambda'$ two ordered abelian groups, that are also $\qq$-vector spaces.

\begin{defnIntro}[{\Cref{def:morphism_apartment}}]
\label{intro:def:morphism_apartment}
    Let $\mathbb{A} = \mathbb{A}(\Phi, \Lambda, T)$ and $\mathbb{A}' = \mathbb{A}(\Phi', \Lambda', T')$ be two model apartments.
    A \emph{morphism of apartments} is a triple $(L,\gamma, \sigma)$, where $L\colon \operatorname{Span}_{\mathbb{Q}}(\Phi) \rightarrow \operatorname{Span}_{\mathbb{Q}}(\Phi')$ is a $\mathbb{Q}$-linear map, $\gamma \colon \Lambda \rightarrow \Lambda '$ is a morphism of ordered abelian groups (an order-preserving group homomorphism) and $\sigma \colon \affwg \to \affwg'$ is a group homomorphism, such that for all $x \in \mathbb{A}$ and $w \in \affwg$
\[\tau (w.x) = \sigma(w) . \tau (x)\]
for $\tau \coloneqq L \otimes \gamma \colon \aa \to \aa'$.
\end{defnIntro}

Since a building is made out of copies of the model apartment, the notion of morphism of apartments naturally extends to the following notion of morphism of buildings.

\begin{defnIntro}[\Cref{def:morphism_building}]
\label{intro:def:morphism_building}
    Let $\build= (\build,\atlas)$ be a building of type $\mathbb{A}= \mathbb{A}(\Phi,\Lambda,T)$ and $\build'=(\build',\atlas')$ a building of type $\mathbb{A}'=\mathbb{A}(\Phi',\Lambda',T')$. A \emph{morphism of generalized affine buildings} is a triple $(\psi,\varphi,\tau)$, where $\psi \colon \build \to \build'$ and $\varphi \colon \atlas  \to \atlas'$ are maps, and $\tau \colon \aa \to \aa'$ is a morphism of apartments, such that for all $f \in \atlas$ we have
    \[\psi \circ f = \varphi(f) \circ \tau. \]
\end{defnIntro}

We say that a morphism of buildings is \emph{injective} (resp.\ \emph{surjective}) if $\psi$, $\varphi$ and $\tau$ are injective (resp.\ surjective).
On the one hand, this definition provides considerable flexibility compared to existing notions of building morphisms in the literature; for instance, it is flexible enough to relate buildings associated with different root systems.
On the other hand, it remains rigid enough to prevent pathological maps: for example, the only morphism from the $\mathbb{R}$-building $\mathbb{R}$ to the $\mathbb{Q}$-building $\mathbb{Q}$ is trivial, reflecting the absence of non-trivial order-preserving group homomorphisms $\mathbb{R} \to \mathbb{Q}$ (\Cref{example: R and Q building}).

Notions of isomorphisms between buildings have been studied since the early works of Tits and Bruhat–Tits; see for instance \cite{Titssphericalbuildings, Titsaffinebuildings, BruhatTits}, and for affine $\Lambda$-buildings in \cite{Schwer_GeneralizedAffineBuildings, SchwerStruyve_LambdaBuildingsBaseChangeFunctors}.
In the discrete setting—when $\Lambda$ is a discrete subgroup of $\mathbb{R}$ and the building is thus simplicial—various notions of morphisms and subbuildings have been proposed and studied, see \cite{Landvogt, kaletha2023bruhat}.
A more metric perspective is taken in \cite{KleinerLeebRigidityofquasiisometries, RousseauBook23}, where morphisms are viewed as isometries between spaces endowed with a building structure.
A more detailed discussion can be found in \Cref{section:ExistingNotions}.
 
\subsection{Extension of morphisms of apartments to morphisms of buildings}
We construct morphisms of generalized affine buildings from morphisms of apartments when a group acts transitively on the atlas of the building.
The idea is to translate the morphism between the standard apartments using the group action.
Let us now make this idea precise.

Let $G$ be a group and $(\build, \atlas)$ an affine $\Lambda$-building of type $\mathbb{A}= \mathbb{A}(\Phi,\Lambda,T)$.
We say that $\build$ is a \emph{$G$-building} if $G$ acts on both $\build$ and $\atlas$ in a compatible way, i.e.\ for all $g \in G$, $f \in \atlas$ and $x \in \mathbb{A}$ it holds
\[(g.f)(x)=g.f(x).\]
If $G'$ is another group and $\rho \colon G\to G'$ is a group homomorphism, a morphism ($\psi,\varphi,\tau$) from a $G$-building to a $G'$-building is $\rho$-\emph{equivariant} if both $\psi$ and $\varphi$ are $\rho$-equivariant.

The main result is the following theorem, which gives sufficient conditions on when morphisms of apartments extend to (injective or surjective) equivariant morphisms of buildings. For this version of the theorem, we assume that $G$ acts transitively both on $\build$ and $\atlas$, but there is another theorem, where we do not need to assume that $G$ acts transitively on $\build$, see \Cref{thm:morphismofGbuildings}.
An important role is played by stabilizers of points and charts. 

\begin{thmIntro}
[\Cref{thm:baby_morphismofGbuildings}]
\label{intro:thm:baby_morphismofGbuildings}
    Let  $(\build,\atlas)$ be a $G$-building and $(\build',\atlas')$ a $G'$-building. Let $\rho \colon G \to G'$ be a group homomorphism and $\tau \colon \aa \to \aa'$ a morphism of the respective model apartments of $\build$ and $\build'$.
    If $G$ acts transitively both on $\build$ and on $\atlas$ and there exist charts $f \in \atlas$ and $f' \in \atlas'$ such that 
    \begin{enumerate}[label=(\arabic*)]
        \item \label{intro:condition:baby_stabilizerofapoint} 
        $\rho(\operatorname{Stab}_G(f(0))) \subseteq \operatorname{Stab}_{G'}(f'(0))$, and
        \item \label{intro:condition:baby_stabilizerofchart} 
        $\rho(\operatorname{Stab}_G(f)) \subseteq \operatorname{Stab}_{G'}(f')$, and
        \item \label{intro:condition:baby_Afw} 
        for all $x\in \aa$, there exists $g\in G$ such that 
        \[
        g.f(0) = f(x) \quad \text{ and } \quad \rho(g).f'(0) = f'(\tau(x)),
        \]
    \end{enumerate}
    then there exists a $\rho$-equivariant morphism $m\coloneqq(\psi, \varphi, \tau) \colon \build \to \build'$.
    
    Additionally,
    \begin{enumerate}[label=(\alph*)]
        \item \label{intro:thm:baby_monomorphism}
        if $\tau$ and $\rho$ are injective and $\rho(\mathrm{Stab}_G(f))=\mathrm{Stab}_{G'}(f')$, then $m$ is injective;
        
        \item \label{intro:thm:baby_epimorphism}
        if $G'$ acts transitively on $\build'$ and $\atlas'$, and $\tau$ and $\rho$ are surjective, then $m$ is surjective; 
        
        \item \label{intro:thm:baby_isomorphism}
        if $\tau$ is an isomorphism, $G'$ acts transitively on $\build'$ and $\atlas'$, $\rho$ is an isomorphism of groups, and the two inclusions \ref{condition:baby_stabilizerofapoint} and \ref{condition:baby_stabilizerofchart} hold as equalities, then there exists an inverse morphism. That is, $(\build,\atlas)$ and $(\build',\atlas')$
        are isomorphic.
    \end{enumerate}
\end{thmIntro}

We say that the $\rho$-equivariant morphism constructed in the theorem above is \emph{induced} by the group homomorphism~$\rho$.
The induced morphism is unique, see \Cref{rem:unique}.

The strength of this result lies in its applicability: in concrete examples, the apartment morphism $\tau$ is easy to define, and the required conditions can often be verified directly.
This will be illustrated in the following results, which show how the theorem applies in several explicit settings.

\subsection{Applications and examples}
The main applications of the above result are twofold.
First it allows to relate diverse models of generalized affine buildings in the literature.
Secondly, it allows to prove functoriality properties for symmetric buildings.

\subsubsection{Relationship between different models of buildings} 
 
In this article we focus on four models of (families of) buildings, namely the \emph{norm buildings} $\build_N$  \cite{GoldmanIwahori63,Parreau_ImmeublesAffines,RousseauBook23}, the \emph{lattice buildings} $\build_L$ \cite{BennettlambdabuildingsI}, the \emph{Bruhat--Tits buildings} $\build_\BT$ (for split algebraic groups) \cite{BruhatTits, BruhatTits84, HebertIzquierdoLoisel_LambdaBuildingsAssocToQuasiSplitGroups} and the \emph{symmetric buildings} $\build_S$ obtained from non-standard symmetric spaces \cite{KramerTent,appenzeller2025buildings}.
Their precise definitions are given in \Cref{subsection:ExamplesNonDiscBuildings}.

All four models above are constructed with the help of some algebraic group $\mathbf{G}$, some field $\ff$ and some valuation $v\colon \ff^\times \twoheadrightarrow \Lambda$. In \Cref{subsection:ExamplesGBuildings} we show that all these buildings constitute examples of $\mathbf{G}(\ff)$-buildings for appropriate $\mathbf{G}$ and $\ff$.
In \Cref{section:RelationsBuildings}, we use \Cref{intro:thm:baby_morphismofGbuildings} to construct morphisms 
\[
\build_L \xrightarrow[\cong]{(\textnormal{\Cref{thm:Lattice_to_KT} })} 
\build_S \xhookrightarrow[]{(\textnormal{\Cref{corol:MorphismKramerTentToBruhatTits} })} 
\build_\BT \xhookrightarrow[]{(\textnormal{\Cref{thm:BTNorm} })} 
\build_N
\]
between the models whenever both the source and the target buildings exist for a given choice of group and valued field. 
When $\Lambda = \rr$, all these morphisms are isomorphisms, which is a folklore result that we make precise.
Through personal communication, the proof of the isomorphism between $\build_\BT$ and $\build_N$ (see \Cref{thm:BTNorm}) when $\Lambda=\rr$ was known to Anne Parreau.
Already much earlier, the strong relation between these two buildings was observed by Bruhat--Tits, see  \cite[Note ajoutée sur épreuves]{BruhatTits} and \cite{BrTi84_norms}.

We believe that it is also possible to find $\glnf$-equivariant morphisms $\build_L \to \build_\BT$ and $\build_L \to \build_N$, even though in general this does not directly follow from composing the above morphisms (as the symmetric building is only defined when $\ff$ is real closed).

To apply \Cref{intro:thm:baby_morphismofGbuildings}, substantial knowledge about stabilizers in the buildings is needed. This theory for the buildings $\build_S$, $\build_\BT$ and $\build_N$ is well established in the literature, but for the lattice buildings $\build_L$ we provide the missing statements in \Cref{sec: L_H_Setup}.
Especially, \Cref{prop:L_H_stabApartment} describes the stabilizers of apartments in this setting, which to our knowledge is new.

\subsubsection{Functoriality properties}
We now use the notion of morphism to address certain functoriality questions for symmetric buildings (\Cref{ex:SymmetricBuilding1}). 
Symmetric buildings are associated to the $\ff$-points of a semisimple self-adjoint connected linear algebraic group $\mathbf{G}< \operatorname{SL}_n$ defined over $\qq$ with reduced root system, and a maximal $\rr$-split torus $\mathbf{S}$ all of whose elements are self-adjoint, where $\ff$ is a real closed field with order-compatible valuation. All algebraic groups and fields in this section are assumed to be of this form, so that the symmetric buildings are defined.

Functoriality with respect to valued field extensions has already been investigated for Bruhat--Tits buildings \cite{Rousseau}, where one asks under which conditions a morphism of valued fields induces an equivariant injection between the associated buildings.
Such functoriality fails in general \cite[III.\ \S 5]{Rousseau}, but holds in the split and quasi-split cases, and more generally under additional technical assumptions on the fields, see  \cite[5.1.2 and errata]{Rousseau} or \cite[1.3.4]{RemyThuillierWerner}.

For symmetric buildings, we obtain the following functoriality property under morphisms of ordered valued fields.

\begin{thmIntro}[{\Cref{thm:corol_FieldExt}}]
\label{intro:corol_FieldExt}
Let $\eta \colon \ff \rightarrow \ff'$ be a morphism of ordered valued fields and $\rho \colon \mathbf{G}(\ff) \to \mathbf{G}(\ff')$ the induced group homomorphism. Then there exists an equivariant morphism of buildings $m \colon \build \rightarrow \build'$, where $\build$ and $
\build'$ are the symmetric buildings associated to $\mathbf{G}(\ff)$ and~$\mathbf{G}(\ff')$. 

   Additionally, if $\gamma \colon \Lambda \to \Lambda'$ denotes the morphism of the value groups induced by $\eta$, then
    \begin{enumerate}[label=(\alph*)]
        \item
        \label{item:intro: injective eta and gamma gives injective morphism}
        if $\gamma$ is injective, then $m$ is injective;
        \item
        \label{item:intro: surjective eta and gamma gives surjective morphism} 
        if $\eta$ is surjective, then $m$ is surjective;
        \item
        \label{item:intro: bijective eta and gamma gives bijective morphism}
        if $\eta$ and $\gamma$ are isomorphisms, then $m$ is an isomorphism.
    \end{enumerate}
\end{thmIntro}

The proof uses \Cref{intro:thm:baby_morphismofGbuildings} applied to an apartment morphism of the form $(\Id,\gamma,\sigma)$, where $\sigma$ is the identity on the spherical Weyl group, and given by $\Id \otimes \gamma$ on translations.

In a more general setting, an analogous statement to the above theorem 
was proven by Schwer--Struyve \cite{SchwerStruyve_LambdaBuildingsBaseChangeFunctors}; see also \Cref{subsection: functoriality generalized affine buildings} for more details.
Given an affine $\Lambda$-building $B$ and a morphism of ordered abelian groups $\gamma \colon \Lambda \to \Lambda'$, they construct an affine $\Lambda'$-building $B'$, and a morphism $B \to B'$.
When $B$ is the symmetric building associated to $\mathbf{G}(\ff)$, we expect their morphism $B \to B'$ to be the same as the one constructed in \Cref{intro:corol_FieldExt}, and $B'$ to be isomorphic to the symmetric building associated to $\mathbf{G}(\ff')$. \\

Secondly, a natural question is whether a (surjective, injective) morphism of algebraic groups $\rho \colon \mathbf{G} \to \mathbf{G'}$, where $\mathbf{G'}$ is another semisimple linear algebraic group as above, induces a natural (surjective, injective) morphism of the associated symmetric buildings $\build$ respectively $\build'$.
A similar question was answered positively for (discrete) Bruhat--Tits buildings over quasi-local fields by Landvogt in \cite{Landvogt} for his notion of morphisms.
The first example is when $\mathbf{G}$ is a subgroup of $\mathbf{G'}$ and $\rho$ is the inclusion.
In this case we have the following result.

\begin{thmIntro}[{\Cref{thm:FunctorialitySubgroups}}]
\label{intro:thm:FunctorialitySubgroups} Let $\iota$ denote the inclusion $\mathbf{G(\ff)} \hookrightarrow \mathbf{G}'(\ff)$ and assume there exist maximal $\rr$-split tori $\mathbf{S}<\mathbf{S'}$ of $\mathbf{G}$, $\mathbf{G'}$ all of whose elements are self-adjoint.
If $\mathbf{G}$ is $\rr$-split, then $\iota$ induces an equivariant injective morphism $\build \to \build'$ of the corresponding symmetric buildings.
\end{thmIntro}

This theorem highlights the full flexibility of our notion of morphism of apartments, as it requires changing the root system of the model apartment.
The key step is to relate the root systems of $\mathbf{G}$ and $\mathbf{G}'$ using the Cartan decomposition for semisimple algebraic groups over real closed fields, see \Cref{lem:condition_cells_for_symmetric_building}.
The main technical challenge is then establishing the existence of a map $\sigma \colon \sphwg \to \sphwg'$, see \Cref{prop:subroot}.
The result then follows from \Cref{intro:thm:baby_morphismofGbuildings}~\ref{intro:thm:baby_monomorphism}.
For more general monomorphisms, we obtain the following theorem, which extends the above result. 

\begin{thmIntro}[{\Cref{thm:inj_group_morphism_implies_inj_morphism_buildings}}]
\label{intro:thm:InjMorphBuildings}
    Let $\rho \colon \mathbf{G} \to \mathbf{G}'$ be a transposition-invariant monomorphism of algebraic groups.
    If $\mathbf{G}$ is $\rr$-split, then there exists an equivariant injective morphism  $\build \to \build'$ of the corresponding symmetric buildings.
\end{thmIntro}
\begin{remark}\label{rmk:independence_S}
    A priori, the construction of the symmetric building $B_S$ depends not only on a group $\mathbf{G}$ and a valued field $\ff$, but also on the choice of a maximal $\rr$-split torus $\mathbf{S}$, all of whose elements are self-adjoint. However, such $\mathbf{S}$ always exist \cite[Theorem 5.17]{appenzeller2024semialgebraic} and are conjugate to each other \cite[Theorem 20.9]{Borellinearalgebraixgroups}. A corollary of  \Cref{intro:thm:InjMorphBuildings} is that up to isomorphism, the building $B_S$ does not depend on $\mathbf{S}$. In particular, the condition on the maximal tori in \Cref{intro:thm:FunctorialitySubgroups} is not necessary.
\end{remark}
The proof of \Cref{intro:thm:InjMorphBuildings} uses \Cref{intro:thm:FunctorialitySubgroups}.
We first show that if $\mathbf{G}$ and $\mathbf{G'}$ are isomorphic, then there is an isomorphism between their associated symmetric buildings, see \Cref{prop:isomorphism_of_group_implies_isomorphism_of_buildings}.
This first needs to be established for the associated model apartments (\Cref{prop:FunctorialitySubgroups_apartments}), and then we use again \Cref{intro:thm:baby_morphismofGbuildings}~\ref{intro:thm:baby_isomorphism}.

We leave open the question of whether every (not necessarily mono-) morphism of semisimple algebraic groups $\mathbf{G} \to \mathbf{G}'$ induces a morphism of the associated buildings.
In that case, the symmetric building (for a specific choice of a valued real closed field $\ff$) would be a functor from the category of semisimple algebraic groups to the category of affine buildings. 

\begin{question}
\label{intro:qu:symm_functor_simple}
    Can the symmetric building be seen as a functor from some category of semisimple algebraic groups to the category of affine buildings?
\end{question}

\begin{remark}
    We believe that most of these functoriality properties for the symmetric buildings can be extended to other models of buildings, for example to Bruhat--Tits buildings.
    However, our proof uses the theory of symmetric spaces and real algebraic geometry, and different proof techniques may be needed for general fields.
\end{remark}

Another loose end are the spherical buildings associated to an affine building.
One may ask whether a morphism of generalized affine buildings induces a simplicial morphism of their respective buildings at infinity or, equivalently, their local buildings.
The answer in general is no, this may not be true, when the affine buildings are not modeled on the same root system, see for example $\operatorname{Sp}_4(\ff) < \operatorname{SL}_4(\ff)$. 
In this case, a chamber of the smaller apartment is not sent to any Weyl simplex (of any dimension) in the larger apartment.
However, when the root systems agree, we believe that our notion of morphism between generalized affine buildings induces a simplicial morphism between the spherical buildings.

\subsection{Structure of the paper}
The article is organized as follows.
In \Cref{section:Buildings} we recall the definition of generalized affine buildings.
The background on valued fields is summarized in \Cref{subsection:ValuedFields}.
We then continue to present all the examples of buildings that will be studied throughout this article in \Cref{subsection:ExamplesNonDiscBuildings}.
We define morphisms of apartments and buildings in \Cref{section:Morphisms}.
In order to prove \Cref{intro:thm:baby_morphismofGbuildings} in \Cref{subsection:ExtendingMorphismApartments}, we introduce $G$-buildings in \Cref{section:GBuildings}, and we investigate the notion of morphisms and the above examples of buildings within this new context.
We then apply \Cref{intro:thm:baby_morphismofGbuildings} in \Cref{section:RelationsBuildings} and \Cref{section:Functoriality} to prove \Cref{intro:corol_FieldExt}, \Cref{intro:thm:FunctorialitySubgroups} and \Cref{intro:thm:InjMorphBuildings}, and the results on the relations between the different examples.
Necessary background on ordered fields and real algebraic geometry, which is only needed for the example of the symmetric buildings, is given in \Cref{appendix} and can be consulted at any moment.

\addtocontents{toc}{\SkipTocEntry}
\subsection*{Acknowledgments}
The authors would like to thank Jacques Audibert, Marc Burger, Auguste Hébert, 
Anne Parreau, José Pedro Quintanilha, Bertrand Rémy, Guy Rousseau and Petra Schwer for interesting discussions.
The authors are particularly grateful to Anne Parreau for sharing her extensive knowledge of buildings.
Her insights were especially valuable in finding a suitable definition of morphism of apartments.

R.A.\ acknowledges support from the Procope project ``Buildings, galleries and beyond''.
X.F.\ is funded by the Max-Planck Institute for Mathematics in the Sciences and the Labex Carmin project.
This project has received funding from the European Research Council (ERC) under the European Union’s Horizon 2020 research and innovation programme (grant agreement No 101018839), ERC GA 101018839.

\setcounter{tocdepth}{2}
\tableofcontents

\newpage

\section{Generalized affine buildings}
\label{section:Buildings}
In this section we recall the definition of generalized affine buildings, and then give several examples.
\subsection{Definition}
\label{subsection:AffineLambdaBuildings}
For more details and a thorough introduction to generalized affine buildings we recommend for example \cite{BennettlambdabuildingsI,Schwer_GeneralizedAffineBuildings, appenzeller2024semialgebraic}. 

\begin{defn}(\cite[Chapter VI \S 1]{Bourbaki2002} and \cite[14.7]{Borellinearalgebraixgroups}) Let $V$ be a $\qq$-vector space and $V^\star$ its dual. A subset $\Phi \subseteq V$ is a \emph{root system} if
\begin{enumerate}
    \item[($\text{RS}_{\text{I}}$)] $\Phi$ is finite, $0 \notin \Phi$, and $\Phi$ generates $V$; and
    \item[($\text{RS}_{\text{II}}$)] for all $\alpha \in \Phi$, there exists $\alpha^\vee \in V^\star$, such that $\alpha^\vee(\alpha) = 2$, and $r_\alpha(\Phi) = \Phi$ for the reflections defined by $r_\alpha(x) = x-\alpha^\vee(x) \alpha$.
\end{enumerate}
The reflections $r_\alpha$ and the coroots $\alpha^\vee$ are uniquely determined by $\alpha \in \Phi$ and there is a scalar product $\langle \cdot, \cdot \rangle$ on $V$ such that $\alpha^\vee(x) = 2\langle \alpha, x \rangle / \langle \alpha, \alpha\rangle$. A root system $\Phi$ is called \emph{crystallographic} if
\begin{enumerate}
    \item[($\text{RS}_{\text{III}}$)] for all $\alpha \in \Phi$, $\alpha^\vee(\Phi) \subseteq \zz$.
\end{enumerate}
\end{defn}

We restrict ourselves to crystallographic root systems, as all root systems that come from algebraic groups are crystallographic, but affine $\Lambda$-buildings have been defined more generally \cite{Schwer_GeneralizedAffineBuildings}.
Let $\Phi$ be a crystallographic root system in a vector space $V$ and $\Lambda$ a non-trivial ordered abelian group.
As is usually done, we assume that $\Lambda$ is a $\mathbb{Q}$-vector space, as otherwise we may replace $\Lambda$ by $\Lambda \otimes_{\mathbb{Z}} \mathbb{Q}$.
In particular, both $\Lambda$ and $\operatorname{Span}_{\mathbb{Q}}(\Phi)$ have the structure of $\mathbb{Q}$-vector spaces and we define the \emph{model apartment} as 
\[
\mathbb{A} \coloneqq \operatorname{Span}_{\mathbb{Q}}(\Phi) \otimes_\mathbb{Q} \Lambda .
\]
If $\Delta \subseteq \Phi$ is a basis of $\Phi$, then a model for the apartment is given by
\[
\mathbb{A} = \left\{ \sum_{\alpha \in \Delta} \lambda_\alpha \alpha \colon \lambda_\alpha \in \Lambda \right\},
\]
so that $\mathbb{A}$ is isomorphic as a group to $\Lambda^n$ for $n = |\Delta| \in \mathbb{N}$, which is called the \emph{dimension of the apartment}. 
Moreover, the root system $\Phi$ defines a spherical Weyl group $\sphwg$.
Let $T$ be a subgroup of $\aa \cong \Lambda^n$ which acts by translation on $\aa$ and is normalized by $\sphwg$. We define the \emph{affine Weyl group with respect to $T$} as $\affwg \coloneqq T \rtimes \sphwg $.
Formally, we call the combined data $(\Phi,\Lambda,T)$ an \emph{apartment}. An apartment determines the model apartment $\aa$ together with the action of the affine Weyl group $\affwg$.
It is customary to write $\aa = \aa(\Phi,\Lambda,T)$. 

\begin{figure}[ht]
  \centering
  \begin{overpic}[width=0.65\linewidth]{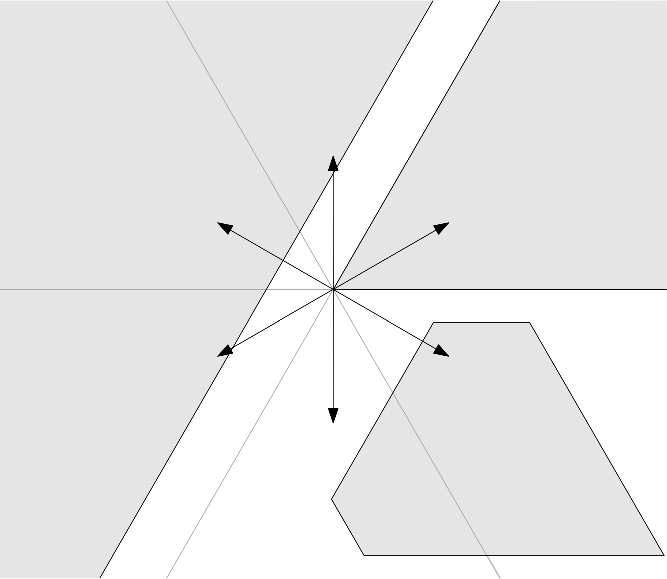}
    \put(10,70){$H_{\delta_1,k}^-$}
    \put(68,32){$\delta_1$}
    \put(45,63){$\delta_2$}
    \put(85,70){$C_0$}
  \end{overpic}
  \caption{An apartment of type $\aa(\Phi, \rr,\rr^2)$ where $\Phi$ is a root system of type $A_2$. Depicted are the fundamental Weyl chamber $C_0$ associated to the basis $\Delta = \{\delta_1, \delta_2\}$, a half-apartment $H_{\delta_1,k}^-$ and a closed subset.}
  \label{fig:apartment}
\end{figure}
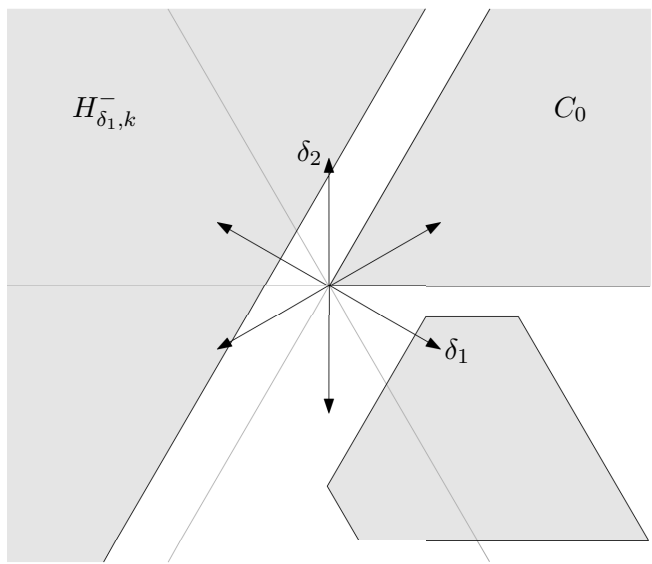

The scalar product $\langle \cdot, \cdot \rangle$ on $V$ extends to a bilinear pairing
\[
\langle \cdot, \cdot \rangle \colon \aa \times \operatorname{Span}_\qq(\Phi) \to \Lambda, \quad 
\left\langle \sum_{\alpha \in \Delta} \lambda_\alpha \alpha , \sum_{\delta \in \Delta} q_\delta\delta\right\rangle \coloneqq \sum_{\alpha,\delta \in \Delta} \lambda_\alpha q_\delta \langle \alpha, \delta \rangle ,
\]
which in general cannot be extended to all of $\aa \times \aa$, since $\Lambda$ is only a $\qq$-vector space and may not have a multiplication. The apartment can be equipped with a $\sphwg$-invariant $\Lambda$-valued norm given by
$$
\lVert x \rVert \coloneqq \sum_{\alpha \in \Phi} \lvert \langle x , \alpha \rangle \rvert,
$$
and the induced $\Lambda$-valued distance, though other norms on $\aa$ may also be considered.
Every root $\alpha \in \Phi$ determines a reflection $r_\alpha \colon \aa \to \aa$ defined by
\[
r_\alpha\left(\sum_{\delta \in \Delta} \lambda_\delta \delta\right) 
\coloneqq \sum_{\delta \in \Delta} \lambda_\delta r_\alpha(\delta) 
= \sum_{\delta\in \Delta} \lambda_\delta \left( \delta - 2\frac{\langle \delta, \alpha \rangle}{\langle \alpha, \alpha \rangle} \alpha \right) .
\]
Elements of the affine Weyl group $\affwg$ that are conjugate to $r_\alpha$ are called \emph{reflections}.
Every reflection $r$ determines a hyperplane
\[
M_r \coloneqq \left\{ x \in \aa \colon r(x) = x  \right\},
\]
which is also called a \emph{wall}. If $r= t r_\alpha  t^{-1}$ for some $\alpha \in \Sigma$ and $t\in T$ a translation, $M_{r} = \{x\in \aa \colon \langle x, \alpha \rangle = k\}$ for $k \coloneqq  \langle t, \alpha \rangle$. Associated to each wall there are thus two \emph{half-apartments} of $\mathbb{A}$ which are of the form
\[
H_{\alpha,k}^+ \coloneqq \left\{  x \in \aa \colon  \langle x , \alpha \rangle \geq k \right\} \quad \text{ and } \quad H_{\alpha,k}^- \coloneqq \left\{  x \in \aa \colon  \langle x , \alpha \rangle \leq k \right\}
\]
for $\alpha \in \Phi$ and $k \in \Lambda$.
The \emph{fundamental Weyl chamber} associated to a basis $\Delta \subseteq \Phi$ is given by
\[
C_0 \coloneqq \bigcap_{\alpha \in \Delta} H_{\alpha,0}^+.
\]
A \emph{Weyl chamber} of $\aa$ is any of the sets $w(C_0)$ for $w \in \affwg$. 

Following \cite[\S 2.5]{BennettlambdabuildingsI}, we say that a subset $\Omega \subset \mathbb{A}$ is \emph{convex}, if it is an intersection of half-apartments. 
A convex set $\Omega \subset \mathbb{A}$ is \emph{closed}, if it is the intersection of finitely many half-apartments.
The definition of affine $\Lambda$-building resembles that of a manifold, where we we have a set $B$ together with an atlas $\atlas$ consisting of charts mapping the model apartment into $B$.
The images under charts of the model apartment, walls, half-apartments, Weyl chambers and closed sets are called \emph{apartments, walls, half-apartments, sectors and closed sets}. If a sector $s$ is a subset of another sector $s'$, then $s$ is called a \emph{subsector} of $s'$.
We follow the definition in \cite[\S 3.1]{BennettlambdabuildingsI}, based on ideas of Tits \cite{Titsaffinebuildings} and generalizing the notion of $\Lambda$-trees in Morgan--Shalen \cite{MorganShalen_ValuationsTressDegenerationsHyperbolicStructuresI}. 

\begin{defn}
\label{dfn:Building}
Let $\build$ be a set and $\atlas$ a set of maps from $\mathbb{A}$ to $\build$. 
We say that $(\build,\atlas)$ is an \emph{affine $\Lambda$-building of type $\mathbb{A}= \mathbb{A}(\Phi,\Lambda,T)$}, if it satisfies the following six axioms:
\begin{enumerate}[label=(A\arabic*)]
    \item \label{axiom:A1}
    For all $f \in \atlas$ and $w \in \affwg$ we have $f \circ w \in \atlas$.
    
    \item \label{axiom:A2}
    For all $f,f'\in \atlas$, the set $\Omega \coloneqq f^{-1}(f(\mathbb{A}) \cap f'(\mathbb{A})) \subseteq \mathbb{A}$ is a closed set and there exists $w \in \affwg$ such that $f|_\Omega = f' \circ w |_\Omega$.
    
    \item \label{axiom:A3}
    For any two points in $\build$ there is an apartment containing both. That is, for any $p,q\in \build$, there exists $f \in \atlas$ such that $p,q \in f(\mathbb{A})$.
    
    \item \label{axiom:A4}
    Given sectors $s_1, s_2$ there exist subsectors $s_1' \subseteq s_1$, $s_2' \subseteq s_2$, such that $s_1'$ and $s_2'$ are contained in a common apartment.

    \item \label{axiom:A5}
    If three apartments pairwise intersect in a halfapartment, then they intersect non-trivially.
    
    \item \label{axiom:A6}
    For any chart $f\in \atlas$ and any point $p \in f(\mathbb{A})$, there is a distance-diminishing retraction $r_{f,p} \colon \build \to f(\mathbb{A})$ with $(r_{f,p})^{-1}(p) = \{p\}$.
\end{enumerate}
\end{defn}

Note that axiom \ref{axiom:A6} makes sense since axioms \ref{axiom:A1}-\ref{axiom:A3} can be used to define a 
$\Lambda$-distance $d$ on $B$ from the $\Lambda$-distance on $\aa$, except that it is not clear whether it satisfies the triangle inequality, see \cite[Remarks 3.1 and 3.2]{BennettlambdabuildingsI}.  
Axiom \ref{axiom:A6} implies that $d$ in fact satisfies the triangle inequality, hence defines a $\Lambda$-distance on $\build$.

If $\Lambda$ can be inferred from the context, we say that $\build$ is a \emph{generalized affine building}, or sometimes only building.
The set $\atlas$ is called the \emph{atlas} of the generalized affine building, and the elements of $\atlas$ are called \emph{charts}.

An important source of generalized affine buildings comes from algebraic groups over valued fields.

\subsection{Valued fields}
\label{subsection:ValuedFields}
For a thorough introduction to the theory of valuations and valued fields we refer to \cite{EnglerPrestel_ValuedFields}.
Let $\kk$ be a field and $\Lambda$ an \emph{ordered abelian group}, i.e.\ $\Lambda$ is an abelian group together with a total order that is compatible with the group operations.

\begin{defn}
A map $v \colon \kk \to \Lambda \cup \{\infty\}$ is called a \emph{$\Lambda$-valuation} (or short just \emph{valuation}) if $v$ is surjective 
and satisfies the following three conditions for all $x, y \in \kk$:
\begin{itemize}
    \item $v(x)=\infty \iff x = 0$,
    \item $v(xy)=v(x)+v(y)$,
    \item $v(x+y) \geq \min\{v(x),v(y)\}$.
\end{itemize}
\end{defn}
If $\Lambda = \{0\}$, we call $v$ the \emph{trivial valuation}; if $\Lambda$ has rank $1$ (i.e.\ it is isomorphic as an ordered abelian group to a subgroup of $\rr$), we call $v$ a \emph{rank-$1$ valuation}.
More generally, we define the \emph{rank} of $v$ as the rank (as a torsion-free abelian group) of the value group $\Lambda = v(\kk^\times)$.
The subset
\[ \mathcal{O} \coloneqq \{ x \in \kk^\times \colon v(x) \geq 0\}\]
forms a subring, which is a \emph{valuation ring} of $\kk$, i.e.\ a subring of $\kk$ such that for all $x \in \kk^\times$ we have $x \in \mathcal{O}$ or $x^{-1} \in \mathcal{O}$.
A field together with a valuation is called a \emph{valued field}.

\begin{example}
    Examples of valued fields are the $p$-adic numbers.
    Whenever $\ff$ is any field, the field of rational functions with coefficients in $\ff$ is naturally a valued field, where the valuation is given by $v \colon \ff(X)^\times \to \zz$, $\frac{P}{Q} \mapsto \deg Q - \deg P$.
\end{example}
\subsection{Examples from the literature}
\label{subsection:ExamplesNonDiscBuildings}

We now give examples of generalized affine buildings from the literature, namely the norm building $\build_N$, the lattice building $\build_L$, symmetric building $\build_S$ and Bruhat--Tits building $\build_{\BT}$. All these constructions use a reductive algebraic group $\mathbf{G}$ and a field $\ff$ endowed with a valuation $v \colon \ff \to \Lambda \cup \{\infty\}$.
In \Cref{section:RelationsBuildings} we will relate these different models in the contexts.
Let us now define the buildings in question.

\begin{example}[Norm building $\build_N$, see also \Cref{ex:NormBuilding2}] \label{ex:NormBuilding1}
This model for an affine building has been studied in various settings in \cite{GoldmanIwahori63,Gerardin_ImmGroupesLinGen,BrTi84_norms,BrTi87,Parreau_ImmeublesAffines,RousseauBook23}; we follow \cite[\S 3]{Parreau_AffineBuildings}.
Let $\ff$ be any field with a rank-$1$ valuation $v \colon \ff^\times \to \Lambda \subseteq \rr$.
For $a\in \ff$, set $|a|\coloneqq \exp(-v(a)) \in \rr$ and let $V\coloneqq \ff^n$.
An \emph{ultrametric norm} is a function $\eta \colon V \to \rr_{\geq 0} $ that satisfies for all $a\in \ff$ and $v,w \in V$
\begin{itemize}
    \item  $\eta(v) = 0$ if and only if $v=0$,
    \item  $\eta(av) = |a|\eta(v)$, and
    \item  $\eta(v+w) \leq \max \{ \eta(v) , \eta(w) \}$.
\end{itemize}
An ultrametric norm $\eta$ is \emph{adapted} to a basis $\mathcal{E} = \{ e_1, \ldots , e_n\}$ of $V$ if 
\[
\eta\left(\sum_{i=1}^n a_i e_i \right) = \max\{ |a_1|\eta(e_1), \ldots , |a_n|\eta(e_n) \},
\]
and $\eta$ is \emph{adaptable} if there exists a basis to which it is adapted.
The \emph{norm building} $\build_N$ is the set of all $\rr^\times$-homothety classes of adaptable ultrametric norms. 

The model apartment of type $A_{n-1}$ can be identified with 
\[
\aa \cong \rr^n/\rr (1, \ldots , 1) \cong \Big\{ (x_1, \ldots , x_n)\in \rr^n \colon \sum_{i=1}^n x_i = 0 \Big\}
\]
and the spherical Weyl group is the symmetric group on $n$ letters acting on $\aa$ by permuting the entries.

To a basis $\mathcal{E}$ and a homothety class $[\eta]$ of an ultrametric norm $\eta$ adapted to $\mathcal{E}$, we associate a \emph{chart} $f_{[\eta],\mathcal{E}} \colon \aa \to \build_N$ by
\[
f_{[\eta],\mathcal{E}} (x_1, \ldots , x_n) \Big( \sum_{i=1}^n a_i e_i \Big) \coloneqq \max_i \left\{  e^{-x_i} \lvert a_i \rvert \eta(e_i)  \right\}.
\]
Let $\atlas$ denote the set of all these charts.
The pair $(\build_N, \atlas)$ is an affine $\rr$-building of type $(A_{n-1}, \rr, \rr^n/\rr(1, \ldots , 1))$, see \cite[Sections 3B-3F]{Parreau_AffineBuildings}. 
We note that $\build_N$ also admits a different atlas $\atlas'$, so that $(\build_N, \atlas')$ is an affine $\mathbb{R}$-building of type $(A_{n-1}, \mathbb{R}, \Lambda^n/\Lambda(1,\ldots,1))$, see \cite[Remark in Section 3B4]{Parreau_AffineBuildings}.
\end{example}

\begin{example}[Lattice building $\build_L$, see also \Cref{ex:LatticeBuilding2}]
\label{ex:LatticeBuilding1}

We now recall what we call the \emph{lattice building}, i.e.\ the space of homethety classes of lattices of $\ff^n$.
This was defined in \cite[Section 9.2]{Ronan_LectursBuildings} in the discrete case, and in \cite[Example 3.2]{BennettlambdabuildingsI} for general $\Lambda$.
We follow the latter exposition.

Let $\Lambda$ be an ordered abelian group (not necessarily of rank one) and $\ff$ a field with a $\Lambda$-valuation $v \colon \ff^\times \to \Lambda$.
Denote by $\mathcal{O}\coloneqq \{x \in \ff \colon v(x) \geq 0\}$ the valuation ring.
A \emph{lattice} (sometimes called $\mathcal{O}$-lattice) of $\ff^n$ is the set $\mathcal O e_1+ \ldots + \mathcal O e_n$, where $\{e_1,\ldots,e_n\}$ is a basis of $\ff^n$.
Two lattices $L_1$ and $L_2$ are \emph{homothetic} if there exists $x \in \ff$ such that $x L_1 =L_2$.
We write $[L]$ for the homothety class of a lattice $L$.
The \emph{lattice building} $\build_L$ is the set of all homothety classes of lattices in $\ff^n$, i.e.\ \[\build_L \coloneqq \{[L] \colon L \textnormal{ is a lattice}\}.\]
The model apartment of type $A_{n-1}$ can be identified with $\aa \cong \Lambda^{n-1}$.
A basis $\mathcal{E}=\{e_1,\ldots,e_n\}$ determines the lattice $\mathcal O ^n$ of $\ff^n$.
To a basis we define a \emph{chart} $f_{\mathcal{E}} \colon \Lambda^{n-1} \to \build_L$ as follows.
For $(\lambda_1,\ldots,\lambda_{n-1}) \in \Lambda^{n-1}$, we set
\begin{align*}
    f_{\mathcal{E}}(\lambda_1&,\ldots,\lambda_{n-1}) \coloneqq \\
    &[\mathcal{O}(x_{\lambda_1}e_1) +\mathcal{O}(x_{\lambda_2-\lambda_1}e_2) +\ldots +\mathcal{O}(x_{\lambda_{n-1}-\lambda_{n-2}}e_{n-1})+\mathcal{O}(x_{-\lambda_{n-1}}e_n)],
\end{align*}
where $x_{\lambda_i} \in \ff$ with $v(x_{\lambda_i})=\lambda_i$.
Note that the so obtained homothety class of lattices is independent of the choices of $x_{\lambda_i}$.

Let $\atlas$ denote the set of all these charts.
The pair $(\build_L,\atlas)$ is a generalized affine building of type $(A_{n-1},\Lambda,\Lambda^{n-1})$, see \cite[Example 3.2]{BennettlambdabuildingsI}.
\end{example}

\begin{example}[Bruhat--Tits buildings $\build_\BT$, see also \Cref{ex:BruhatTitsBuilding2}]
\label{ex:BruhatTitsBuilding1}

In their seminal works \cite{BruhatTits,BruhatTits84}, Bruhat--Tits construct affine buildings from reductive algebraic groups that are quasi-split with respect to a valued field $\ff$ with value group $\Lambda \subseteq \rr$.
More general cases in which the Bruhat--Tits building exists have been investigated over the years for example in \cite{MulherrStruyveVanMaldeghem,Struyve14}.
More recently, this construction was generalized to value groups $\Lambda$ that are not necessarily a subgroup of $\rr$ \cite{HebertIzquierdoLoisel_LambdaBuildingsAssocToQuasiSplitGroups}. 
For simplicity we restrict ourselves to the case when the algebraic group is semisimple and split \cite[18.6]{Borellinearalgebraixgroups}, but we allow for general value groups.
Our main reference when $\Lambda \subseteq \rr$ is \cite[Chapters 6 and 7]{BruhatTits}, especially examples 6.1.3 b) and 6.2.3 b), and for general $\Lambda$ \cite{HebertIzquierdoLoisel_LambdaBuildingsAssocToQuasiSplitGroups}, especially \S 7.44 therein. For a more gentle introduction, we recommend Izquierdo's notes \cite{Izquierdo2024HigherBT} of a mini-course he gave on the topic.

Let $\ff$ be a field with a valuation $v\colon \ff^\times \to \Lambda$. By Hahn's embedding theorem, see \Cref{thm:hahn}, there exists an ordered abelian group $\mathfrak{R}$ that contains $\Lambda$ as an ordered subgroup and that is an $\rr$-vector space.
In the classical case \cite{BruhatTits}, $\Lambda \subseteq \rr = \mathfrak{R}$.
Let $\mathbf{G}$ be a semisimple, connected, simply-connected algebraic $\ff$-group that is split over $\ff$ (such as $\operatorname{SL}_n$).
Let $\mathbf{S} $ be a maximal torus. 
Since $\mathbf{G}$ is $\ff$-split, $\mathbf{S}$ is $\ff$-split and since $\mathbf{G}$ is semisimple, we have $\mathbf{S} = \operatorname{Cent}_{\mathbf{G}}(\mathbf{S})$. 
For a root $\alpha \in \Sigma$ in the relative root system $\Sigma \coloneqq \,_{\ff}\Phi$, we consider the root group $\mathbf{U}_\alpha$.
Let $G \coloneqq \mathbf{G}(\ff)$ and $U_\alpha \coloneqq \mathbf{U}_\alpha(\ff)$ for $\alpha \in \Sigma$. 
From \cite[(6.1.3)b)]{BruhatTits} and \cite[Section 7.19]{HebertIzquierdoLoisel_LambdaBuildingsAssocToQuasiSplitGroups}, we get that there are subgroups $M_\alpha$ of $G$ such that $(\mathbf{S}(\ff),(U_\alpha, M_\alpha)_{\alpha \in \Sigma} )$ is a \emph{generating root group datum} that admits a \emph{valuation} $\varphi_\alpha \colon U_\alpha \to \mathfrak{R} \cup \{\infty\}$ for all $\alpha \in \Sigma$, see \cite[(6.2.3)b)]{BruhatTits}, \cite[(4.1.19)(ii)]{BruhatTits84}, and \cite[Proposition 7.24]{HebertIzquierdoLoisel_LambdaBuildingsAssocToQuasiSplitGroups}; for the classical groups see \cite[Chapter 10, page 208ff]{BruhatTits}.

We consider the Euclidean vector space $(V, \langle \cdot, \cdot \rangle )$ such that the dual space $(V^\star, \langle \cdot , \cdot \rangle )$ is spanned by $\Sigma$.
The group $N \coloneqq \operatorname{Nor}_{\mathbf{G}}(\mathbf{S})(\ff)$ acts on the root groups by conjugation $nU_{\alpha}n^{-1} = U_{n.\alpha}$ and this action descends to the action of the spherical Weyl group $\,_{\ff}W \coloneqq N/\mathbf{S}(\ff)$ on $\Sigma \subseteq V^\star$. The dual root system $\Sigma^\vee \subseteq V$ of $\Sigma \subset V^\star$ consists of the dual roots $\alpha^\vee \in \Sigma^{\vee}$ defined by
\[
 \langle \alpha^\vee , x \rangle = 2 \frac{  \alpha(x)}{\langle \alpha , \alpha \rangle} \quad \text{ for all } x \in V,
\]
where $\alpha \in \Sigma$. 
We note that the linear maps $\alpha \in V^\star$ extend to linear maps $\alpha \colon V \otimes_\rr \mathfrak{R} \to \mathfrak{R}$ on the $\rr$-vector space $V \otimes_{\rr} \mathfrak{R}$, and the action of the spherical Weyl group $\,_{\ff}W$ extends to $V \otimes_{\rr} \mathfrak{R}$.
The space $V\otimes_{\rr} \mathfrak{R} \cong \operatorname{Span}_{\mathbb{Q}}(\Sigma^\vee) \otimes_{\qq} \mathfrak{R}$ can be identified with the \emph{affine space of root group valuations}
\begin{align*}
  \mathfrak{A} \coloneqq \{ (\varphi_\alpha^x  \colon U_\alpha \to \mathfrak{R} \cup \{\infty\} )_\alpha \colon & 
\exists \, x \in V\otimes_{\rr}\mathfrak{R},
\forall \alpha \in \Sigma, \\
& \forall u \in U_\alpha  \colon \varphi_\alpha^x(u) = \varphi_\alpha(u) + \alpha(x) \}  ,
\end{align*}
and we abbreviate the root group valuation $(\varphi_\alpha^x)_{\alpha \in \Sigma}$ by $\varphi + x$. For $n \in N$, 
\[
(n.\varphi)_\alpha (u) = \varphi_{n^{-1}.\alpha}(n^{-1}un)
\]
defines a root group valuation and an action $ \nu$ of $N $ on $\mathfrak{A}$ by
\[
\nu(n)(\varphi+x) \coloneqq n.\varphi + n.x,
\]
where the $n.x$ comes from the action of the spherical Weyl group on $V\otimes_{\rr} \mathfrak{R}$.
The kernel of this action is denoted by $H$ and $N/H$ is called the affine Weyl group. For $x \in \mathfrak{A} \cong V \otimes_{\rr} \mathfrak{R}$, let 
\[
P_x \coloneqq \left\langle u \in U_{\alpha} \colon \exists\, \lambda \in \Lambda \textnormal{ such that }\varphi_\alpha(u) \geq \lambda \geq -\alpha(x)  \right\rangle \cdot H \subseteq  G.
\]
The \emph{Bruhat--Tits building} is now defined as the quotient 
\[
\build_\BT \coloneqq (G \times  \mathfrak{A}) / \!\!\sim
\] 
for the equivalence relation $(g,x) \sim (h,y)$ whenever
\[
 \exists\, n \in N \colon g^{-1}hn \in P_x \quad \text{ and } \quad \nu(n)(x) = y.
\]
The equivalence class of $(g,x)$ is denoted by $[g,x]$.

We now explain how $\build_\BT$ is an affine $\mathfrak{R}$-building of type $(\Sigma^\vee, \mathfrak{R}, \Lambda^n)$, see \cite[Theorem 3.30, Notation 4.37, Fact 7.45]{HebertIzquierdoLoisel_LambdaBuildingsAssocToQuasiSplitGroups}. The model apartment $\aa \coloneqq \operatorname{Span}_\qq(\Sigma^\vee) \otimes_\qq \mathfrak{R}$ can be identified with $\mathfrak{A}$ and is thus endowed with an action by $N/H$, which is the affine Weyl group $\affwg$. Note that the translational part $T$ of $\affwg$ consists only of translations in $\Lambda^n \subseteq \mathfrak{R}^n$. We have the \emph{standard chart} 
$$
f_0 \colon \aa \to \build_\BT, \quad x \mapsto [\Id,x].
$$
The natural action of $G$ on $G \times \mathfrak{A}$ descends to an action on $\build_{\BT}$ given by $g.[h,x] = [gh,x]$.
The atlas of $\build_\BT$ is defined by $\mathcal{A} \coloneqq \{g.f_0 \colon g \in G\}$, where $g.f_0(x) \coloneqq [g,x]$ for all $x \in \mathfrak{A}$.

We remark that the root group valuations $(\varphi_\alpha)_\alpha$ are \emph{compatible with} the field valuation $v$, see \cite[4.2.7(2)]{BruhatTits84} and \cite[Section 7.44, Fact 7.12]{HebertIzquierdoLoisel_LambdaBuildingsAssocToQuasiSplitGroups}, namely for all $t \in \mathbf{S}(\ff)$ and $\alpha \in \Sigma$ we have
\begin{align}
   \varphi_\alpha(t u t^{-1}) &= \varphi_\alpha(u) + v(\alpha(t)).
   \tag{$C_{\BT}$}
   \label{eq:BT_compatible} 
\end{align}
In particular for all $t\in \mathbf{S}(\ff)$ and $x\in \mathfrak{A} \cong V\otimes_\rr \mathfrak{R}$ we have that $[t,0] = [\Id, x] \in \build_\BT$ if and only if $(-v)(\chi_\alpha(t)) = \alpha(x)$ for all $\alpha\in \Sigma$.
\end{example}

\begin{example}[Symmetric buildings $\build_S$, see also \Cref{ex:SymmetricBuilding2}] \label{ex:SymmetricBuilding1}

For this example we need to assume that the field in question is real closed, and that the valuation is compatible with the order.
For the precise definitions and a short introduction to ordered fields and real algebraic geometry we refer to \Cref{appendix}.

Following ideas from \cite{KramerTent} (and \cite{Brumfiel_TreeNonArchimedeanHyperbolicPlane} for $\operatorname{SL}_2$), the first author defines in \cite{ appenzeller2024semialgebraic, appenzeller2025buildings, appenzeller2026semialgebraic} a building associated to the following data.
Let $\mathbf{G}$ be a semisimple connected self-adjoint linear algebraic $\qq$-group $\mathbf{G} < \operatorname{SL}_n$ for some $n\in \mathbb{N}$ and let $\ff$ be a real closed field.
To define a building, we need $\ff$ to be endowed with an order-compatible valuation $v\colon \ff^\times \to \Lambda$ (not necessarily of rank one), but for now let $\ff$ just be a real closed field, possibly $\ff=\rr$.
Denote by $G\coloneqq \mathbf{G}(\ff)$ the $\ff$-points of the algebraic group $\mathbf{G}$.

The group $\operatorname{SL}_n(\ff)$ acts transitively on the set 
\[
P_1(n,\ff)\coloneqq\{ M \in \ff^{n\times n} \colon M=M^T, \det(M)=1, M \gg 0 \}\] 
of positive definite symmetric matrices of determinant one by congruence, i.e.\ $g.M \coloneqq gMg^T$. Let $X_\ff $ be the orbit $G .\operatorname{Id} $ of $\operatorname{Id}\in P_1(n,\ff)$. 
If $\ff = \rr$, $X_\rr$ is a model of the symmetric space of non-compact type associated to $\mathbf{G}(\rr)$, which is totally geodesic submanifold of the symmetric space $P_1(n,\rr)$. When $\ff$ is not a subfield of $\rr$, we call $X_\ff$ a \emph{non-standard symmetric space}.
Given an order-compatible valuation $v \colon \ff^\times \to \Lambda$, it is then possible to define a $G$-invariant, $\Lambda$-valued pseudo-distance $d \colon X_\ff \times X_\ff \to \Lambda$ whose quotient $\build_S = X_\ff / \!\! \sim$ after identifying all points of distance $0$ is an affine $\Lambda$-building in the following sense \cite{appenzeller2025buildings}.

Let $\mathbf{K} \coloneqq \mathbf{G} \cap \operatorname{SO}_n$.
Let $\mathbf{S}<\mathbf{G}$ be a maximal $\rr$-split torus (equivalently maximal $\ff$-split for any real closed field, see \cite[Theorem 5.17]{appenzeller2024semialgebraic}) whose elements are self-adjoint ($g = g^T$ for all $g\in \mathbf{S}$).
Let $A_\ff$ be the semialgebraically connected component of the identity in $\mathbf{S}(\ff)$.
When $\ff =\rr$, the Lie algebra $\frakg$ of $G$ admits a Cartan decomposition $\frakg = \frakk \oplus \frakp$, where $\frakk$ is the Lie-algebra of $\mathbf{K}(\rr)$, given by the Cartan involution $\theta \colon X \mapsto -X^T$.
The group $A_\rr$ is then a connected real Lie group whose Lie algebra $\fraka \subseteq \frakp$ can be used to obtain the restricted root space decomposition
\[
\frakg = \frakg_0 \oplus \bigoplus_{\alpha \in \Sigma} \frakg_\alpha
\]
for a root system $(\Sigma, \fraka^\star)$ with spherical Weyl group $\sphwg \cong N_\rr / M_\rr$, where
\[
 N_\rr \coloneqq \operatorname{Nor}_{\mathbf{K}(\rr)}(A_\rr) \quad \text{ and } \quad M_\rr \coloneqq \operatorname{Cent}_{\mathbf{K}(\rr)}(A_\rr).
\]
Associated to every root $\alpha \in \Sigma$, there is an algebraic character $\chi_\alpha \colon A_\ff \to \ff_{>0}$ \cite[Lemma 6.1]{appenzeller2024semialgebraic}.

The Cartan decomposition $G = \mathbf{K}(\ff) \mathbf{A}(\ff) \mathbf{K}(\ff)$ \cite[Theorem 6.8]{appenzeller2024semialgebraic} can then be used to define a \emph{Cartan projection} 
$$
\delta_\ff \colon X_\ff \to  \{ a \in A_\ff \colon \chi_\delta(a) \geq 1 \text{ for all } \delta \in \Delta\}, \quad kak'.\Id \mapsto a,
$$
where $\Delta$ is a basis of $\Sigma$.
Then, $\lVert \cdot \rVert_{\ff} \colon A_\ff  \to \ff_{>0}$ defined by
\[
\lVert a \rVert_{\ff} \coloneqq \prod_{\alpha \in \Sigma} \max\left\{ \chi_\alpha(a), \chi_\alpha(a)^{-1}   \right\} 
\]
is a semialgebraic multiplicative $\ff_{\geq 1}$-valued norm on $A_\ff$.
The norm $\lVert \cdot \rVert_{\ff}$ and $\delta_\ff$ together with the order-compatible valuation $v$, can be used to define the distance of a point $x \in X_\ff$ to the base point $\Id$ by 
$$
d(\Id, x) \coloneqq -v ( \lVert \delta_\ff(x) \rVert_{\ff} ).
$$
The $G$-action extends this to a $\Lambda$-valued pseudo-metric $d$ on $X_\ff$, and a $\Lambda$-valued metric on $\build_S \coloneqq X_\ff / \!\! \sim $, where $\sim$ identifies points of distance $0$ \cite[Theorem 4.5]{appenzeller2025buildings}.

Let $o = [\operatorname{Id}] \in \build_S$ be a base point.
It is then possible to show that $A_\ff .o$ admits a $\qq$-vector space structure isomorphic to $\aa \coloneqq \operatorname{Span}_\qq(\Sigma^\vee) \otimes_{\qq} \Lambda$, where $\Sigma^\vee \subseteq \fraka$ is the coroot system.
If $f_0 \colon \aa \to \build_S$ denotes the identification $\aa \cong A_\ff.o$ followed by the inclusion $A_\ff.o \subseteq \build_S$, the atlas is given by $\atlas = \{ g.f_0 \colon \aa \to \build_S \colon g \in G\}$. 
Taking $T\coloneqq\aa \cong \Lambda^r$, $(\build_S,\atlas)$ is an affine $\Lambda$-building of type $(\Sigma^\vee, \Lambda, T)$ in the case when $\Sigma^\vee$ is reduced \cite[Theorem 5.1]{appenzeller2025buildings}.

The following characterizing compatibility condition will be useful 
\cite[Proposition 4.13]{appenzeller2025buildings}: for every $x\in \aa$ and $a\in A_\ff$, one has $f_0(x) = a.o$ if and only if
\begin{align}
    (-v)(\chi_\alpha(a)) = \alpha(x) \quad \text{for all $\alpha \in \Sigma$.}
    \tag{$C_S$}
   \label{eq:BH_compatible}
\end{align}
\end{example}

\section{Morphisms of generalized affine buildings}
\label{section:Morphisms}
The goal of this section is to define
morphisms of generalized affine buildings.
Before we can do so we define morphisms of apartments.

\subsection{Morphisms of apartments}
Let $\Phi$ and $\Phi'$ be two crystallographic root systems, $\Lambda$ and $\Lambda'$ ordered abelian groups, and $T$ and $T'$ translation groups such that $\aa = \aa(\Phi, \Lambda, T)$ and $\aa' = \aa(\Phi', \Lambda', T')$ are model apartments in the sense of \Cref{subsection:AffineLambdaBuildings}. 

\begin{defn}\label{def:morphism_apartment}
A \emph{morphism of apartments} is a triple $(L,\gamma, \sigma)$ of maps, where $L\colon \operatorname{Span}_{\mathbb{Q}}(\Phi) \rightarrow \operatorname{Span}_{\mathbb{Q}}(\Phi')$ is a $\mathbb{Q}$-linear map, $\gamma \colon \Lambda \rightarrow \Lambda '$ is a morphism of ordered abelian groups and $\sigma \colon \affwg \to \affwg'$ is a group homomorphism, such that for all $w \in \affwg$ the following diagram commutes
\[
\begin{tikzcd}
    \mathbb{A} \arrow[r, "L \otimes \gamma"] \arrow[d, "w"'] & \mathbb{A}' \arrow[d, "\sigma(w)"] \\
    \mathbb{A} \arrow[r, "L \otimes \gamma"']                & \mathbb{A}' \ .
\end{tikzcd}
\]
We will denote morphisms of apartments by $\tau = (L,\gamma, \sigma)$, and when it is convenient we will write by slight abuse of notation $\tau \colon \aa \to \aa'$ and $\tau = L \otimes \gamma$.

We call $\tau$ \emph{injective (resp.\ surjective)} if $L$ and $\gamma$ are injective (resp.\ surjective).
It is a linear algebra exercise to see that $\tau$ is injective (resp.\ surjective) if and only if $L \otimes \gamma$ is injective (resp.\ surjective). If $\tau$ is injective, so is $\sigma$, but if $\tau$ is surjective, $\sigma$ may not be surjective, see \Cref{fig:A2_in_G2}.

\begin{figure}[ht]
  \centering
  \includegraphics[width=0.4\linewidth]{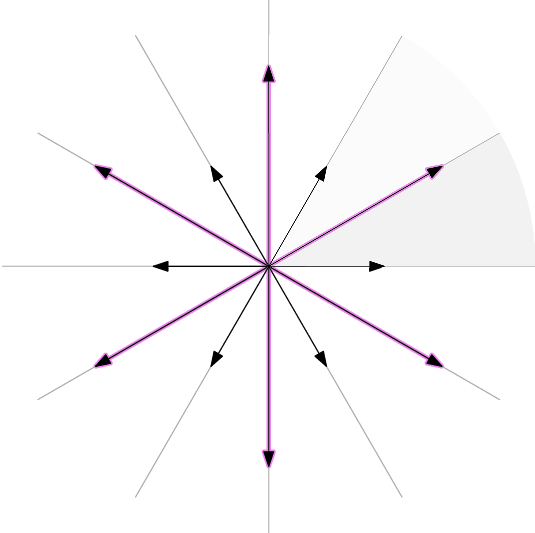} 
  \caption{The root system $\Phi'$ of type $G_2$ (in black) contains a subsystem $\Phi$ of type $A_2$ (in pink) formed by the long roots of $\Phi'$. The inclusion $\sigma \colon \sphwg \to \sphwg'$ gives rise to a surjective morphism $(\Id,\Id,\sigma)$ of apartments, but $\sigma$ is not surjective.}
  \label{fig:A2_in_G2}
\end{figure}

\begin{example}\label{remark:groups_to_modules}
    Suppose $\Phi = \Phi'$ and $\gamma \colon \Lambda \to \Lambda'$ is a morphism of ordered abelian groups such that $\gamma(T)\subseteq T'$.
    Then $\Id \otimes \gamma \,  \colon \mathbb{A} \to \mathbb{A}'$ defines a morphism of apartments.
    This recovers the definition implicitly present in \cite{SchwerStruyve_LambdaBuildingsBaseChangeFunctors}.
\end{example}

\begin{remark}\label{rem:morph_apart_category}
    The identity morphism is given by $L=\Id_{\operatorname{Span}_{\mathbb{Q}}(\Phi)}$, $\gamma = \Id_\Lambda$ with $\sigma= \Id_{\affwg}$ and composition of morphisms is given by composition of $L$, $\gamma$ and $\sigma$.
It is not hard to check that with these notions, model apartments form a category.
A morphism $\tau \colon \aa \to \aa'$ then is an \emph{isomorphism of apartments} if there is a morphism $\tau^{-1} \colon \aa' \to \aa$ with $\tau^{-1} \circ \tau = \Id_{\aa}$ and $\tau \circ \tau^{-1} = \Id_{\aa'}$.
\end{remark}

By the following lemma it suffices to verify the diagram in the definition for $w \in \sphwg$.
\end{defn}
\begin{lemma}\label{lem:morphism_apartment}
    Let $L\colon \operatorname{Span}_{\mathbb{Q}}(\Phi) \rightarrow \operatorname{Span}_{\mathbb{Q}}(\Phi')$ be a $\mathbb{Q}$-linear map, $\gamma \colon \Lambda \rightarrow \Lambda '$ a morphism of ordered abelian groups and set $\tau = L \otimes \gamma$. If $\tau(T) \subseteq T'$ and $\sigma_s \colon \sphwg \to \sphwg'$ is a group homomorphism (of the spherical Weyl group only) such that the diagram
    \[
\begin{tikzcd}
    \mathbb{A} \arrow[r, "\tau"] \arrow[d, "w"'] & \mathbb{A}' \arrow[d, "\sigma_s(w)"] \\
    \mathbb{A} \arrow[r, "\tau"']                & \mathbb{A}'
\end{tikzcd}
\]
    commutes for all $w\in \sphwg$, then $\sigma \coloneqq  \tau |_T \times \sigma_s  \colon \affwg \to \affwg'$ is a group homomorphism such that $(L,\gamma,\sigma)$ is a morphism of apartments $\aa \to \aa'$.
\end{lemma}
\begin{proof}
    We denote by $t^{x}\colon \aa \to \aa$ the translation by $x\in \aa$. Let $t^x, t^{x_1},t^{x_2}\in T$, $w, w_1,w_2 \in \sphwg$. We define the map $\sigma \colon \affwg \to \affwg'$ by $\sigma(t^{x}w) \coloneqq t^{\tau (x)}\sigma_s(w)$. By the semidirect product structure of $\affwg = T \rtimes \sphwg$, $w_1t^{x_2} = t^{w_1(x_2)}w_1$, so
    \begin{align*}
        \sigma(t^{x_1}w_1\cdot t^{x_2}w_2) 
        &= \sigma\left( t^{x_1} t^{w_1(x_2)} w_1 w_2 \right)
        = t^{\tau (x_1 + w_1(x_2))} \sigma_s(w_1 w_2) \\
        &= t^{\tau(x_1)} t^{\tau (w_1(x_2))} \sigma_s(w_1)\sigma_s(w_2) \\
        &= t^{\tau (x_1)} t^{\sigma_s(w_1)(\tau(x_2))}\sigma_s(w_1)\sigma_s(w_2)\\
        &= t^{\tau (x_1)} \sigma_s(w_1) t^{\tau (x_2)} \sigma_s(w_2)
        = \sigma(t^{x_1}w_1)\cdot \sigma(t^{x_2}w_2)
    \end{align*}
    shows that $\sigma$ is a group homomorphism. If $y \in \aa$, we have
    \begin{align*}
        \tau (t^{x}w(y)) &= \tau ( x + w(y))
        = \tau (x) + \tau (w(y)) 
        = t^{\tau (x)}(\sigma_s(w)(y)) 
        = \sigma(t^xw)(y),
    \end{align*}
    so the diagram commutes.
\end{proof}

\subsection{Morphisms of buildings}
 Since buildings are made out of apartments, morphisms of apartments lead to the definition of morphisms of generalized affine buildings as follows.

\begin{defn}\label{def:morphism_building}
Let $\build= (\build,\atlas)$ be an affine $\Lambda$-building of type $\mathbb{A}= \mathbb{A}(\Phi,\Lambda,T)$ and $\build'=(\build',\atlas')$ an affine $\Lambda'$-building of type $\mathbb{A}'=\mathbb{A}(\Phi',\Lambda',T')$. A \emph{morphism of generalized affine buildings} is a triple of maps
\[\psi \colon \build  \to \build',  \quad
    \varphi \colon \atlas  \to \atlas', \quad
    \tau \colon \mathbb{A}  \to \mathbb{A}',
\]
where $\tau$ is a morphism of apartments, such that the following diagram commutes for all $f \in \atlas$
\[\begin{tikzcd}
	{\mathbb{A}} & {\mathbb{A}'} \\
	\build  & \build' \ . 
	\arrow["\tau", from=1-1, to=1-2]
	\arrow["f"', from=1-1, to=2-1]
	\arrow["\varphi(f)", from=1-2, to=2-2]
	\arrow["{\psi}"', from=2-1, to=2-2]
\end{tikzcd}\]
We write morphisms of generalized affine buildings as $m= (\psi,\varphi,\tau)$, and when it is convenient, we will write by slight abuse of notation $m \colon (\build,\atlas) \to (\build',\atlas')$ or $m \colon \build \to \build'$.
We say that a morphism $(\psi,\varphi, \tau)$ is \emph{injective (resp.\ surjective)} if $\psi$, $\varphi$ and $\tau$ are injective (resp.\ surjective).
\end{defn}

\begin{example}\label{ex:apt_as_build}
If $(\build,\atlas)$ is a generalized affine building of type $\aa=\aa(\Phi,\Lambda,T)$, then any $f \in \atlas$ defines an injective morphism of generalized affine buildings $\psi = f \colon \aa \to \build$, where we view $\aa$ itself as a generalized affine building with atlas $\affwg$. Indeed, we define $\varphi \colon \affwg \to \atlas$ by precomposition with $f$, which by axiom \ref{axiom:A1} is contained in $\atlas$, and $\tau \colon \aa \to \aa$ to be the identity.
Then the above diagram commutes.
\end{example}
\begin{example} For a fixed ordered abelian group $\Lambda$, the projection of a product of affine $\Lambda$-buildings (with the product atlas and of product type) to one of the factors is a surjective morphism of affine $\Lambda$-buildings.
\end{example}
\begin{example}\label{ex:completion_translations}
If $(\build,\atlas)$ is a generalized affine building of type $(\Phi, \Lambda, T)$, let $T' \coloneqq \aa = \operatorname{Span}_{\mathbb{Q}}(\Phi) \otimes \Lambda$ be the full translation group, $\affwg' = \aa \rtimes \sphwg$ and $\atlas' \coloneqq \{ f\circ w \colon f \in \atlas, w \in \affwg' \}$. Then $(\build,\atlas')$ is a generalized affine building of type $(\Phi,\Lambda, \aa)$ and the inclusions $\sigma \colon \affwg \to \affwg'$, $\varphi \colon \atlas \to \atlas'$ together with the identities on $\aa$ and $\build$ define an injective morphism $(\build,\atlas) \to (\build, \atlas')$.
\end{example}
\begin{example}
\label{example: R and Q building}
We may view $\rr$ as an affine $\rr$-building and $\qq$ as an affine $\qq$-building. The inclusion $\qq \subseteq \rr$ gives rise to an injective morphism $\qq \to \rr$ of generalized affine buildings (for appropriately chosen atlases). However, since the only order-preserving group homomorphism from $\rr$ to $\qq$ is the trivial one, the only morphisms of generalized affine buildings from $\rr$ to $\qq$ are the trivial morphisms (where $\psi$ is constant).
\end{example}

\begin{remark}\label{rem:morph_build_category}
    The identity morphism is given by $\psi = \Id_{\atlas}$, $\varphi = \Id_B$ and  $\tau = \Id_\aa$ and composition of morphisms is given by the composition of all three $\psi$, $\varphi$ and $\tau$. It is not hard to check that with these notions, generalized affine buildings form a category. A morphism $m = (\psi,\varphi,\tau) \colon \build \to \build'$ is then called an \emph{isomorphism} if there exists a morphism $m^{-1} \colon \build' \to \build$ with $m^{-1} \circ m = \Id_{\build}$ and $m \circ m^{-1} = \Id_{\build'}$.
\end{remark}

\begin{remark}\label{rem:distance_preserving}
    Morphisms of generalized affine buildings are not necessarily distance-preserving.
    Contrary to what it seems like, this is a good property, as a single building can carry many different distances: in \cite{BruhatTits}, the distance comes from a choice of scalar product on the apartment and in \cite{appenzeller2025buildings}, the distance comes from a choice of a norm on the apartment. When $\Lambda \subsetneq \rr$, these norms may not come from a scalar product, so there may not be a compatible way to choose distances such that the injective morphism $B_S \to B_\BT$ in \Cref{corol:MorphismKramerTentToBruhatTits} preserves distances.
\end{remark}

\begin{remark}\label{remark: why L linear not affine}
    Since morphisms of apartments always fix $0$, affine maps $\aa \to \aa$ may not be morphisms of apartments. However, if we view the apartment $ \aa$ as an affine building itself as in \Cref{ex:apt_as_build}, then any affine map $w \in \affwg $ is a morphism of affine buildings. Indeed, setting $\psi \coloneqq w \colon \build \to \build$, $\varphi\colon \atlas \to \atlas$ defined by $\varphi(f) \coloneqq w \circ f $ and $\tau \coloneqq \Id_\aa \colon \aa \to \aa$ defines a morphism of affine buildings $(\psi, \varphi, \tau) $ realizing the affine transformation $w$ on $B=\aa$. We note that by \ref{axiom:A1}, $\affwg \subseteq \atlas$, and by \ref{axiom:A2}, $\atlas \subseteq \affwg$, so $w \circ f \in \affwg = \atlas$ and $\varphi$ is well-defined.
\end{remark}

\subsection{Relation to existing notions in the literature}
\label{section:ExistingNotions}

We finish this section by relating our definition to some existing notions of morphisms and isomorphisms between apartments and (generalized affine) buildings already present in the literature. 

\subsubsection{Notions of isomorphisms}
Already in \cite{Titssphericalbuildings, Titsaffinebuildings}, Tits defines a notion of isomorphism of affine $\rr$-buildings, and partially classifies them. 
In the discrete case, a detailed treatment can be found in \cite{Weiss_StructureAffineBuildings}.
Similarly, still in the discrete case, Gérardin states that the norm building ``is'' the Bruhat--Tits building
\cite[Theorem in \S 2.3.6]{Gerardin_ImmGroupesLinGen}.
He defines in \S 1.3.7 an \emph{isomorphism of apartments} as a linear map that preserves distances and that exchanges the walls.
This is stronger than our definition, as we do not ask a morphism of apartments to preserve distances, see \Cref{rem:distance_preserving}.

\subsubsection{Discrete valuations}
When the valuation takes values in a discrete subset of $\rr$, there are several notions of morphisms, see e.g.\ \cite{Landvogt, kaletha2023bruhat}.
We briefly explain Landvogt's notion \cite{Landvogt} and his result on Bruhat--Tits buildings, as it can be compared to \Cref{intro:thm:InjMorphBuildings}.
Landvogt showed a functoriality property, namely that a homomorphism of algebraic $k$-groups, where $k$ is a quasi-local field, induces an equivariant continuous map between the associated Bruhat--Tits buildings, that is \emph{toral}---a notion that ensures that apartments are mapped to apartments.
Furthermore, after suitable normalization of the metric, this map is an isometry.
A toral map should be thought of as the analog of our notion of morphism of apartments, and it would be interesting to compare these two notions in the discrete setting.

A comparison to the notion in \cite{kaletha2023bruhat} would also be desirable.
However in general it is not so clear how $\Lambda$-buildings for $\Lambda$ a discrete subgroup of $\rr$ relate to simplicial affine buildings. This is why we have restricted our attention to divisible $\Lambda$.

\subsubsection{Euclidean buildings}
In the non-discrete case, but in the setting of Euclidean buildings, Rousseau defines notions of (weak) morphisms of apartments, buildings and subbuildings
\cite{RousseauBook23}.
A \emph{weak morphism of apartments} (endowed with a Euclidean distance) is an affine map $\varphi \colon \aa\to \aa'$ that satisfies three axioms related to the affine Weyl group actions, the walls of the apartments and their Euclidean distances \cite[Definition 1.1.4.1]{RousseauBook23}, compare with \Cref{rem:distance_preserving}.
However the composition of two weak morphisms is in general not a weak morphism.
Rousseau then defines a \emph{morphism of apartments} to be a metric weak morphism, meaning that $\varphi$ preserves distances up to the kernel of the linear part of $\varphi$.
Note that our definition does not take the metrics on $\aa$ and $\aa'$ into account, and allows thus for more flexibility.
The notion of \emph{(weak) morphisms of buildings} \cite[Definition 2.1.13.1]{RousseauBook23} is defined as a map between two Euclidean buildings that maps apartments to apartments, and that when restricted to an apartment is a (weak) morphism of apartments, which resembles our construction.
When the (weak) morphism of buildings is injective, this gives rise to \emph{(weak) subbuildings}.
It is however unclear, that, even when we place ourselves in the setting of Euclidean buildings, how his notion relates to ours.

Similarly, still in the setting of Euclidean buildings, Kleiner--Leeb define a \emph{subbuilding}
as a metric subspace of a Euclidean building which admits a Euclidean building structure \cite[Subsections 4.7]{KleinerLeebRigidityofquasiisometries}.
Note that not every subbuilding in that sense can be realized by an injective morphism: consider geodesics in higher rank buildings and compare with \Cref{fig:A1_in_A2}. We leave open whether the image of every injective morphism of Euclidean buildings is a subbuilding.

\begin{question}
    If $(\psi,\varphi,\tau) \colon \build \to \build'$ is an injective morphism (in our sense) of Euclidean buildings ($\Lambda = \rr$) with CAT(0)-metrics à la Kleiner--Leeb, can the metric on $\build$ be rescaled (on factors) so that $\psi$ is an isometric embedding?  
\end{question}

\subsubsection{Morphisms of root systems}
The main difficulty in defining morphisms of apartments, is in the case when the root systems do not agree. 
Several notions of morphisms of root systems have been developed. In \cite{LoosNeher2004} and similarly in \cite{Dyer_root_systems} morphisms are linear maps $L \colon \operatorname{Span}_{\mathbb{R}}(\Phi) \to \operatorname{Span}_{\mathbb{R}}(\Phi')$ such that $L(\Phi) \subseteq \Phi'$, possibly with some extra conditions. 
In our definition of morphisms of apartments, we also have a linear map $L$, but roots may not be sent to roots. In applications this happens for instance in the context of functoriality under subgroups $\mathbf{G} \subseteq \mathbf{G}'$ (see Section \ref{sec:functoriality_subgroups}, and Figure \ref{fig:intro:Sp4inSL4} for $\operatorname{Sp}_4 < \operatorname{SL}_4$), where the linear map $L$ does not send roots to roots. 
A morphism of root systems in the category $\mathbf{RCE}$ defined in \cite{LoosNeher2004} satisfies the compatibility condition on the Weyl groups \cite[Theorem 5.7]{LoosNeher2004}, and thus defines a morphism of apartments (when $\Lambda = \mathbb{R}$) in our setting. Thus, \Cref{def:morphism_apartment} generalizes the one in \cite{LoosNeher2004}. The notion in \cite{Dyer_root_systems} is neither stronger nor weaker than ours.

\subsubsection{Generalized affine buildings and functoriality}
\label{subsection: functoriality generalized affine buildings}
The most general and at the same time closest to our notions are probably \cite{Schwer_GeneralizedAffineBuildings} and \cite{SchwerStruyve_LambdaBuildingsBaseChangeFunctors}, where the latter generalizes the former.
Let us discuss these two now in more detail.
In \cite[Definition 5.5]{Schwer_GeneralizedAffineBuildings}, Schwer defines a notion of \emph{isomorphism} of generalized affine buildings, which agrees with our definition (see \Cref{rem:morph_build_category}).

Schwer--Struyve consider in \cite[Sections 3 and 5]{SchwerStruyve_LambdaBuildingsBaseChangeFunctors} a generalized affine building $(\build,\atlas)$ of type $\mathbb{A} \coloneqq \mathbb{A}(\Phi,\Lambda,T)$ and an order-preserving group homomorphism $\gamma \colon \Lambda \rightarrow \Lambda'$.
From this data, they construct a model apartment $\mathbb{A}'\coloneqq \mathbb{A}(\Phi,\Lambda',T')$, where $T'$ is the component-wise image of $T$ under $\gamma$.
In Section 3, the authors construct a map $\mathbb{A} \rightarrow \mathbb{A}'$, which in our notation corresponds to two maps $\Id \otimes \gamma \colon \mathbb{A} \rightarrow \mathbb{A}'$ and $\Id \colon \affwg \rightarrow \affwg$.
In fact, they show that these maps define a morphism of apartments in the sense of \Cref{def:morphism_apartment}.
Note however that the root system $\Phi$ is the same.
Let us denote by $\tau$ this morphism of apartments.

In a second step, Schwer--Struyve construct a space $\build'$ as the quotient of $\build$ by an equivalence relation, a map $\psi \colon \build \rightarrow \build'$ and for each $f\in \atlas$ a map $f' \colon \mathbb{A}' \rightarrow \build'$.
If we denote by $\atlas'$ the set of these maps $f'$ constructed, then we get a map $\varphi \colon \atlas \rightarrow \atlas'$.
They then show in \cite[Sections 3 and 5]{SchwerStruyve_LambdaBuildingsBaseChangeFunctors} that $(\build',\atlas')$ is a building and that $m=(\psi,\varphi,\tau)$ is a morphism (in the sense of \Cref{def:morphism_building}) from $(\build,\atlas)$ to $(\build',\atlas')$. When $\gamma$ is surjective or injective, then so is $m$. 
Furthermore, in \cite[Theorem 1.1]{SchwerStruyve_LambdaBuildingsBaseChangeFunctors} the authors prove that the respective spherical buildings of $\build$ and $\build'$ at infinity are isomorphic, and this isomorphism is induced by the morphism of buildings $m$.
    
In the same spirit, in \cite[\S 7]{HebertIzquierdoLoisel_LambdaBuildingsAssocToQuasiSplitGroups}, the authors associate to a quasi-split reductive algebraic group $\mathbf{G}$ and a Henselian field $\ff$, equipped with a valuation $v \colon \ff^\times \to \Lambda$, a generalized affine building, denoted by $I(\ff, v, \mathbf{G})$.
If the group is split, this is described in \Cref{ex:BruhatTitsBuilding1}.
They then construct, given a surjective morphism of ordered abelian groups $f \colon \Lambda \to \Lambda'$, a surjective map $I(\ff, v, \mathbf{G}) \to I(\ff, f \circ v , \mathbf{G})$, which is compatible with the action of $\mathbf{G}(\ff)$.
It would be interesting to investigate whether $f$ is a surjective morphism of generalized affine buildings in the sense of \Cref{def:morphism_building}, and to compare it with \Cref{intro:corol_FieldExt} in the context of symmetric buildings.

\section{%
\texorpdfstring{%
$G$-buildings and their morphisms}%
{G-buildings and their morphisms}}
\label{section:GBuildings}
The goal of this section is to construct morphisms of generalized affine buildings that are equipped with group actions.
Under sufficient transitivity assumptions this allows to define a morphism of generalized affine buildings by specifying a morphism between the model apartments, and then moving this map around under the action of the group.
We recall the following definition.

\begin{defn}\label{def:G_building}
Let $G$ be a group and $(\build,\atlas)$ an affine $\Lambda$-building.
We call $(\build,\atlas)$ a \emph{$G$-building} if $G$ acts on both $\build$ and the atlas $\atlas$ such that the actions are compatible, i.e.\ for all $ g \in G$, $f \in \atlas$, and $x\in \mathbb{A}$ we have
\[
(g.f)(x) = g.(f(x)).
\]

If $\rho \colon G \to G'$ is a group homomorphism, $B$ is a $G$-building and $B'$ is a $G'$-building, a morphism $B \to B'$ is called $\rho$-\emph{equivariant} if both $\psi$ and $\varphi$ are $\rho$-equivariant, i.e.\ for all $g\in G$, $x\in \build$ and $f \in \atlas$ we have
\[
\psi(g.x) = \rho(g).\psi(x), \textnormal{ and } \varphi(g.f)= \rho(g).\varphi(f).
\]
\end{defn}

\begin{remark}
    In the language of category theory, an action of a group $G$ on an affine building $(\build, \atlas)$ is a group homomorphism $G \to \operatorname{Aut}(\build, \atlas) \coloneqq \{m \colon (\build,\atlas) \to (\build, \atlas) \colon m \text{ is an isomorphism}\}$, so for each $g\in G$ the action associates an isomorphism $m_g=(\psi_g,\varphi_g,\tau_g)$.
    The actions that appear in the definition of $G$-buildings correspond exactly to those actions $G\to \operatorname{Aut}(\build, \atlas)$ for which $\tau_g = \Id_{\aa}$ for all $g\in G$.
\end{remark}

\subsection{%
\texorpdfstring{%
Transitivity properties of $G$-buildings}%
{Transitivity properties of G-buildings}}
Here are some direct consequences of the definition of $G$-buildings.

\begin{prop}\label{prop:TransitiveActionOnBuilding}
    Let $(\build,\atlas)$ be a $G$-building of type $\mathbb{A}= \mathbb{A}(\Phi, \Lambda, T)$. 
    If $G$ acts transitively on $\build$, then $T =\aa $.
\end{prop}
\begin{proof}
    Let $f\in \atlas$ and $x\in \mathbb{A}$. By the transitivity of the action of $G$ on $\build$, there is $g\in G$ such that $f(0) = g.f(x)$. 
    Recall axiom \ref{axiom:A2}, that says that for all $f$, $f'\in \atlas$, the set $\Omega \coloneqq f^{-1}(f(\mathbb{A}) \cap f'(\mathbb{A}))$ is $\affwg$-convex and there exists $w \in \affwg$ such that $f|_\Omega = f \circ w |_\Omega$.
    We apply it to $f$ and $f'=g.f$.
    We note that $0 \in \Omega$.
    The second part of axiom \ref{axiom:A2} implies that there exists $w\in \affwg$ such that $f(0)=g.f \circ w (0)$.
    Applying $g^{-1}$ we get $f (w(0))= g^{-1} .f(0) = f(x)$.
    Since $f$ is injective we have $w(0)=x$.
    The element $w \in \affwg = T \rtimes \sphwg$ can be written as $w = (t,w_s)$ with $t \in T$ and $w_s \in \sphwg$.
    Since $w_s(0)=0$, we have $w(0) = t(0) = x$, hence $t=t^x$ is the translation by $x$.
    This shows that for all $x \in \mathbb{A}$, $t^x \in T$, and hence $T = \aa$.    
\end{proof}

Conversely we have the following.

\begin{prop}\label{lem:trans}
    Let $(\build,\atlas)$ be a $G$-building of type $\mathbb{A}= \mathbb{A}(\Phi, \Lambda, T)$. If $T = \aa$ is the full translation group and $G$ acts transitively on $\atlas$, then $G$ also acts transitively on $\build$.
\end{prop}
\begin{proof}
    Let $p,q \in \build$.
    By axiom \ref{axiom:A3}, there exist $f \in \atlas$ and $x,y \in \mathbb{A}$ such that $p=f(x)$ and $q=f(y)$.
    Since we assumed that $T=\aa$ is the full translation group, we have $w=y-x \in \affwg = T \rtimes \sphwg$.
    By axiom \ref{axiom:A1}, we have that $f \circ w \in \atlas$.
    Since $G$ acts transitively on $\atlas$, there is a $g\in G$, such that $g.f = f \circ w$.
    In particular $g.p = g.f(x) = f\circ w (x) = f(y) = q$, and hence $G$ acts transitively on $\build$.    
\end{proof}
\subsection{%
\texorpdfstring{%
Examples of $G$-buildings}%
{Examples of G-buildings}}
\label{subsection:ExamplesGBuildings}

All the constructions of affine buildings discussed in \Cref{subsection:ExamplesNonDiscBuildings} fall in fact in the framework of $G$-buildings for some appropriate group $G$.
Let us now explain this and discuss certain transitivity properties in more detail.

\begin{example}[Norm building $\build_N$, \Cref{ex:NormBuilding1} revisited]
\label{ex:NormBuilding2}
    Recall that the norm building $\build_N$ associated to $V=\ff^n$, where $\ff$ is a field with a rank-$1$ valuation $v \colon \ff^\times \to \Lambda \subseteq \rr$, is an $\rr$-building of type $(A_{n-1},\rr,\rr^n/\rr(1,\ldots,1))$.
    We claim that $\build_N$ is a $G$-building for $G=\operatorname{GL}_n(\ff)$.
    
    Indeed, for $g \in \operatorname{GL}_n(\ff)$ the action on $\build_N$ is given by $g. \eta\coloneqq \eta \circ g^{-1}$ for $\eta$ (a homothety class of) an adaptable ultrametric norm on $V$, and, if $\eta$ is adapted to a basis $\mathcal{E}$, then $g.\eta$ is adapted to $g\mathcal{E}$.
    The action on $\atlas$ is given by $g.f_{[\eta],\mathcal{E}}\coloneqq f_{[g.\eta],g\mathcal{E}}$, and this action is compatible with the action on $\build_N$; see also \cite[Sections 3A and 3B2]{Parreau_AffineBuildings}.

    We claim that, if $\Lambda=\rr$, then  $\glnf$ acts transitively on $\build_N$ and $\atlas$.
    Otherwise, $\glnf$ does not act transitively neither on $\build_N$ nor on $\atlas$.
    To prove this, we first remark that $\glnf$ acts transitively on the bases of $\ff^n$, thus to show that $\glnf$ acts transitively on $\build_N$ it suffices to show that any two norms adapted to the standard basis are in the same $\glnf$-orbit.
    
    If $\Lambda=\rr$, then for two ultrametric norms $\eta$ and $\eta'$ adapted to the standard basis $\mathcal{E}_0= \{e_1, \ldots , e_n\}$ there is some $g \in \operatorname{GL}_n(\ff)$ such that $g.\eta' = \eta$.
    Indeed, take $g=a=\operatorname{Diag}(a_1, \ldots , a_n)\in \operatorname{GL}_n(\ff)$ such that $\eta(e_i)=|a_i|\eta'(e_i)$ (recall that $|\cdot| =\exp(-v(\cdot))$ and use the surjectivity of $v$) for all $i=1,\ldots,n$.
    Then since both $\eta$ and $\eta'$ are adapted to $\mathcal{E}_0$, we have
    \[
    a.\eta\left(\sum_{i=1}^n v_i e_i \right) =
    \eta\left(\sum_{i=1}^n a_i^{-1}v_i e_i \right)=
    \max_i \left\{ |v_i||a_i|^{-1}\eta(e_i)  \right\} = \eta'\left(\sum_{i=1}^n v_i e_i \right)
    \]
    for all $v_i\in \ff$.
    Thus $\glnf$ acts transitively on $\build_N$.

    Now if $\Lambda\neq \rr$, we also see that the action on $\build_N$ is not surjective: For example, let $\eta$ be the ultrametric norm adapted to $\mathcal{E}_0$ satisfying $\eta(e_i)=1$ for all $i=1,\ldots,n$.
    Then for $s\in \rr \setminus \Lambda$, there is no $g \in \glnf$ such that $\eta$ is in the same orbit as $\eta'$, where $\eta'$ is the ultrametric norm adapted to $\mathcal{E}_0$ satisfying $\eta'(e_1)=s$ and $\eta'(e_2)=\ldots=\eta'(e_n)=1$.

    From the above discussion and the definition of the atlas $\atlas$, we also directly see that $\glnf$ acts transitively on $\atlas$ if and only if $\Lambda=\rr$.
    Thus $\operatorname{GL}_n(\ff)$ acts transitively on the set of apartments, but not on $\build_N$ or $\atlas$, unless $\Lambda=\rr$.
\end{example}

\begin{example}[Lattice building $\build_L$, \Cref{ex:LatticeBuilding1} revisited]\label{ex:LatticeBuilding2}
    The lattice building $\build_L$ associated to $\ff^n$, where $\ff$ is a field with a valuation $v \colon \ff^\times \to \Lambda$ (not necessarily of rank one), is an affine $\Lambda$-building of type $(A_{n-1},\Lambda,\Lambda^{n-1})$.
    It is also a $G$-building for both $G=\slnf$ and $G=\glnf$.
    Indeed, if $G$ is either $\slnf$ or $\glnf$, $\build_L$ consists of homothety classes of lattices $L$ of the form $\mathcal O e_1 + \mathcal O e_2 + \ldots + \mathcal O e_n$, where $\{e_1, \ldots , e_n\}$ is a basis of $\ff^n$.
    Thus $G$ acts on a lattice by acting on the basis, i.e.\ $g.L = \mathcal O ge_1 + \ldots + \mathcal O ge_n$.
    We also define an action of $G$ on a chart $f_\mathcal{E}$ for $\mathcal{E}$ a basis of $\ff^n$ by setting $g.f_\mathcal{E} = f_{g\mathcal{E}}$.
    Then these two actions commute.
    Recall that the atlas $\atlas$ consists of all charts $f_\mathcal{E}$ for $\mathcal{E}$ a basis of $\ff^n$. The group $G$ acts transitively on the set of homothety classes of unordered bases of $\ff^n$.
    Thus $G$ acts transitively on $\atlas$.
    Furthermore, $G$ also acts transitively on the set of lattices up to homothety, and thus $G$ acts transitively on the building $\build_L$.
\end{example}

\begin{example}[Bruhat--Tits building $\build_\BT$, \Cref{ex:BruhatTitsBuilding1} revisited]\label{ex:BruhatTitsBuilding2}
    Let $\build_{\BT}$ be the Bruhat--Tits building associated to $G \coloneqq \mathbf{G}(\ff)$, 
    where $\mathbf{G}$, $\ff$, $\Lambda$, and $\Sigma$ are as in \Cref{ex:BruhatTitsBuilding1}.
    We claim that $\build_{\BT}$ is a $G$-building of type $(\Sigma^\vee,\mathfrak{R},\Lambda^n)$.
    Recall that $\build_{\BT} = ( G\times \mathfrak{A})/\!\!\sim$, where $\mathfrak{A}$ is the affine space of root group valuations.
    We already saw in \Cref{ex:BruhatTitsBuilding1} that $G$ acts on $\build_{\BT}$ via $g.[h,x] = [gh,x]$.
    There is a chart $f_0 \colon \aa \to \build_{\BT}$, $x \mapsto [\Id,x]$ and the atlas $\atlas$ is the orbit of this chart under the action of $G$.
    Thus $\atlas$ is naturally endowed with a $G$-action and these actions are compatible, hence $\build_{\BT}$ is a $G$-building of type $(\Sigma^\vee,\mathfrak{R},\Lambda^n)$.
    Note that by definition $G$ acts transitively on the atlas $\atlas$.
    However, when $\Lambda \neq \rr$, the action of $G$ on $\build_\BT$ is not transitive, see e.g.\ \Cref{prop:TransitiveActionOnBuilding}.
\end{example}

\begin{example}[Symmetric building $\build_S$, \Cref{ex:SymmetricBuilding1} revisited]
\label{ex:SymmetricBuilding2}
We claim that the symmetric building $\build_S$ associated to a semisimple self-adjoint linear algebraic $\qq$-group $\mathbf{G} < \operatorname{SL}_n$ and a real closed field $\ff$ endowed with an order-compatible valuation $v \colon \ff^\times \to \Lambda$ (not necessarily of rank one), as defined in \Cref{ex:SymmetricBuilding1}, is an example of a $G$-building, where $G\coloneqq \mathbf{G}(\ff)$.

Recall that $G$ acts on $X_\ff=G. \operatorname{Id} \subseteq P_1(n,\ff)$ by congruence, and thus on $\build_S=X_\ff/\!\! \sim$.
It is left to define the action on the atlas $\atlas$ and to show that these actions are compatible.
Fix the base point $o =[\Id] \in \build_S$.
The standard apartment is identified with $\aa = A_\ff. o$, where $A_\ff$ is the semialgebraically connected component of the $\ff$-points $\mathbf{S}(\ff)$ of a maximal $\ff$-split torus $\mathbf{S}<\mathbf{G}$.
Denote by $f_0$ the inclusion from $\aa$ to $\build_S$.
The atlas $\atlas$ is the set $\{g.f_0 \colon \aa \to \build_S \colon g \in G\}$, which is naturally endowed with a left $G$-action.
Clearly, the two actions are compatible, and hence $\build_S$ is a $G$-building of type $(\Sigma^\vee, \Lambda,\Lambda^n)$.
By definition, the actions of $G$ on $\build$ and $\atlas$ are transitive.

The descriptions of the following stabilizers for the symmetric building will be useful.
We define $M_\ff$ to be the $\ff$-extension of $M_\rr$, or equivalently the centralizer of $A_\ff$ in $\mathbf{K}(\ff)$, see \Cref{ex:SymmetricBuilding1}.
Recall that $\mathcal{O}$ denotes the valuation ring of $\ff$.
For any semialgebraic group $H \subseteq \slnf \subset \ff^{n \times n}$, let $H(\mathcal O) \coloneqq H \cap \mathcal{O}^{n \times n}$.
Then $H(\mathcal O)$ is again a group.
If $H$ are the $\ff$-points of a semisimple algebraic group $\mathbf{H}<\operatorname{SL}_n$, we sometimes also write $\mathbf{H}(\mathcal{O})$.

\begin{prop}\label{prop:homogeneous_stabilizers}
    The stabilizers satisfy
    \begin{itemize}
        \item $\operatorname{Stab}_{G}(o) = \mathbf{G}(\mathcal{O})$, and
        \item
        $\operatorname{Stab}_{G}(f_0) =        M_\ff A_\ff(\mathcal{O}) 
        = \operatorname{Cent}_{G}(A_\ff) \cap \mathcal{O}^{n\times n} = \operatorname{Cent}_{G}(\mathbf{S}(\ff)) \cap \mathcal{O}^{n \times n}$. 
    \end{itemize}
    When $\mathbf{G}$ is $\ff$-split, $\operatorname{Stab}_{G}(f_0) = \mathbf{S}(\ff) \cap \mathcal{O}^{n \times n}$.
\end{prop}
\begin{proof}
    The first statement is \cite[Theorem 4.19]{appenzeller2025buildings}. By \cite[Lemma 5.18 and Theorem 5.19]{appenzeller2025buildings} we have
    $
    \operatorname{Stab}_{G}(f_0) = M_\ff A_\ff(\mathcal{O}) = \operatorname{Cent}_{G}(A_\ff) \cap \mathcal{O}^{n\times n},
    $
    and since $A_\ff$ is Zariski-dense in $\mathbf{S}(\ff)$ ($\mathbf{S}$ is Zariski-connected), also the last equality follows.

    When $\mathbf{G}$ is $\ff$-split, $\mathbf{S}$ is a maximal torus (not just a maximal $\ff$-split torus). Since $\mathbf{G}$ is semisimple, the root space decomposition then implies that the Lie algebra of $\mathbf{S}$ coincides with the Lie algebra of $\operatorname{Cent}_{G}(\mathbf{S}(\ff))$.
    However, $\mathbf{S}(\ff) \subseteq \operatorname{Cent}_{G}(\mathbf{S}(\ff)) = \operatorname{Cent}_{\mathbf{G}}(\mathbf{S})(\ff)$ and $\operatorname{Cent}_{\mathbf{G}}(\mathbf{S})$ is connected \cite[Corollary 11.12]{Borellinearalgebraixgroups}, so equality holds. 
\end{proof}
Since $G$ acts transitively on $\build_S$, by \Cref{prop:homogeneous_stabilizers}, the orbit-stabilizer-theorem implies that $\build_S \cong G/\mathbf{G}(\mathcal{O})$ as $G$-homogeneous sets.
\end{example}

\subsection{Extending morphisms of apartments I}
\label{subsection:ExtendingMorphismApartments}
The goal of this section is to prove \Cref{intro:thm:baby_morphismofGbuildings}, which gives conditions to construct morphisms of $G$-buildings.  
Let us recall the theorem. 

For this let $\Lambda$, $\Lambda'$ be ordered abelian groups, $\Phi$, $\Phi'$ crystallographic root systems of rank $n$, $m$ respectively with spherical Weyl groups $\sphwg$, $\sphwg'$, and $\mathbb{A}$, $\mathbb{A}'$ the associated model apartments and consider $T < \mathbb{A}$, $T'< \mathbb{A}'$ the translation subgroups of the respective affine Weyl groups $\affwg = T \rtimes \sphwg$, $\affwg' = T' \rtimes \sphwg'$.

We assume $G$ acts transitively on $\build$ and $\atlas$, (for the situation when $G$ does not act transitively on $\build$, see \Cref{thm:morphismofGbuildings} in the next section).

\begin{thm}[\Cref{intro:thm:baby_morphismofGbuildings}]
\label{thm:baby_morphismofGbuildings}
    Let  $(\build,\atlas)$ be a $G$-building and $(\build',\atlas')$ a $G'$-building. Let $\rho \colon G \to G'$ be a group homomorphism and $\tau \colon \aa \to \aa'$ a morphism of the respective model apartments of $\build$ and $\build'$.
    If $G$ acts transitively both on $\build$ and on $\atlas$ and there exist charts $f \in \atlas$ and $f' \in \atlas'$ such that 
    \begin{enumerate}[label=(\arabic*)]
        \item \label{condition:baby_stabilizerofapoint} 
        $\rho(\operatorname{Stab}_G(f(0))) \subseteq \operatorname{Stab}_{G'}(f'(0))$, and
        \item \label{condition:baby_stabilizerofchart} 
        $\rho(\operatorname{Stab}_G(f)) \subseteq \operatorname{Stab}_{G'}(f')$, and
        \item \label{condition:baby_Afw} 
        for all $x\in \aa$, there exists $g\in G$ such that 
        $$
        g.f(0) = f(x) \quad \text{ and } \quad \rho(g).f'(0) = f'(\tau(x)),
        $$
    \end{enumerate}
    then there exists a $\rho$-equivariant morphism $m\coloneqq(\psi, \varphi, \tau) \colon \build \to \build'$.
    
    Additionally,
    \begin{enumerate}[label=(\alph*)]
        \item \label{thm:baby_monomorphism}
        if $\tau$ and $\rho$ are injective and $\rho(\mathrm{Stab}_G(f))=\mathrm{Stab}_{G'}(f')$, then $m$ is injective;
        
        \item \label{thm:baby_epimorphism}
        if $G'$ acts transitively on $\build'$ and $\atlas'$, and $\tau$ and $\rho$ are surjective, then $m$ is surjective; 
        
        \item \label{thm:baby_isomorphism}
        if $\tau$ is an isomorphism, $G'$ acts transitively on $\build'$ and $\atlas'$, $\rho$ is an isomorphism of groups, and the two inclusions \ref{condition:baby_stabilizerofapoint} and \ref{condition:baby_stabilizerofchart} hold as equalities, then there exists an inverse morphism. That is, $(\build,\atlas)$ and $(\build',\atlas')$
        are isomorphic.
    \end{enumerate}
\end{thm}
\begin{proof}
We start with the construction of $\psi \colon \build \to \build'$. Since $G$ acts transitively on $\build$, every element of $\build$ is of the form $g.f(0)$ for some $g\in G$.
    We define $\psi$ by
    \[
    \psi \left(g.f\left(0\right)\right) \coloneqq \rho(g).f'\left(0\right)
    \]
    and note that $\psi$ is well-defined: if $g.f(0) = h.f(0)$, then $g^{-1}h \in \operatorname{Stab}_G(f(0))$, so $\rho(g^{-1}h) \in \operatorname{Stab}_{G'}(f'(0))$ by condition \ref{condition:baby_stabilizerofapoint} and hence $\psi(g.f(0)) = \rho(g).f'(0)=  \rho(h).f'(0) =\psi(h.f(0))$.

    Similarly, since $G$ acts transitively on $\atlas$, every element of $\atlas$ is of the form $g.f$ for some $g \in G$, so we define $\varphi \colon \atlas \to \atlas'$ by 
    \[
    \varphi(g.f) \coloneqq \rho(g).f'.
    \]
    The map $\varphi$ is well-defined by condition \ref{condition:baby_stabilizerofchart}.
    By construction, $\psi$ and $\varphi$ are $\rho$-equivariant.

    To show that $(\psi,\varphi,\tau)$ is a morphism of buildings,
    it remains to check that the following diagram commutes for every chart $ z  \in \atlas$
\[
\begin{tikzcd}
	{\mathbb{A}} & \build \\
	{\mathbb{A'}} & \build'.
	\arrow["z", from=1-1, to=1-2]
	\arrow["\tau"', from=1-1, to=2-1]
	\arrow["\psi", from=1-2, to=2-2]
	\arrow["{\varphi(z)}"', from=2-1, to=2-2]
\end{tikzcd}\] 
    Let $z\in \atlas$ and $x \in \mathbb{A}$.
    By transitivity of the $G$-action on $\atlas$, there exists $g\in G$ with $g.f = z$. Moreover, condition \ref{condition:baby_Afw} provides us with $h\in G$ such that $h.f = f\circ t^x$ and $\rho(h).f' = f' \circ t^{\tau(x)}$. Then
    \begin{align*}
         (\psi \circ z) (x) &= \psi ( g.f(x) ) 
         = \psi(g.h.f(0)) = \rho(g).\rho(h).f'(0)  \\ 
        &= \rho(g).f'(\tau(x)) 
        = \varphi(g.f)(\tau(x)) 
        = \varphi(z)(\tau(x)),
    \end{align*}
    so the diagram commutes.
    
    To prove \ref{thm:monomorphism} (the injectivity of $m$) we assume that $\tau$ and $\rho$ are injective and that $\rho(\operatorname{Stab}_G(f)) = \operatorname{Stab}_{G'}(f')$.
     To show that $\psi$ is injective, let $g,h\in G$ such that $\psi(g.f(0)) = \psi(h.f(0))$, and we check that $g.f(0) = h.f(0)$.
    From axiom \ref{axiom:A3} and the fact that $G$ acts transitively on $\atlas$, there exist $k \in G$ and $x$, $y \in \mathbb{A}$ so that 
    \[
        k.f(x)=g.f(0) \text{ and } k.f(y)=h.f(0).
    \]
    We compose with $\psi$ and obtain 
    \[
        \rho(k).f'(\tau(x))
        = \psi(k.f(x))
        =\psi(k.f(y))
        =\rho(k).f'(\tau(y)).
    \]
    Thus by the injectivity of $f'$ and $\tau$, it follows that $x=y$, so
    \[
        g.f(0)=k.f(x)=h.f(0),
    \]
    which proves that $\psi$ is injective. 
    
    Now let $f_1,f_2\in \atlas$ with $\varphi(f_1)=\varphi (f_2)$.
    Since $G$ acts transitively on $\atlas$, there exist $g_1$, $g_2 \in G$ such that $f_1=g_1.f$ and $f_2=g_2.f$.
    By construction of $\varphi$ we get $\rho(g_1).f' = \rho(g_2).f'$, so
    \[
        \rho(g_1^{-1}g_2) \in \mathrm{Stab}_{G'}(f')= \rho (\mathrm{Stab}_{G}(f)).
    \]
    By injectivity of $\rho$, we have $g_1^{-1}g_2\in \mathrm{Stab}_{G}(f)$, and thus $f_1=g_1.f = g_2.f =f_2$, so $\varphi$ is injective.

    To prove \ref{thm:epimorphism} (the surjectivity of $m$), we assume that $G'$ acts transitively on $\atlas'$, and $\tau$ and $\rho$ are surjective.
    Let $z'\in \atlas'$. By transitivity of the $G'$-action on $\atlas'$, there exists $g'\in G'$ such that $z'=g'.f'$.
    Moreover, $\rho$ is surjective so that $g' = \rho(g)$ for some $g\in G$. Hence $\varphi(g.f)=g'.f'=z'$, so $\varphi$ is surjective.
    
    We now prove the surjectivity of $\psi$.
    Let $p'\in \build'$. 
    By transitivity of the action of $G$ on $\build'$, there exists $z'\in \atlas'$ such that $p'=z'(0)$. 
    Since the action of $G'$ on $\atlas'$ is transitive and $\rho$ is surjective, there exists $g\in G$ with $p' = z'(0) = \rho(g).f'(0)$. Hence 
    \[
    \psi(g.f(0))=\rho(g).f'(0)=z'(0)=p',
    \]
    thus $\psi$ is also surjective. Finally, $\tau$ is surjective by assumption.

    To prove \ref{thm:isomorphism} (that $m$ is an isomorphism) we suppose that $\tau \colon \mathbb{A} \to \mathbb{A}'$ is an isomorphism of apartments, $G'$ acts transitively on $\build'$ and $\atlas'$, $\rho$ is an isomorphism, and the two inclusions in the conditions of the theorem are equalities.
    We can now apply \Cref{thm:baby_morphismofGbuildings} to $\tau^{-1}$, $\rho^{-1}$ and the charts $f'$ and $f$ to obtain a $\rho^{-1}$-equivariant morphism $(\psi',\varphi',\tau^{-1})$ from $(\build', \atlas')$ to $(\build,\atlas)$, such that $\varphi'(f')=f$.
    By equivariance of $\psi$ and $\varphi$, we have for every $g
    \in G$ and $x\in \mathbb{A}$ 
        \begin{align*}
        \psi' \circ \psi (g.f(x))
        &= \psi' (\rho(g).f'(x)) 
        = g. f(x) \quad \text{ and }\\
        \varphi' \circ \varphi (g.f)
        &= \varphi' (\rho(g). f') 
        = g.f,
        \end{align*}
    and similarly $\psi \circ \psi' = \Id_{\build'}$ and $ \varphi \circ \varphi' = \Id_{\atlas'}$.
    Hence $\build$ and $\build'$ are isomorphic as generalized affine buildings.
\end{proof}

\subsection{Extending morphisms of apartments II}
\label{subsection:ExtendingMorphismApartments2}

The goal of this section is to prove \Cref{thm:morphismofGbuildings}, which provides morphisms between $G$-buildings even when $G$ does not act transitively on $\build$. The tradeoff is, that the conditions \ref{condition:baby_stabilizerofapoint}, \ref{condition:baby_stabilizerofchart} and \ref{condition:baby_Afw} are replaced by slightly stronger conditions. For this we introduce some notation to study more closely the stabilizers (and some of their cosets) of charts.

\begin{defn}\label{def:A_fw}
Let $G$ be a group and $(\build,\atlas)$ a $G$-building.
For $f \in \atlas$ and $w \in \affwg = T \rtimes \sphwg$, we define the subset $A_{f,w}$ of elements of $G$ that act on $f(\aa)$ the same way as $w$ acts on $\aa$, i.e.\
\[
A_{f,w} \coloneqq \{g \in G \colon g.f = f \circ w \} \subseteq G.
\]
\end{defn}
Note that $A_{f,w}$ may be empty. However if $G$ acts transitively on the atlas $\atlas$ (which is often the case), then axiom \ref{axiom:A1} is equivalent to asking that $A_{f,w} \neq \emptyset$ for all $f \in \atlas $ and $ w \in \affwg$. 

\begin{prop}\label{prop:Afw}
Assume $G$ acts transitively on $\atlas$ and let $f\in \atlas$. 
Then,
\begin{itemize}
    \item for all $w,w'\in \affwg$ we have $A_{f,w}A_{f,w'} = A_{f,ww'}$;
    \item  for $w=\Id$ we have $A_{f,\Id} = \operatorname{Stab}_G(f)$, and for every $w\in \affwg$ there exists $g\in G$ such that $A_{f,w} = gA_{f,\Id}$, so $A_{f,w}$ is a coset of $\operatorname{Stab}_G(f)$;
    \item the union $N_f \coloneqq \bigcup_{w\in \affwg} A_{f,w} $ of all these cosets is a group satisfying $ N_f/\operatorname{Stab}_{G}(f) \cong \affwg$.
\end{itemize}
\end{prop}
\begin{remark}
\label{remark: Nf for BT}
    For the Bruhat--Tits building (\Cref{ex:BruhatTitsBuilding1}), one has $N_{f_0} = \operatorname{Nor}_{\mathbf{G}}(\mathbf{S})(\ff)$ for the standard chart $f_0 \colon \aa \to \build_\BT $, \cite[(7.4.10)]{BruhatTits}.
\end{remark}
\begin{proof}[Proof of \Cref{prop:Afw}]
    By definition, $A_{f,\Id} = \operatorname{Stab}_G(f)$.
    For $w\in \affwg$, we have $f\circ w \in \atlas$ by axiom \ref{axiom:A1} and by the transitivity of the action of $G$ on $\mathcal{A}$ there exists some $g\in G$ with $g.f = f\circ w$.
    Then
    \begin{align*}
        A_{f,w} &= \{ h \in G \colon h.f = f\circ w = g.f \} 
        = \{h \in G \colon g^{-1}h.f = f\} \\
        &= g\{h' \in G \colon h'.f=f\} = gA_{f,\Id} = g \operatorname{Stab}_G(f).
    \end{align*}
    
    If $w,w' \in \affwg$ and $g\in A_{f,w}$, $g' \in A_{f,w'}$, then $gg'.f = g.f \circ w' = f \circ ww'$, so $gg' \in A_{f,ww'}$. Similarly we have $g^{-1} \in A_{f,w^{-1}}$, since for all $x \in \aa$ we have
    $$
    g^{-1}.f(x) = g^{-1}.(f \circ w) (w^{-1} x) = g^{-1}g.f(w^{-1} x) = f \circ w^{-1} (x).
    $$
    This shows that $N_f$ is a subgroup of $G$ and that the projection $N_f \twoheadrightarrow \affwg$ sending $g\in A_{f,w}$ to $w$ is a group homomorphism with kernel $A_{f,\Id} \triangleleft N_f$.
    
    We have seen that $A_{f,w}A_{f,w'} \subseteq A_{f,ww'}$.
    For the other direction let $g\in A_{f,ww'}$, take some arbitrary $h\in A_{f,w}$ and define $h' \coloneqq h^{-1}g \in A_{f,w^{-1}}A_{f,ww'} \subseteq A_{f,w'}$.
    Then $g= hh' \in A_{f,w}A_{f,w'}$.
\end{proof}

We can now state the version of our main theorem, that holds even in the case when $G$ does not act transitively on $\build$. 
The proof is analogous to the one of \Cref{thm:baby_morphismofGbuildings}, but a bit more involved.

\begin{thm}
\label{thm:morphismofGbuildings}
    Let $(\build,\atlas)$ be a $G$-building and $(\build',\atlas')$ a $G'$-building. Let $\rho \colon G \to G'$ be a group homomorphism and $\tau \colon \aa \to \aa'$ a morphism of the respective model apartments of $\build$ and $\build'$.
    If $G$ acts transitively on $\atlas$ and there exist charts $f \in \atlas$ and $f' \in \atlas'$ such that 
    \begin{enumerate}[label=(\roman*)]
        \item
        \label{condition:stabilizerofapoint}
        $\rho(\operatorname{Stab}_G(f(x
        ))) \subseteq \operatorname{Stab}_{G'}(f'(\tau (x)))$ for all $x \in \mathbb{A}$, and
        \item
        \label{condition:Afw} 
        $\rho(A_{f,w}) \subseteq A_{f',\sigma(w)}$ for all $w\in \affwg$,
    \end{enumerate}
    then there exists a $\rho$-equivariant morphism $m\coloneqq(\psi, \varphi, \tau) \colon \build \to \build'$.
    
    Additionally,
    \begin{enumerate}[label=(\alph*)]
        \item \label{thm:monomorphism}
        if $\tau$ and $\rho$ are injective and $\rho(\mathrm{Stab}_G(f))=\mathrm{Stab}_{G'}(f')$, then $m$ is injective;
        
        \item \label{thm:epimorphism}
        if $G'$ acts transitively on $\atlas'$, and $\tau$ and $\rho$ are surjective, then $m$ is surjective;
        
        \item \label{thm:isomorphism}
        if $\tau$ is an isomorphism, $G'$ acts transitively on $\atlas'$, $\rho$ is an isomorphism of groups, and the two inclusions \ref{condition:stabilizerofapoint} and \ref{condition:Afw} hold as equalities, then there exists an inverse morphism. That is, $(\build,\atlas)$ and $(\build',\atlas')$
        are isomorphic.
    \end{enumerate}
\end{thm}
\begin{proof}
    Assume $G$ acts transitively on $\atlas$, and fix the charts $f \in \atlas$ and $f' \in \atlas'$ satisfying \ref{condition:stabilizerofapoint} and \ref{condition:Afw}.
    We start with the construction of $\varphi \colon \atlas \to \atlas'$.
    Indeed, since $G$ acts transitively on $\atlas$, every element of $\atlas$ is of the form $g.f$ for some $g \in G$, so we define a $\rho$-equivariant function $\varphi \colon \atlas \to \atlas'$ by 
    \[
    \varphi(g.f) \coloneqq \rho(g).f'.
    \]
    The map $\varphi$ is well-defined.
    Indeed, if $g.f=h.f$ for some $g,h \in G$, then $g^{-1}h.f = f$ so that using \ref{condition:Afw} for $w=\Id$ implies that $\rho(g^{-1}h) \in \operatorname{Stab}_{G'}(f')$ and thus
    \[
    \varphi (g.f)=\rho(g).f' = \rho(h).f'=\varphi(h.f).
    \]
    
    Next, we define $\psi \colon \build \to \build'$.
    Any element in $\build$ is of the form $g.f(x)$ for some $g\in G$ and $x \in \mathbb{A}$---this follows from axiom \ref{axiom:A3} and the fact that $G$ acts transitively on $\atlas$.
    For every $g\in G$ and $x\in \mathbb{A}$, we define $\psi$ by
    \[
    \psi \left(g.f\left(x\right)\right) \coloneqq \rho(g).f'\left(\tau \left(x\right)\right).
    \]
    We check that $\psi$ is well-defined.
    Let $g,h \in G$ and $x,y \in \mathbb{A}$ such that $g.f(x)=h.f(y)$.
    From axiom \ref{axiom:A2} there exists $w \in \affwg$ such that  
    \[
    (g.f)|_\Omega = \left(h.f \circ w \right)|_\Omega, \ \text{ where } \Omega \coloneqq (g.f)^{-1}\left(g.f\left(\mathbb{A}\right) \cap h.f\left(\mathbb{A}\right) \right).
    \]
    Note that $x\in \Omega$ since $g.f(x)=h.f(y)$. 
    In particular 
    \[
    h.f(y)=g.f(x) = h.f(w(x)),
    \]
    so that $w(x)=y$ by injectivity of $f$. 
    The set $A_{f,w}$ is non-empty by transitivity of the action of $G$ on $\atlas$ and axiom \ref{axiom:A1}.
    For every $g_w \in A_{f,w}$ we have
    \[
    g.f(x) = h.f(w(x)) = hg_w.f(x),
    \]
    so that $\rho(g^{-1}hg_w) \in \rho(\operatorname{Stab}_G(f(x))) \subseteq \operatorname{Stab}_G(f'(\tau(x)))$ using \ref{condition:stabilizerofapoint}, or equivalently 
    \[
    \rho(g).f'(\tau (x)) = \rho(hg_w).f'(\tau(x)).
    \]
    By \ref{condition:Afw}, we also have $\rho(g_w) \in \rho(A_{f,w}) \subseteq A_{f',\sigma(w)}$ so that
    \begin{align*}
        \rho(g).f'(\tau (x)) &= \rho(hg_w).f'(\tau(x)) 
        = \rho(h).f' \left( \sigma\left(w\right)\left(\tau\left(x\right)\right) \right)\\
        &= \rho(h).f'(\tau(w(x)))  
        = \rho(h).f'(\tau(y)),
    \end{align*}
    where the third equality follows since $\tau$ is a morphism of apartments.
    This shows that $\psi$ is well-defined.
    By construction, $\varphi$ and $\psi$ are $\rho$-equivariant.

    To show that $(\psi,\varphi,\tau)$ is a $\rho$-equivariant morphism of buildings,
    it remains to check that the following diagram commutes for every chart $ z  \in \atlas$
\[
\begin{tikzcd}
	{\mathbb{A}} & \build \\
	{\mathbb{A'}} & \build'.
	\arrow["z", from=1-1, to=1-2]
	\arrow["\tau"', from=1-1, to=2-1]
	\arrow["\psi", from=1-2, to=2-2]
	\arrow["{\varphi(z)}"', from=2-1, to=2-2]
\end{tikzcd}\] 
    Let $z\in \atlas$ and $x \in \mathbb{A}$.
    By transitivity of the $G$-action on $\atlas$, there exists $g\in G$ with $g.f = z$. Then
    \[
        (\psi \circ z) (x) = \psi ( g.f(x) ) = \rho(g).f'(\tau(x)) 
        = \varphi(g.f)(\tau(x)) 
        = \varphi(z)(\tau(x)),
    \]
    so the diagram commutes.
    
    To prove \ref{thm:monomorphism} (the injectivity of $m$) we assume that $\tau$ and $\rho$ are injective and $\rho(\operatorname{Stab}_G(f)) = \operatorname{Stab}_{G'}(f')$.
     For every $p\in \build$, there exists $g\in G$ and $x\in \mathbb{A}$ such that $g.f(x)=p$.
     So to show that $\psi$ is injective, let $g,h\in G$ and $x,y\in \mathbb{A}$ such that $\psi(g.f(x)) = \psi(h.f(y))$,  and we check that $g.f(x) = h.f(y)$. 
    From axiom \ref{axiom:A3} and the fact that $G$ acts transitively on $\atlas$, there exist $g_1 \in G$ and $x_1$, $y_1\in \mathbb{A}$ so that 
    \[
        g_1.f(x_1)=g.f(x) \text{ and } g_1.f(y_1)=h.f(y).
    \]
    We compose with $\psi$ and obtain 
    \[
        \rho(g_1).f'(\tau(x_1))
        = \psi(g_1.f(x_1))
        =\psi(g_1.f(y_1))
        =\rho(g_1).f'(\tau(y_1)).
    \]
    Thus by the injectivity of $\tau$ and the injectivity of $f$ and $f'$, it follows that $x_1=y_1$, so
    \[
        g.f(x)=g_1.f(x_1)=h.f(y),
    \]
    which proves that $\psi$ is injective. 
    
    Now suppose $\rho$ is injective and $\rho(\mathrm{Stab}_G(f))=\mathrm{Stab}_{G'}(f')$. 
    Let $f_1,f_2\in \atlas$ with $\varphi(f_1)=\varphi (f_2)$.
    Since $G$ acts transitively on $\atlas$, there exist $g_1$, $g_2 \in G$ such that $f_1=g_1.f$ and $f_2=g_2.f$.
    By construction of $\varphi$ we get $\rho(g_1).f' = \rho(g_2).f'$, so
    \[
        \rho(g_1^{-1}g_2) \in \mathrm{Stab}_{G'}(f')= \rho (\mathrm{Stab}_{G}(f)).
    \]
    By injectivity of $\rho$, we have $g_1^{-1}g_2\in \mathrm{Stab}_{G}(f)$, and thus $f_1=g_1.f = g_2.f =f_2$, so $\varphi$ is injective.

    To prove \ref{thm:epimorphism} (the surjectivity of $m$), we assume that $G'$ acts transitively on $\atlas'$, and $\tau$ and $\rho$ are surjective.
    Let $z'\in \atlas'$. By transitivity of the $G'$-action on $\atlas'$, there exists $g'\in G'$ such that $z'=g'.f'$.
    Moreover, $\rho$ is surjective so that $g' = \rho(g)$ for some $g\in G$. Hence $\varphi(g.f)=g'.f'=z'$, so $\varphi$ is surjective.
    
    We now prove the surjectivity of $\psi$.
    Let $p'\in \build'$. 
    By axiom \ref{axiom:A3}, there exist $x' \in \mathbb{A}'$ and $z'\in \atlas'$ such that $p'=z'(x')$. 
    Since the action of $G'$ on $\atlas'$ is transitive and $\rho$ is surjective, there exists $g\in G$ with $p' = z'(x') = \rho(g).f'(x')$.
    By the surjectivity of $\tau$, there exists $x\in\mathbb{A}$ such that $x' = \tau(x)$. Hence 
    \[
    \psi(g.f(x))=\rho(g).f'(\tau(x))=z'(x')=p',
    \]
    thus $\psi$ is also surjective.

    To prove \ref{thm:isomorphism} (that $m$ is an isomorphism) we suppose that $\tau \colon \mathbb{A} \to \mathbb{A}'$ is an isomorphism of apartments, $G'$ acts transitively on $\atlas'$, $\rho$ is an isomorphism, and the two inclusions in the conditions of the theorem are equalities.
    We can now apply \Cref{thm:morphismofGbuildings} to $\tau^{-1}$, $\rho^{-1}$ and the charts $f'$ and $f$ to obtain a $\rho^{-1}$-equivariant morphism $(\psi',\varphi',\tau^{-1})$ from $(\build', \atlas')$ to $(\build,\atlas)$, such that $\varphi'(f')=f$.
    By equivariance of $\psi$ and $\varphi$, we have for every $g
    \in G$ and $x\in \mathbb{A}$ 
        \begin{align*}
        \psi' \circ \psi (g.f(x))
        &= \psi' (\rho(g).f'(x)) 
        = g. f(x) \quad \text{ and }\\
        \varphi' \circ \varphi (g.f)
        &= \varphi' (\rho(g). f') 
        = g.f,
        \end{align*}
    and similarly $\psi \circ \psi' = \Id_{\build'}$ and $ \varphi \circ \varphi' = \Id_{\atlas'}$.
    Hence $\build$ and $\build'$ are isomorphic as generalized affine buildings.
\end{proof}

\begin{remark}\label{rem:unique}
    The morphisms of Theorems \ref{thm:baby_morphismofGbuildings} and \ref{thm:morphismofGbuildings} are unique in the following sense. Assume $G$ acts transitively on $\atlas$ and fix a chart $f\in \atlas$. If two $\rho$-equivariant morphisms $m = (\psi, \varphi, \tau)$, $m'=(\psi',\varphi', \tau') \colon (\build, \atlas) \to (\build', \atlas')$ satisfy $\tau = \tau'$ and $\varphi(f) = \varphi(f')$, then $m = m'$. 
\end{remark}
\begin{proof}
    Every element of $\atlas$ is of the form $g.f$ for some $g\in G$. Then
    $$
    \varphi(g.f) = \rho(g)\varphi(f) = \rho(g)\varphi'(f') = \varphi'(g.f),
    $$
    so $\varphi = \varphi'$. Any element of $\build$ is of the form $g.f(x)$ for some $g\in G$, $x \in \aa$. Then
    $$
    \psi(g.f(x)) = \rho(g).\varphi(f)(\tau(x)) = \rho(g).\varphi'(f)(\tau(x)) = \psi'(g.f(x)),
    $$
    so $\psi = \psi'$.
\end{proof}

\section{Relations between the different buildings}
\label{section:RelationsBuildings}
The goal of this section is to apply \Cref{thm:morphismofGbuildings} to construct morphisms 
\[
\build_L \xrightarrow[\cong]{(\textnormal{\Cref{thm:Lattice_to_KT} })} 
\build_S \xhookrightarrow[]{(\textnormal{\Cref{corol:MorphismKramerTentToBruhatTits} })} 
\build_\BT \xhookrightarrow[]{(\textnormal{\Cref{thm:BTNorm} })} 
\build_N,
\]
whenever both the source and the target are defined for a choice of algebraic group $\mathbf{G}$ and valued field $\ff$.
All three morphisms are injective morphisms, and when $\Lambda = \rr$ (the valuation is surjective to $\rr$), all morphisms are isomorphisms.

\subsection{Lattice and symmetric buildings}
\label{subsection:LatticeSymmetricBuildings}

The goal of this section is the proof of \Cref{thm:Lattice_to_KT}, in which we construct an isomorphism from the lattice building $\build_L$ (\Cref{ex:LatticeBuilding1}) to the symmetric building $\build_S$ (\Cref{ex:SymmetricBuilding1}), in the case where $\mathbf{G} = \operatorname{SL}_n$ and $\ff$ is a real closed field with an order-compatible valuation $v \colon \ff^\times \to \Lambda = \ff^\times / \mathcal O^\times$.
The isomorphism is less natural than one might expect: it uses a non-trivial isomorphism $\tau$ of apartments given by multiplying by $(-1)$, see \Cref{lem:L_KT_aptiso}. 
This is a situation analogous to isomorphisms in some discrete buildings that are not label-preserving.

Both $\build_L$ and $\build_S$ are of type $\aa(\Phi, \Lambda, \Lambda^{n-1})$, where the underlying root system of type $A_{n-1}$ is given by
\[
\Phi = \left\{ x_{ij} \coloneqq e_i-e_j \in V, \right\},
\]
where $V=\rr^n$ is the standard vector space with standard basis $\{e_1, \ldots , e_n\}$.
The basis $\Delta = \left\{ x_{i, i+1} \colon i \in \{1,2, \ldots, n-1\}\right\}$ of $\Phi$ induces an isomorphism 
\begin{align*}
   \operatorname{Span}_\qq(\Phi) \otimes \Lambda \to \Lambda^{n-1}, \quad \sum_{i=1}^{n-1}  x_{i, i+1} \otimes \lambda_i   \mapsto ( \lambda_1 , \ldots , \lambda_{n-1})
\end{align*}
of the apartment $\aa = \operatorname{Span}_\qq(\Phi) \otimes \Lambda$ with $\Lambda^{n-1}$.
The spherical Weyl group $\sphwg$ is the symmetric group $S_n$ on $n$ elements and acts on $\Phi$ via $\sigma(x_{ij}) = x_{\sigma(i)\sigma(j)}$ for $\sigma \in \sphwg$, and by linear extension on $\aa$.
The translation group $T$ is the full translation group $T \cong \Lambda^{n-1}$. 
We first develop the theory for the lattice building $\build_L$.

\subsubsection{Setup for the lattice building}
\label{sec: L_H_Setup}
For now, let $\ff$ be any valued field (not necessarily ordered).
Recall from \Cref{ex:LatticeBuilding1} that lattices are of the form $\mathcal O e_1' + \ldots + \mathcal O e_n'$, where $\{e_1', \ldots , e_n'\}$ is a basis of $\ff^n$ and $\mathcal{O}$ is the valuation ring, the lattice building $\build_L$ is the set of homothety classes of lattices, and that $\slnf$ acts transitively on $\build_L$ (\Cref{ex:LatticeBuilding2}).
Viewing $\slnf \subseteq \ff^{n\times n}$ we define $\operatorname{SL}_n(\mathcal O) \coloneqq \operatorname{SL}_n(\ff) \cap \mathcal O ^{n\times n}$.
The lattice $L_0= \mathcal O ^n$ corresponding to the standard basis of $\ff^n$ is called the \emph{standard lattice} and we call $[L_0]$ the \emph{base point} of $\build_L$.
The following is well-known in the discrete case, and a proof can be found in \cite{HebertIzquierdoLoisel_LambdaBuildingsAssocToQuasiSplitGroups}.

\begin{prop}[{\cite[Lemma 11.4]{HebertIzquierdoLoisel_LambdaBuildingsAssocToQuasiSplitGroups}}]  \label{prop:L_H_stab}
We have $\operatorname{Stab}_{\operatorname{SL}_n(\ff)}([L_0]) = \operatorname{SL}_n(\mathcal O)$.
\end{prop}

Recall that the chart given by \cite{BennettlambdabuildingsI} corresponding to the standard basis $\mathcal{E}$ is given by
\begin{align*}
f_{\mathcal{E}} \colon \aa \cong \Lambda^{n-1} & \to \build_L \\
    (\lambda_1,\ldots,\lambda_{n-1}) &\mapsto 
    \left[\mathcal{O}x_{\lambda_1}e_1 +\mathcal{O}\frac{x_{\lambda_2}}{x_{\lambda_1}} e_2 +\ldots +\mathcal{O}\frac{x_{\lambda_{n-1}}}{x_{\lambda_{n-2}}}e_{n-1}+\mathcal{O}\frac{1}{x_{\lambda_{n-1}}}e_n\right],
\end{align*}
where $x_\lambda \in \ff^\times$ such that $v(x_{\lambda}) = \lambda$.
\begin{lemma}\label{lem:L_KT_Lambdan_to_Lattice}
    For $a = \operatorname{Diag}(a_1, \ldots , a_n) \in \operatorname{SL}_n(\ff)$ and $\lambda_k = v\left( \prod_{i=1}^k a_i \right)$, we have 
    \[
    f_{\mathcal{E}}(\lambda_1 ,\ldots,\lambda_{n-1}) = a.[L_0].
    \]
    Every point in the image of $f_{\mathcal E}$ is of the form $a.[L_0]$ for some diagonal matrix $a = \operatorname{Diag}(a_1, \ldots , a_n) \in \slnf$.

    Moreover, if $\ff$ is also ordered, we can choose $a_i > 0$. 
\end{lemma}
\begin{proof}
    For $k \in \{1,\ldots,n-1\}$, we set $x_{\lambda_k} \coloneqq \prod_{i=1}^k a_i $, so that $x_{\lambda_k}/x_{\lambda_{k-1}} = a_k$ for all $ k \in \left\{2,3, \ldots , n-1\right\} $  and $1/x_{\lambda_{n-1}} =a_n$, since $\det(a)=1$.
    The first description then follows directly from the definition of $f_{\mathcal E}$.
    If we start with some $\lambda = (\lambda_1, \ldots, \lambda_{n-1}) \in \Lambda^{n-1}$, we can choose $a_1 \in \ff^\times$ with $v(a_1) = \lambda_1$, and then iteratively $a_k\in \ff^\times$ with $v(a_k) = \lambda_k-\lambda_{k-1}$ for all $k \in \{2, 3, \ldots , n-1\}$.
    Finally define $a_n \coloneqq 1/(a_1 \cdots a_{n-1}) \in \ff_{>0}$.
    Then
    \[
    v\left(\prod_{i=1}^k a_i\right) = \sum_{i=1}^k v(a_i) = \lambda_1 + (\lambda_2 - \lambda_1) + \ldots + (\lambda_k - \lambda_{k-1}) = \lambda_k,
    \]
    so $f_{\mathcal E}(\lambda_1, \ldots ,\lambda_{n-1}) = a.[L_0]$.
    
    Since $v(a_i) = v(-a_i)$, the last claim follows.
\end{proof}
The dual roots $\alpha_{ij} \in \Sigma \subseteq V^\star$ defined by $\alpha_{ij} \colon V \to \rr$, $(v_1, \ldots , v_n) \mapsto v_i - v_j$ define linear maps
\begin{align*}
    \alpha_{ij}\colon \operatorname{Span}_{\qq}(\Phi)\otimes_\qq \Lambda \to \Lambda, \quad \sum_{k=1}^{n-1} x_{k,k+1} \otimes \lambda_k \mapsto \sum_{k=1}^{n-1} \lambda_k \alpha_{ij}(x_{k,k+1}),
\end{align*}
that can be used to characterize the diagonal element $a$ in \Cref{lem:L_KT_Lambdan_to_Lattice}. 

\begin{lemma}\label{lem:L_KT_char_A_L}
    Let $a = \operatorname{Diag}(a_1, \ldots , a_n) \in \operatorname{SL}_n(\ff)$.
    A point $x\in \aa$ satisfies $f_{\mathcal E}(x) = a.[L_0]$ if and only if $\alpha_{ij}(x) = v(a_i/a_j)$ for all $\alpha_{ij} \in \Sigma$.
\end{lemma}
\begin{proof}
    Let $A \coloneqq \left\{\operatorname{Diag}(a_1, \ldots , a_n) \colon a_i \neq 0\right\}$ be the diagonal subgroup of $\slnf$.
    We describe the point $x\in\aa$ that corresponds to $a.[L_0]$ via the identifications
    \[\begin{tabular}{rcccl}
        $\aa$ & $\xleftarrow{\sim}$ &  $\Lambda^{n-1}$ & $\xrightarrow{\sim}$ & $A . [L_0]$\\
        $\sum_{k=1}^{n-1} x_{k,k+1} \otimes \lambda_k$ & $\mapsfrom$ & $(\lambda_k)_{k=1}^{n-1}$ & $\mapsto $ & $f_{\mathcal E} (\lambda_1, \ldots , \lambda_{n-1})$,
    \end{tabular}
    \]
    where $\mathcal{E}$ is the standard basis.
    By \Cref{lem:L_KT_Lambdan_to_Lattice}, $\lambda_k = \sum_{\ell = 1}^{k} v(a_\ell)$, so 
    \[
    x = \sum_{k=1}^{n-1} x_{k,k+1}\otimes v\left(\prod_{\ell=1}^k a_\ell\right).
    \]
    Now applying $\alpha_{ij}$ and using $\alpha_{ij}(x_{k,k+1}) = \delta_{ik} + \delta_{j,k+1} - \delta_{i,k+1} - \delta_{jk}$, where $\delta$ is the Kronecker-symbol, we obtain
    \begin{align*}
        \alpha_{ij}(x) & = 
        \sum_{k=1}^{n-1} v\left(\prod_{\ell =1}^k a_\ell \right) \alpha_{ij}(x_{k,k+1}) \\
        &= v\left(\prod_{\ell =1}^i a_\ell \right) + v\left(\prod_{\ell =1}^{j-1} a_\ell \right)
        - v\left(\prod_{\ell =1}^{i-1} a_\ell \right) - v\left(\prod_{\ell =1}^j a_\ell \right) 
        = v(a_i/a_j).
    \end{align*}
    On the other hand, if we know that $\alpha_{ij}(x)=v(a_i/a_j)$ for some $x \in \aa$ with $x = \sum_{k=1}^{n-1}x_{k,k+1} \otimes \lambda_k$, then we know by the same calculation and by the uniqueness of the $\lambda_k$ that $\lambda_k = v\left( \prod_{\ell = 1}^k a_\ell \right)$.
    By \Cref{lem:L_KT_Lambdan_to_Lattice}, $f_{\mathcal{E}}(x) = a.[L_0]$.
\end{proof}

\begin{prop} \label{prop:L_H_stabApartment}
The pointwise stabilizer of the standard apartment in $\build_L$ is given by
\[
\operatorname{Stab}_{\slnf}(f_{\mathcal E}(\aa)) = \left\{ \operatorname{Diag}(a_1, \ldots, a_n) \in \operatorname{SL}_n(\ff) \colon v(a_i) = 0  \right\}.
\]
\end{prop}
\begin{proof}
    A point $p\in f_{\mathcal E}(\aa)$ is a homothety class of lattices  of the form $p = \left[ \sum_{i=1}^n \mathcal O x_i e_i \right]$ where $x_i \in \ff^\times$.
    Acting by $g = \operatorname{Diag}(a_1, \ldots, a_n) $ with $v(a_i) = 0 $  gives $g.p = \left[ \sum_{i=1}^n \mathcal O a_i x_i e_i \right] = p$ since $\mathcal O a_i = \mathcal O$ when $a_i \in \mathcal O^\times$.
    On the other hand if
    \[
    g = \begin{pmatrix}
        g_{11} & \cdots & g_{1n} \\
        \vdots & \ddots& \vdots \\
        g_{n1} & \cdots & g_{nn}
    \end{pmatrix} \in \operatorname{SL}_n(\ff)
    \]
    fixes all points $p \in f_{\mathcal E}(\aa)$, then writing $p=a.[L_0]$ for $a = \operatorname{Diag}(a_1, \ldots , a_n)$ (using \Cref{lem:L_KT_Lambdan_to_Lattice} ) this means that $g.a.[L_0] = a.[L_0]$, so $a^{-1}ga \in \operatorname{Stab}_{\operatorname{SL}_n(\ff)}([L_0]) = \operatorname{SL}_n(\mathcal O)$ by \Cref{prop:L_H_stab}.
    In coordinates $g_{ij}a_j/a_i \in \mathcal O$ for all such $a$. For $i \neq j$, this implies $g_{ij} = 0$, so $g$ has to be diagonal and the diagonal entries have to satisfy $g_{ii} = g_{ii}a_i/a_i\in \mathcal O$. Since the stabilizer is closed under inverses, also $g_{ii}^{-1} \in \mathcal O$, so $v(g_{ii}) =0$.
\end{proof}

\subsubsection{Setup for the symmetric building}
Let us now suppose that $\ff$ is additionally real closed and the valuation and the order are compatible.
For $G=\slnf$, we have $K_\ff = \operatorname{SO}_n(\ff)$, and the semialgebraically connected component of the diagonal subgroup containing the identity is given by 
\[
A_{\ff} \coloneqq \left\{ \operatorname{Diag}(a_1, \ldots , a_n) \in \operatorname{SL}_n(\ff) \colon a_i > 0 \right\}.
\]
The root system relative to the maximal torus given by the diagonal subgroup can be identified with
\[
 \Sigma = \left\{ \alpha_{ij} \in V^\star \colon \alpha_{ij}(v_1, \ldots , v_n) = v_i - v_j \text{  for all } (v_1, \ldots , v_n) \in V \right\},
\]
so that the dual root system $\Sigma^\vee = \Phi$ is as defined earlier.
Similar to \Cref{lem:L_KT_char_A_L}, we set up a characterization of those $a\in A_\ff$ for which $a.o=f_0(x)$ for some $x \in \aa$, where $f_0$ is the standard chart as defined in \Cref{ex:SymmetricBuilding1}.

\begin{lemma}\label{lem:L_KT_char_A_H}
    A point $x\in \aa$ satisfies $f_0(x) = a.o$ for $a\in A_\ff$ if and only if $\alpha_{ij}(x) = (-v)(a_i/a_j)$ for all $\alpha_{ij} \in \Sigma$.
\end{lemma}
\begin{proof}
    For the direction ``$\implies$'', recall that the compatibility condition (\ref{eq:BH_compatible}) implies that for $\alpha \in \Sigma$, we have
    \[
    (-v)(\chi_\alpha(a)) = \alpha(x).
    \]
    In our specific case, if $a.o=f_0(x)$ for some $x\in \aa$, setting $\alpha = \alpha_{ij}$ we obtain exactly $(-v)(a_i/a_j)= \alpha_{ij}(x)$.
    
    On the other hand, any $x = \sum_{k=1}^{n-1} x_{k,k+1} \otimes \lambda_k \in \aa$ is uniquely determined by the $\lambda_k \in \Lambda$.
    So if $(-v)(a_i/a_j) = \alpha_{ij}(x) $, then
    \[
    (-v)(a_i/a_j)  = \sum_{k=1}^{n-1} \lambda_k \alpha_{ij}(x_{k,k+1})
    = \lambda_i + \lambda_{j-1} - \lambda_{i-1} - \lambda_{j}
    \]
    (with the convention $\lambda_0  = 0$).
    This results in the system of linear equations
    \[
    \begin{pmatrix}
        2 & -1 & 0 & \cdots & \\
        -1 & 2 & -1 & \ddots & \vdots\\
        0 & \ddots  & \ddots  &\ddots & 0 \\
        \vdots &\ddots&\ddots &2 & -1 \\
        &\cdots &0&-1 & 2
    \end{pmatrix} 
\begin{pmatrix}
    \lambda_1 \\ \lambda_2 \\ \vdots \\ \\ \lambda_k
\end{pmatrix} =
\begin{pmatrix}
    (-v)(a_1/a_2) \\ (-v)(a_2/a_3) \\ \vdots \\ \\ (-v)(a_{n-1}/a_n)
\end{pmatrix},
    \]
    which determines the solution uniquely since the matrix is invertible.
\end{proof}

The stabilizers for charts and points of $\build_L$ in Propositions \ref{prop:L_H_stab} and \ref{prop:L_H_stabApartment} coincide with the stabilizers for the action on $\build_S$.

\begin{prop}\label{prop:L_H_stab_H}
    The stabilizer of the base point $o\in \build_S$ is $\operatorname{Stab}_{\operatorname{SL}_n(\ff)}( o ) = \operatorname{SL}_n(\mathcal O) $ and the pointwise stabilizer of the standard apartment is
    \[
    \operatorname{Stab}_{\operatorname{SL}_n(\ff)}(f_0(\aa)) = \left\{ \operatorname{Diag}(a_1, \ldots, a_n) \in \operatorname{SL}_n(\ff) \colon v(a_i) = 0  \right\}.
    \]
\end{prop}
\begin{proof}
    By \Cref{prop:homogeneous_stabilizers}, we have $\operatorname{Stab}_{\slnf}(f_0(\aa)) = M_\ff A_\ff(\mathcal O)$, where $M_\ff \coloneqq \operatorname{Cent}_{\operatorname{SO}_{n}(\ff)}(A_\ff)$ consists of all diagonal matrices in $\operatorname{SL}_n(\ff)$ with entries $\pm 1$. 
\end{proof}

\subsubsection{Isomorphism between lattice and symmetric building}

In this subsection we will show that the lattice building $\build_L$ is isomorphic to the symmetric building $\build_S$.
Even though these two buildings have the same apartment $\aa$, we will not use the identity as an apartment morphism, but instead the inversion $\aa \to \aa$, $x \mapsto -x$.
In the discrete setting this corresponds to an isomorphism that may not preserve the type of the vertices.
\begin{lemma}\label{lem:L_KT_aptiso}
    Let $L \colon \operatorname{Span}_{\qq}(\Phi) \to \operatorname{Span}_{\qq}(\Phi)$, $v \mapsto -v$ be the inversion, $\gamma = \Id \colon \Lambda \to \Lambda$ the identity and $\sigma \colon \affwg \to \affwg$ the group homomorphism given by $\sigma(w)(x)=-\sigma(-x)$ for all $x \in \aa$.
    Then $\tau = (L,\Id, \sigma)$ is an isomorphism of apartments $\aa \to \aa$ of type $(\Phi, \Lambda, \Lambda^{n-1})$.
\end{lemma}
\begin{proof}
By definition, $L$ is linear, $\gamma$ is an order-preserving group homomorphism, $\sigma$ is a group homomorphism and $L$, $\gamma$ and $\sigma$ clearly verify
\[ \tau (w(x)) = \sigma(w) (\tau (x))\]
for all $w \in \affwg$ and $x \in \aa$.
Thus $\tau$ is a morphism of apartments.
Since $\tau$ is also its own inverse, it is an isomorphism. 
\end{proof}

Let now $\ff$ be a valued real closed field. Before we prove that $\build_L$ and $\build_S$ are isomorphic, we investigate the action of the translation part of the affine Weyl group. 
\begin{lemma}\label{lem:L_H_translation}
    For every $x\in \aa$, there exists $a = \operatorname{Diag}(a_1, \ldots , a_n) \in A_\ff$ that satisfies $f_{\mathcal E}(x) = a.f_{\mathcal{E}}(0)$ and $f_0(\tau(y)) = a^{-1}.f_0(0)$.
\end{lemma}
\begin{proof}
By \Cref{lem:L_KT_Lambdan_to_Lattice}, there exists $a = \operatorname{Diag}(a_1, \ldots , a_n) \in A_\ff$ with $a.f_{\mathcal E}(0) = f_{\mathcal E}(y)$. By \Cref{lem:L_KT_char_A_L}, it satisfies $\alpha_{ij}(y)=v(a_i/a_j)$. By \Cref{lem:L_KT_char_A_H}, $f_0(\tau(y)) = a^{-1}.f_0(0)$.
\end{proof}

\begin{thm}\label{thm:Lattice_to_KT}
There is an equivariant isomorphism of affine $\Lambda$-buildings between the lattice building $\build_L$ and the symmetric building $\build_S$.
\end{thm}

\begin{proof}
We use \Cref{thm:baby_morphismofGbuildings} as $\slnf$ acts transitively on $\build_L$ and $\atlas_L$, see \Cref{ex:LatticeBuilding2}.
Let $G=G'=\operatorname{SL}_n(\ff)$ and $\rho$ the identity map.
We take the apartment isomorphism $\tau = (L,\gamma,\sigma)$ from Lemma \ref{lem:L_KT_aptiso}. We consider the standard charts
\[
f \coloneqq f_{\mathcal E} \colon \aa  \to \build_L \quad \text{ and }\quad f' \coloneqq f_0 \colon \aa \to \build_S.
\]
By Propositions \ref{prop:L_H_stab} and \ref{prop:L_H_stab_H} we have 
$
   \operatorname{Stab}_{\operatorname{SL}_n(\ff)}(f(0)) =  \operatorname{Stab}_{\operatorname{SL}_n(\ff)}(f' (0)), 
$
implying condition \ref{condition:baby_stabilizerofapoint}.
By Propositions \ref{prop:L_H_stabApartment} and \ref{prop:L_H_stab_H} we have $\operatorname{Stab}_{\slnf}(f) = \operatorname{Stab}_{\slnf}(f')$, which implies \ref{condition:baby_stabilizerofchart}.
By \Cref{lem:L_H_translation} there exists $a\in A_\ff$ such that $f(x) = a.f(0)$ and $f'(x) = a^{-1} . f'(0)$, or equivalently $f' (\tau(x)) = a.f'(0)$, which is condition \ref{condition:baby_Afw}.

Since $\slnf$ also acts transitively on $\build_S$ and its atlas (\Cref{ex:SymmetricBuilding2}), $\rho=\Id$ is an isomorphism of groups, $L$, $\gamma$ and $\sigma$ are injective, and the inclusions in the conditions \ref{condition:baby_stabilizerofapoint} and \ref{condition:baby_stabilizerofchart} are equalities. We conclude that $\build_L$ and $\build_{S}$ are equivariantly isomorphic by \Cref{thm:baby_morphismofGbuildings}~\ref{thm:isomorphism}.
\end{proof}

\subsection{Symmetric and Bruhat--Tits buildings}
\label{subsection:SymmetricBTBuildings}

In this section we show that there is an injective morphism from the symmetric building $\build_S$ to the Bruhat--Tits building $\build_\BT$. When $\Lambda = \rr$, it is an isomorphism.

Let thus $\ff$ be a real closed field with an order-compatible valuation $v \colon \ff^\times \to \Lambda \subseteq \mathfrak{R}$, where $\mathfrak{R}$ is as in Hahn's embedding \Cref{thm:hahn} (see also \Cref{ex:BruhatTitsBuilding1}), and denote by $\mathcal{O}$ the valuation ring of $v$.
Set $G \coloneqq \mathbf{G}(\ff)$, where $\mathbf{G}$ is a semisimple, connected, self-adjoint, $\ff$-split algebraic group $\mathbf{G} < \operatorname{SL}_n$.
Let $\mathbf{S}$ be a maximal ($\ff$-split) torus consisting of self-adjoint elements. Since $\mathbf{G}$ is $\ff$-split, the root system $\Sigma$ is reduced.
These conditions ensure that both the Bruhat--Tits building $\build_{\BT}$ (\Cref{ex:BruhatTitsBuilding1}) and the symmetric building $\build_S$ (\Cref{ex:SymmetricBuilding1}) associated to $\mathbf{G}$ and $\ff$ are defined.

Let $K=G \cap \operatorname{SO}_n(\mathbb{F})$ and $U_\alpha$ the root groups for $\alpha \in \Phi$.
\begin{lemma}\label{lem:Nor_equiv}
   Let $A_{\ff} \subseteq \operatorname{Cent}_G(\mathbf{S}(\ff)) = \mathbf{S}(\ff)$ be the semialgebraically connected component of the identity of $\mathbf{S}(\mathbb{F})$.
   Then
   \[
   \operatorname{Nor}_{G}(\mathbf{S}(\ff) )= \mathbf{S}(\ff) \cdot \operatorname{Nor}_{K}(\mathbf{S}(\ff)) = A_\ff \cdot \operatorname{Nor}_K(A_\ff).
   \]
\end{lemma}
\begin{proof}
    Since $A_\ff$ is Zariski-dense in $\mathbf{S}(\ff)$ (as $S$ is connected as algebraic group), $\operatorname{Nor}_K(\mathbf{S}(\ff)) = \operatorname{Nor}_{K}(A_{\ff})$, hence the inclusions $\supseteq$ hold.
    For the first inclusion $\subseteq$, we note that any element $g \in N \coloneqq \operatorname{Nor}_{\mathbf{G}(\ff)}(\mathbf{S}(\ff))$ represents a unique $w \in \sphwg$ by the defining property $nU_\alpha n^{-1} = U_{w(\alpha)}$ for all $\alpha \in \Sigma$ \cite[Definition 4.2]{HebertIzquierdoLoisel_LambdaBuildingsAssocToQuasiSplitGroups}.
    We can also find $k \in \operatorname{Nor}_{K}(A_\ff)$ representing $w$ \cite[Proposition 6.2]{appenzeller2024semialgebraic} with the same defining property, so $k^{-1}g$ represents the trivial element of $\sphwg$ and thus $k^{-1}g \in \mathbf{S}(\ff)$ \cite[Notation 4.37]{HebertIzquierdoLoisel_LambdaBuildingsAssocToQuasiSplitGroups}, so $g \in \mathbf{S}(\ff) \cdot \operatorname{Nor}_K(\mathbf{S}(\ff))$. 
    Finally, we use that $A_\ff$ is Zariski-dense in $\mathbf{S}(\ff)$ and \cite[Lemma 5.18]{appenzeller2025buildings} to see $\mathbf{S}(\ff) = \operatorname{Cent}_{G}(\mathbf{S}(\ff)) = \operatorname{Cent}_{G}(A_\ff) = \operatorname{Cent}_{K}(A) \cdot A_\ff \subseteq A_\ff \cdot \operatorname{Nor}_K(A_\ff)$ concluding the proof.
\end{proof}

We now compute the stabilizer of the base point in the Bruhat--Tits building.
The following description holds in our context, and uses the description of the stabilizer of the base point in the symmetric building.

\begin{lemma}\label{lem:BT_stab_compatible} 
    For the base points $[\Id,0] \in \build_{\BT}$ and $o \in \build_S$, we have 
    \[\operatorname{Stab}_G([\Id,0]) = \operatorname{Stab}_{G}(o).\]
    In particular, $\operatorname{Stab}_G([\Id,0])=\mathbf{G}(\mathcal{O})$.
\end{lemma}

\begin{remark}
     In \cite[Corollaire (4.6.7)]{BruhatTits84} and \cite[Proposition 7.48]{HebertIzquierdoLoisel_LambdaBuildingsAssocToQuasiSplitGroups}, it is shown that $\operatorname{Stab}([\Id,0]) = \mathfrak{G}(\mathcal O)$, where $\mathfrak{G}$ is a group $\mathcal O$-scheme defined in terms of some Chevalley basis.
     In our definition $G=\mathbf{G}(\ff)$ is just a group of matrices and the definition $\mathbf{G}(\mathcal O) \coloneqq G \cap \mathcal O^{n\times n}$ avoids any algebraic geometry.
 \end{remark}
 
\begin{proof}[Proof of \Cref{lem:BT_stab_compatible}]
    Let $N\coloneqq \operatorname{Nor}_{G}(\mathbf{S}(\ff))$.
    We first show that
    \[\operatorname{Stab}_G([\Id,0]) \cap N = \operatorname{Stab}_{G}(o) \cap N.\]

    Let $g \in \operatorname{Stab}_G([\Id,0]) \cap N$.
    Since $g\in \operatorname{Stab}_G([\Id,0])$, there exists $n \in N$ with $g^{-1}n \in P_0$ and $\nu(n)(0)=0$.
    Since we assumed $g \in N$, we also have $g^{-1}n \in N \cap P_0 \eqqcolon N_0$ and thus $\nu(g^{-1}n)(0) = 0$ \cite[(7.1.8)]{BruhatTits}\cite[Corollary 5.15]{HebertIzquierdoLoisel_LambdaBuildingsAssocToQuasiSplitGroups}, and thus $\nu(g)(0)=0$.
    Using \Cref{lem:Nor_equiv} to write $g=ak \in A_\ff \cdot \operatorname{Nor}_K(A_\ff)$, we conclude that $\nu(a)=\Id_{V\otimes_{\rr} \mathfrak{R}}$, since $\nu(k)(0)=0$. 
    By the compatibility condition (\ref{eq:BT_compatible}), we get that $v (\chi_\alpha(a)) = 0$ for all $\alpha \in \Sigma$.
    By \cite[Proposition 4.18]{appenzeller2025buildings} this implies that $a \in \operatorname{Stab}_G(o)$.
    Moreover $k\in K \subseteq \operatorname{Stab}_G(o)$ \cite[Theorem 4.19 and Corollary 4.20]{appenzeller2025buildings}, and so $g\in \operatorname{Stab}_G(o)$.

    Assume now that $g\in \operatorname{Stab}_G(o) \cap N$.
    Then we can write $g = ak \in N=A_\ff \cdot \operatorname{Nor}_K(A_\ff)$ with $a \in A_\ff \cap \mathcal O^{n\times n}$ and $k \in \operatorname{Nor}_K(A_\ff)$ \cite[Corollary 4.20]{appenzeller2025buildings}.
    Then by \cite[Proposition 4.18]{appenzeller2025buildings}, $v(\chi_\alpha(a)) =0$ for all $\alpha \in \Sigma$, so $\nu(a)=\Id_{V\otimes_\rr \mathfrak{R}}$ by (\ref{eq:BT_compatible}), in particular $\nu(a)(0)=0$.
    Now we can conclude that $g \in \operatorname{Stab}_G([\Id,0])$, since $a \in N$  and satisfies $g^{-1}a = k^{-1} \in \hat{N}_0 \coloneqq \{ n \in N \colon \nu(n)(0) = 0 \}$ and $\nu(a)(0) = 0$ \cite[(7.4.1)]{BruhatTits}\cite[Definition 6.1]{HebertIzquierdoLoisel_LambdaBuildingsAssocToQuasiSplitGroups}.

    We now prove that $\operatorname{Stab}_G([\Id,0]) = \operatorname{Stab}_{G}(o)$.
    For this we use the subgroups $U_{\alpha,0} \coloneqq \{ u \in U_\alpha \colon \varphi_\alpha(u) \geq 0 \}$, where $\varphi_\alpha$ is the root group valuation defined via the Jacobson--Morozov maps $\operatorname{SL}_2 \to \mathbf{G}$ \cite[(6.1.3.b.2), (6.2.3.b)]{BruhatTits}\cite[Notation 7.22]{HebertIzquierdoLoisel_LambdaBuildingsAssocToQuasiSplitGroups}\footnote{We emphasize that while $U_{\alpha,0}$ is defined differently in \cite{appenzeller2025buildings}, the two concepts agree due to \cite[Lemmas 5.22 and 5.26]{appenzeller2025buildings}.}.
    Recall that $H = \nu^{-1}(\Id_{V\otimes_\rr \mathfrak{R}}) \subseteq N$ and that
    \[
    P_0 = \langle U_{\alpha,0} , \,H   \rangle. 
    \]
    
    If now $g \in \operatorname{Stab}_G([\Id, 0])$, then, by definition, there exists $n\in N$ such that $g^{-1}n \in P_0$ and $\nu(n)(0) = 0$.
    Since $n \in \operatorname{Stab}_G([\Id,0]) \cap N$, we have $n\in \operatorname{Stab}_G(o)$ by the above.
    Any $h\in H \subseteq N$ satisfies $h.[\Id,0] = [\Id,\nu(h)(0)] = [\Id,0]$, and thus $h \in \operatorname{Stab}_G(o)$ by the above.
    By \cite[Lemma 5.22]{appenzeller2025buildings}, we also get that $U_{\alpha,0} \subseteq \operatorname{Stab}_G(o)$, hence $P_0 \in \operatorname{Stab}_G(o)$.
    Thus $g = n(g^{-1}n)^{-1} \in \operatorname{Stab}_G(o)$.

    If we start with $g \in \operatorname{Stab}_G(o)$, then we use \cite[Theorem 5.45]{appenzeller2025buildings} to obtain 
    \[
    g \in \langle N_0 , U_{\alpha,0} \rangle \subseteq P_0,
    \]
    so we can take $n=\Id\in N$ with $g^{-1}n \in P_0$ and $\nu(n)(0) = 0$ to obtain $g \in \operatorname{Stab}_G([\Id, 0])$.

    The stabilizer for the action on $B_S$ was computed in \Cref{prop:homogeneous_stabilizers}.
\end{proof}

\begin{lemma}\label{lem:S_BT_stab_apt}
    For the standard apartments $f_0(\aa) \subseteq \build_S$ and $f_0'(\aa') \subseteq \build_\BT$, the pointwise stabilizers agree
    \[
    \operatorname{Stab}_G(f_0(\aa)) = \operatorname{Stab}_G(f_0'(\aa')).
    \]
    In particular $\operatorname{Stab}_G(f_0'(\aa')) = \operatorname{Cent}_G(\mathbf{S}(\ff)) \cap \mathcal{O}^{n\times n}$.
\end{lemma}
\begin{proof}
    By (\ref{eq:BH_compatible}) at the end of \Cref{subsection:ExamplesNonDiscBuildings}, for all $x\in \aa$ there exists $a\in A_\ff \subseteq \mathbf{S}(\ff)$ such that $f_0(x) = a.f_0(0)$ with $(-v)(\chi_\alpha(a)) = \alpha(x)$ for all $\alpha \in \Sigma$. By (\ref{eq:BT_compatible}) in \Cref{ex:BruhatTitsBuilding1}, then $a.f_0'(0)=[a,0] = [\Id, \tau (x)] = f_0'(\tau(x)) \in \build_{\BT}$.
In particular, using \Cref{lem:BT_stab_compatible}, we have for all $x\in \aa$
\begin{align*}
     \operatorname{Stab}_G(f_0(x))  &= \operatorname{Stab}_{G}(a.o)
    = a \operatorname{Stab}_G(o) a^{-1} 
    = a \operatorname{Stab}_G([\Id,0]) a^{-1} \\
    &= \operatorname{Stab}_G([a,0]) 
    = \operatorname{Stab}_G([\Id,\tau (x)])= \operatorname{Stab}_G(f_0'(\tau (x))).
\end{align*}
Since $\tau(\aa) \subseteq \aa'$, we already have 
$$
\operatorname{Stab}(f_0'(\aa')) \subseteq \operatorname{Stab}_G(f_0'(\tau(\aa))) = \operatorname{Stab}_G(f_0(\aa)).
$$ 
In the setting of Bruhat-Tits ($\Lambda \subseteq \rr$), we can use that $\tau(\aa) \subseteq \aa'$ is dense, so equality holds by \cite[Proposition (7.1.9)]{BruhatTits}, so $A_{f,\Id} = A_{f',\Id}$. For the general case, we compute $\operatorname{Stab}_G(f_0'(\tau(\aa)))$ following \cite{HebertIzquierdoLoisel_LambdaBuildingsAssocToQuasiSplitGroups}. Let $\Omega \coloneqq \tau(\aa) \subseteq \aa'$. Since $\Omega$ contains points $x$ with arbitrarily small $\alpha(x)$ for every $\alpha \in\Sigma$, we have $U_{\alpha,\Omega} = \bigcap_{x\in \Omega} U_{\alpha, -\alpha(x)} = \{\Id\}$ and $U_\Omega = \langle U_{\alpha,\Omega} \colon \alpha \in \Sigma \rangle \{\Id\}$, see \cite[Definition 4.8 (V1), Notation 4.55]{HebertIzquierdoLoisel_LambdaBuildingsAssocToQuasiSplitGroups}. By \cite[Lemma 6.5, Corollary 5.14]{HebertIzquierdoLoisel_LambdaBuildingsAssocToQuasiSplitGroups}, $\operatorname{Stab}_G(f_0'(\Omega)) = \hat{P}_\Omega = U_\Omega (\hat{P}_\Omega \cap N) \subseteq N$, but by the calculation at the beginning of the proof we have for every $g\in \operatorname{Stab}_G(f_0'(\Omega)) \subseteq N$, $\nu(g) = \Id_{\aa'}$ on a basis of $\aa'$ and hence everywhere. Therefore, $\operatorname{Stab}_G(f_0'(\tau(\aa))) \subseteq \ker(\nu) = T_b = \operatorname{Stab}_G(f_0'(\aa'))$ \cite[Notation 4.37 and Corollary 6.9]{HebertIzquierdoLoisel_LambdaBuildingsAssocToQuasiSplitGroups}, whence equality holds.

The last statement follows from \Cref{prop:homogeneous_stabilizers}.
\end{proof}

We can now prove the existence of an injective morphism from $\build_S$ to $\build_\BT$.

\begin{thm}
\label{corol:MorphismKramerTentToBruhatTits}
    The identity homomorphism on $G$ induces an equivariant injective morphism of buildings from the symmetric building $(\build_S,\atlas)$ of type $\mathbb{A}(\Sigma^\vee, \Lambda, \Lambda^n)$ to the Bruhat--Tits building $(\build_\BT,\atlas')$ of type $\mathbb{A}(\Sigma^\vee, \mathfrak{R}, \Lambda^n)$.
    
    If $\Lambda = \mathbb{R}$, then this morphism is an isomorphism.
\end{thm}

\begin{proof} 
We can apply \Cref{thm:baby_morphismofGbuildings}~\ref{thm:baby_monomorphism}, as $G$ acts transitively on $\build_S$ and $\atlas$, see \Cref{ex:SymmetricBuilding2}.
Let $L = \Id_{\operatorname{Span}_{\mathbb{Q}}(\Sigma^\vee)}$, $\gamma \colon \Lambda \to \mathfrak{R}$ be the inclusion, $\sigma = \Id_{\affwg}$ and $\rho = \Id_G$.
For all $w\in \affwg$ and $x \in \aa$ we have $L \otimes \gamma \, (w(x)) = \sigma(w) \left(L \otimes \gamma \, (x) \right)$, and thus $\tau \coloneqq L \otimes \gamma \colon \aa \to \aa'$ is a morphism of apartments, that is injective as $L$ and $\gamma$ are.

We verify the conditions of the theorem for the standard charts.
Namely, we choose $f=f_0 \colon \aa \to \build_S \in \atlas$ as in \Cref{ex:SymmetricBuilding1}.
For $\build_{\BT}$ we choose $f'=f_0' \colon \aa' \to \build_{\BT}$ by $f_0'(x') = [\Id,x']$ as in \Cref{ex:BruhatTitsBuilding1}.
By \Cref{lem:BT_stab_compatible}, we have $\operatorname{Stab}_G(f(0)) = \operatorname{Stab}_{G}(f'(0))$, which implies condition \ref{condition:baby_stabilizerofapoint} of \Cref{thm:baby_morphismofGbuildings}.
By \Cref{lem:S_BT_stab_apt}, we have $\operatorname{Stab}_{G}(f) = \operatorname{Stab}_{G'}(f')$, which implies condition \ref{condition:baby_stabilizerofchart}. 

By (\ref{eq:BH_compatible}) at the end of \Cref{subsection:ExamplesNonDiscBuildings}, for all $x\in \aa$ there exists $a\in A_\ff \subseteq \mathbf{S}(\ff)$ such that $f(x) = a.f(0)$ with $(-v)(\chi_\alpha(a)) = \alpha(x)$ for all $\alpha \in \Sigma$. By (\ref{eq:BT_compatible}) in \Cref{ex:BruhatTitsBuilding1}, then $a.f'(0)=[a,0] = [\Id, \tau (x)] = f'(\tau(x)) \in \build_{\BT}$.
We have shown condition \ref{condition:baby_Afw} and thus obtain an equivariant morphism $m=(\psi,\varphi,\tau) \colon \build_S \to \build_\BT$.

The conditions in \ref{thm:baby_monomorphism} hold since $\tau$ and $\rho$ are injective by definition and by \Cref{lem:S_BT_stab_apt}, $\operatorname{Stab}_G(f)=\operatorname{Stab}_G(f')$, so $m$ is injective.

If $\Lambda = \mathbb{R}$, then $\gamma$ is an isomorphism, so $\tau$ is an isomorphism of apartments, since in this case $\mathfrak{R}=\rr$, see \Cref{thm:hahn}.
By \Cref{ex:BruhatTitsBuilding2}, $G$ also acts transitively on $\build'$ and $\atlas'$ and we have already shown the two equalities, so the conditions of \Cref{thm:baby_morphismofGbuildings}~\ref{thm:baby_isomorphism} hold and $m$ is an equivariant isomorphism $\build_S \xrightarrow{\sim} \build_\BT$.
\end{proof}

\subsection{Bruhat--Tits and norm buildings}
\label{BTtonorms}

In this section, we use \Cref{thm:morphismofGbuildings} to show that there is an equivariant morphism from the Bruhat--Tits building $\build_\BT$ to the norms building $\build_N$.
For the definitions of the buildings we refer to the Examples \ref{ex:NormBuilding1} and \ref{ex:BruhatTitsBuilding1}.

Let now $G \coloneqq \operatorname{GL}_n(\ff)$, where $\ff$ is a field with a rank-$1$ valuation $v \colon \ff^\times \twoheadrightarrow \Lambda\subseteq \rr$, and $\Phi \subseteq V \coloneqq \rr^n/\rr(1,\ldots, 1)$ the standard root system of type $A_{n-1}$.
Then $\aa = \operatorname{Span}_{\qq}(\Phi) \otimes_\qq \rr$ identifies with the subspace $V_0 \coloneqq \{(x_1,\ldots,x_n) \in \rr^n \colon \sum_i x_i =0 \} \subseteq \rr^n$ equipped with the action of the spherical Weyl group action given by the permutations of entries.

 Let 
 \[f \colon \aa \to \build_\BT, \;x \mapsto [\Id, x] \quad \textnormal{ and } \quad
    f'\colon \aa \to \build_N, \; x \mapsto [\eta_x]\]
be the standard charts, where we recall
    \[
    \eta_x \Big( \sum_{i=1}^n v_i e_i \Big)
    = \max \left\{ e^{-x_1} |v_1|, \ldots , e^{-x_n} |v_n| \right\},
    \]
    for $x = (x_1, \ldots , x_n) \in V_0$ and $|a| = \exp(-v(a))$ for $a\in \ff$.

    We begin by recalling the description of stabilizer of points in $f(\aa)$ in the Bruhat--Tits building.\footnote{In \cite{BruhatTits}, $\alpha_{ij}(x)$ is denoted by $a_j(x)-a_i(x)$, see \cite[(10.2.1)]{BruhatTits}.}

\begin{lemma}[{\cite[Corollaire (10.2.9)]{BruhatTits}}]
    \label{lem:BBT_stabGL}
    For $x\in \aa$, we have 
    \[
    \operatorname{Stab}_{\glnf}(f(x))  = \left\{ g\in \glnf \colon 
    v(g_{ij}) \geq \frac{v(\det(g))}{n} - \alpha_{ij}(x) 
    \ \textnormal { for all }i,j \right\},
    \]
    where $g= (g_{ij})_{ij}$ as matrices. 
\end{lemma}

Parreau \cite[Section 3B1]{Parreau_AffineBuildings} gave a description of the stabilizers of points in $f'(\aa)$ in the norm building.
In the case $x=0$ (and its $\glnf$-translates) the stabilizer can be described as $\ff^\times\cdot \operatorname{GL}_n(\mathcal{O})$ (and its conjugates), where
\[\operatorname{GL}_n(\mathcal{O}) \coloneqq \{g \in \glnf : g_{ij} \in \mathcal{O},\ \det(g) \in \mathcal{O}^\times\}.\]
We give an alternative description of the stabilizers that agrees with the one given in \Cref{lem:BBT_stabGL}.
Our computation is based on the one in \cite[Corollary 3.4]{Parreau_AffineBuildings}.

\begin{lemma}\label{lem:BN_stabGL}
    For $x\in \aa$, we have 
    \[
    \operatorname{Stab}_{\glnf}(f'(x))  = \left\{ g\in \glnf \colon 
    \begin{array}{l}
    v(g_{ij}) \geq \frac{v(\det(g))}{n} - \alpha_{ij}(x) \\
    v(g^{ij}) \geq \frac{v\left(\det\left(g^{-1}\right)\right)}{n} - \alpha_{ij}(x) \end{array}  
    \textnormal {for all }i,j \right\},
    \]
    where $g= (g_{ij})_{ij}$ and $g^{-1} = (g^{ij})_{ij}$ as matrices. 
\end{lemma}
Note that we could a posteriori, in view of \Cref{lem:BBT_stabGL}, remove the condition on the valuations of the entries of $g^{-1}$.

\begin{proof}
We denote
\[
\mathcal S \coloneqq  \left\{ g\in \glnf \colon \begin{array}{l} v(g_{ij}) \geq \frac{v(\det(g))}{n} - \alpha_{ij}(x) \\
    v(g^{ij}) \geq \frac{v\left(\det\left(g^{-1}\right)\right)}{n} - \alpha_{ij}(x) \end{array} \text{for all }i,j \right\}
\]
as in the statement of the lemma.
Let $x = (x_0 , \ldots , x_n) \in V_0 \cong \aa$. Let $g\in \glnf$ such that $g.f'(x) = f'(x)$.
This means that there exists some $r \in \rr_{>0}$ such that $g.\eta_x = r\eta_x$, equivalently $g.\eta_{x}(e_j)=r\eta_x(e_j)$ or $r^{-1} \eta_x(e_j) = g^{-1}.\eta_x(e_j) =  \eta_x(ge_j)$ for all $j$. We note that $\eta_x$ is adapted to both the standard basis $\mathcal{E} = \{e_1, \ldots , e_n\}$ and $g\mathcal{E}$, so by \cite[Corollary 3.3]{Parreau_AffineBuildings}, we have
    \[
    \lvert \det(g)\rvert = \frac{\prod_{j=1}^n \eta_x(ge_j)}{\prod_{j=1}^n \eta_x(e_j)} = r^{-n},
    \]
    and thus $\log(r) = v(\det(g))/n$.
    Since for all $j$
    \[
    \eta_x(ge_j) = \eta_x \Big(  \sum_{i=1}^n g_{ij}e_i \Big) = \max_i \left\{ \lvert g_{ij} \rvert e^{-x_i} \right\},
    \]
    we obtain that for all $i,j$
    \[
    \lvert g_{ij} \rvert e^{-x_i} \leq \eta_x(ge_j) = r^{-1} \eta_x(e_j) = r^{-1} e^{-x_j}.
    \]
    Taking the logarithm we obtain
    \[
    -v(g_{ij}) \leq -\log(r) -x_j + x_i = \alpha_{ij}(x) - \frac{v(\det(g))}{n},
    \]
    which implies the first condition in the definition of $\mathcal S$. If $g$ stabilizes $f'(x)$, then so does $g^{-1}$, hence the second condition in the definition of $\mathcal S$ also holds, and thus $g \in \mathcal S$.

    Starting with $g\in \mathcal S$, we can reverse the above arguments, to see that 
    \[
    \eta_x(ge_j) \leq \lvert\det(g)\rvert^{\frac{1}{n}}\eta_x(e_j)
    \]
    for all $j$. Since also $g^{-1} \in \mathcal S$, we have $g.\eta_x(e_j) = \eta_x(g^{-1}e_j) \leq \lvert \det(g)\rvert^{-\frac{1}{n}} \eta_x(e_j)$, so
    \[
    \eta_x(e_j) \leq \lvert \det(g) \rvert^{-\frac{1}{n}} g^{-1}.\eta_x(e_j) = \lvert \det(g) \rvert^{-\frac{1}{n}} \eta_x(ge_j) \leq  \eta_x(e_j)
    \]
    for all $j$, and equality holds.
    In particular, $g.\eta_x = \lvert \det(g) \rvert^{-\frac{1}{n}} \eta_x$, so $g$ stabilizes the equivalence class of $\eta_x$, which is exactly $f'(x)$.
\end{proof}

We summarize the results on stabilizers of (points in) the standard apartments $f$ for $\build_{\BT}$ and $f'$ for $\build_N$.
\begin{cor}
\label{lem:stab_BT_N}
    For $x\in  \aa$, we have 
    $
    \operatorname{Stab}_{\glnf}(f(x))= \operatorname{Stab}_{\glnf}(f'(x)).
    $
    Moreover, for the pointwise stabilizers, we have
    \[
    \operatorname{Stab}_{\glnf}(f(\aa))= \operatorname{Stab}_{\glnf}(f'(\aa)) .
    \]
\end{cor}

\begin{lemma}\label{lem:Stabx_GLO}
    For every $y = (y_1, \ldots, y_n) \in V_0 \cap \Lambda^n \subseteq \aa$, there exists a diagonal matrix $a =\operatorname{Diag}(a_1,\ldots, a_n) \in \operatorname{SL}_n(\ff)$ such that $-v(a_i) = y_i$. 
    
    Moreover, $a.f'(x) = f'(x+y)$ for all $x\in \aa$.
\end{lemma}
\begin{proof}
Let $y = (y_1, \ldots , y_n) \in V_0 \cap \Lambda^n$. Then there exist $a_1, \ldots , a_{n-1} \in \ff^\times$ such that $-v(a_i) = y_i$ for $i \leq n-1$. Let $a_n \coloneqq (a_1\cdots a_{n-1})^{-1}$, then $a = \operatorname{Diag}(a_1, \ldots , a_n)\in \operatorname{SL}_n(\ff)$ and
\[
-v(a_{n}) = - \sum_{i=1}^{n-1} (-v)(a_i) = - \sum_{i=1}^{n-1} y_i = y_n.
\]
Moreover, for all $v = (v_1, \ldots, v_n) \in V_0 \cong \aa$, we have
\begin{align*}
    a.\eta_x(v) &= \eta_x(a^{-1}v) = \max_i \left\{ e^{-x_i} \left\lvert \frac{v_i}{a_i}\right\rvert\right\}
    = \max_i \left\{ e^{-(x_i - v(a_i))} \lvert v_i \rvert \right\} = \eta_{x + y}(v). \qedhere
\end{align*}
\end{proof}

\begin{thm}
\label{thm:BTNorm}
    There is an $\glnf$-equivariant injective morphism of buildings from the Bruhat--Tits building $\build_\BT$ of type $\aa =\aa(\Phi, \rr, \Lambda^n/\Lambda (1,\ldots , 1 ))$ to the norms building $\build_N$ of type $\aa'=\aa(\Phi, \rr,  \rr^n/\rr(1, \ldots , 1)) $.
    
    If $\Lambda = \rr$, this morphism is an isomorphism. 
\end{thm}

The morphism that we will construct in the proof consists of a bijective map $\psi \colon \build_\BT \to \build_N$ on the buildings themselves, but when $\Lambda \neq \rr$, the atlas map $\varphi$ is not bijective. Therefore the morphism is not an isomorphism. 
However, as remarked in \cite[Remark in Section 3B4]{Parreau_AffineBuildings}, $\build_N$ could also be equipped with a different atlas by restricting the translation part of the affine Weyl group, compare with \Cref{ex:completion_translations}. For the so-defined atlas, the morphism we construct is always an isomorphism.

Since the action of $G$ on $\build_{\BT}$ is not transitive, we use \Cref{thm:morphismofGbuildings} instead of \Cref{intro:thm:baby_morphismofGbuildings} in the proof. 

\begin{proof}
    We note that the apartments $\aa$, $\aa'$ are the same as sets by definition.
    However, if $\Lambda\neq \rr$ the affine Weyl group of the Bruhat--Tits building $\affwg = \Lambda^n/\Lambda(1,\ldots,1) \rtimes \sphwg$ includes as a proper subgroup into the affine Weyl group $\affwg' = \mathbb{R}^n/\rr(1, \ldots , 1) \rtimes \sphwg$ of the norms building.
    Together with the identity map $\aa \to \aa'$ this inclusion gives a morphism $\tau$ of apartments. 

    Let $G=\glnf$.
    We apply \Cref{thm:morphismofGbuildings} with $\rho=\Id_G$. 
    Since $G$ acts transitively on the atlas of the Bruhat--Tits building $\build_\BT$, see \Cref{ex:BruhatTitsBuilding2}, it remains to show that the conditions \ref{condition:stabilizerofapoint} and \ref{condition:Afw} in \Cref{thm:morphismofGbuildings} hold.
    
    Condition~\ref{condition:stabilizerofapoint} and condition~\ref{condition:Afw} for $w=\Id$ follow from \Cref{lem:stab_BT_N}.
   If $w = (t^y, w_s ) \in \affwg$ for $y \in V_0 \cap \Lambda^n$, let $a = \operatorname{Diag}(a_1, \ldots  , a_n) \in \operatorname{SL}_n(\ff)$ with $y_i = (-v)(a_i)$ as in \Cref{lem:Stabx_GLO}, and let $k$ be the permutation matrix in $\glnf$ associated to $w_s$.
    Then $a \in A_{f,t^y}$ by (\ref{eq:BT_compatible}), and $a \in  A_{f',t^y}$ by \Cref{lem:Stabx_GLO}.
    We also have $k\in A_{f,w_s} \cap A_{f',w_s}$.
    
    Now for $g \in A_{f,w} = \left\{ h\in \operatorname{GL}_{n}(\ff)  \colon h.f = f \circ w \right\}$ we have $g.f = f\circ w = a.f \circ w_s = ak.f $, so $g^{-1}ak \in \operatorname{Stab}_{\glnf}(f(\aa))$, which by \Cref{lem:stab_BT_N} is equivalent to $g.f' = ak.f' = a.f' \circ w_s = f' \circ (t^y, w_s) = f' \circ w$, so $g \in A_{f', w}$ as required for the condition \ref{condition:Afw}.
    We have shown $A_{f,w} = A_{f',w}$ for all $w\in \affwg$.

    When $\Lambda = \rr$ (meaning $v \colon \ff^\times \to \rr$ is surjective), recall from \Cref{ex:NormBuilding2} that $\glnf$ acts transitively on the atlas $\atlas'$.
    Moreover, $\rho$ is an isomorphism and the inclusions in conditions \ref{condition:stabilizerofapoint} and \ref{condition:Afw} are equalities.
    Thus all the conditions of \ref{thm:isomorphism} are satisfied and $\build_{\BT} \cong \build_N$.
\end{proof}

\section{Functoriality for symmetric buildings}
\label{section:Functoriality}
We would like to apply our notion of morphism of generalized affine buildings and their construction using \Cref{thm:morphismofGbuildings} to prove certain natural functoriality properties.
For this we put ourself in the setting of symmetric buildings (\Cref{ex:SymmetricBuilding1}), but we expect similar results also in the context of the other models of buildings introduced in \Cref{subsection:ExamplesNonDiscBuildings}.

The goal is to prove the existence of morphisms that are induced by homomorphisms of ordered valued fields and monomorphisms of algebraic groups.

\subsection{Functoriality under field morphisms}

We now discuss functoriality for symmetric buildings under field morphisms. This is a generalization of \cite[Corollary 5.19]{BurgerIozziParreauPozzetti_RSCCharacterVarieties} for a special class of rank-$1$ valuations coming from so called \emph{big} elements.
It also exemplifies the $\Lambda$-functoriality in \cite{SchwerStruyve_LambdaBuildingsBaseChangeFunctors}.
At the end of this section, we use two order-compatible valuations on the field of real Puiseux series to construct an explicit example of a surjective morphism between two symmetric buildings associated to the same group, see \Cref{ex:morphism_puiseux}.

Let $\ff$, $\ff'$ be two ordered fields with order-compatible valuations $v \colon \ff^\times \to \Lambda$, $v' \colon (\ff')^\times \to \Lambda'$ and respective valuation rings $\mathcal{O}$,  $\mathcal{O}'$.
A field homomorphism $\eta \colon \ff \to \ff'$ is called a \emph{morphism of ordered valued fields} if it is order-preserving and $\eta(\mathcal{O}) \subseteq \mathcal{O}'$.

\begin{lemma}\label{lem:field_morphism_equiv}
An order-preserving field homomorphism $\eta\colon \ff \to \ff'$ is a morphism of ordered valued fields if and only if there exists an order-preserving group homomorphism $\gamma\colon \Lambda \to \Lambda'$ such that $\gamma \circ v = v' \circ \eta$.
Equivalently, the following diagram commutes
\[\begin{tikzcd}
	{\ff^\times} & {(\ff')^\times} \\
	{\Lambda} & {\Lambda'.}
	\arrow["\eta", from=1-1, to=1-2]
	\arrow["v"', from=1-1, to=2-1]
	\arrow["v'", from=1-2, to=2-2]
	\arrow["\gamma"', from=2-1, to=2-2]
\end{tikzcd}\]
Furthermore, if $\eta$ is surjective, then so is $\gamma$.
\end{lemma}
\begin{proof}
    Let $\eta$ be a morphism of ordered valued fields. For $\lambda \in \Lambda$, let $a\in \ff^\times$ such that $(-v)(a) = \lambda$.
    We then define $\gamma(\lambda) \coloneqq (-v')(\eta(a))$.
    This is well-defined, because if $a,b \in \ff^\times$ satisfying $(-v)(a)=(-v)(b)=\lambda$, then $a/b \in \mathcal{O}$ and $b/a \in \mathcal{O}$, so $(-v')(\eta(a/b)) = 0 $ and thus $(-v')(\eta(a)) = (-v')(\eta(b))$.
    
    The map $\gamma$ is a group homomorphism since for all   for all $a,b\in \ff^{\times}$, we have
    \begin{align*}
                \gamma\left((-v)(a) + (-v)(b) \right) 
                = \gamma((-v)(ab)) & = (-v')(\eta(ab)) \\
                = (-v)(\eta(a)) + (-v)(\eta(b))
                & = \gamma(-v(a)) + \gamma(-v(b)).
    \end{align*}
    We now show that $\gamma$ is order-preserving.
    It suffices to show that if $a \in \ff^\times$ with $0 \leq -v(a)$, then $0 \leq \gamma(a)$.
    The condition $0 \leq -v(a)$ is equivalent to $a \in \mathcal{O}$, and thus $\eta(a) \in \mathcal{O}'$, so $0 \leq -v'(\eta(a))=\gamma(a)$.

    If, on the other hand, we start with an order-compatible group homomorphism $\gamma\colon \Lambda \to \Lambda'$, then for all $a\in \ff^\times$ with $(-v)(a) \leq 0$, we have $(-v')(\eta(a)) = \gamma((-v)(a)) \leq 0$.

    Assume now that $\eta$ is surjective.
    Let $\lambda' \in \Lambda'$, and choose $x' \in \ff'$ with $v'(x')=\lambda'$. 
    Since $\eta$ is surjective, there exists $x \in \ff$ with $\eta(x)=x'$.
    Then $v(x) \in \Lambda$ satisfies
    \[\gamma(v(x))=v'(\eta(x))=v'(x')=\lambda',\]
    so $\gamma$ is surjective.
\end{proof}

Assume from now on that $\ff$, $\ff'$ are real closed fields. Semisimple algebraic groups $\mathbf{G}$ are $\ff$-split if and only if they are $\ff'$-split \cite[Theorem 5.17]{appenzeller2024semialgebraic}.
This enables to use the same split torus $\mathbf{S}$ in the construction of the symmetric buildings for $\mathbf{G}(\ff)$ and $\mathbf{G}(\ff')$.
We will later see that $\build_S$ is independent of the choice of $\mathbf{S}$ (\Cref{prop:isomorphism_of_group_implies_isomorphism_of_buildings}).

\begin{thm}[{\Cref{intro:corol_FieldExt}}]
\label{thm:corol_FieldExt}
    Let $\ff, \ff'$ be real closed valued fields. Let $\mathbf{G} < \mathrm{SL}_n$ be a connected semisimple self-adjoint linear algebraic $\mathbb{Q}$-group with reduced root system, $\mathbf{S}$ a maximal $\ff$-split (and hence $\ff'$-split) torus all of whose elements are self-adjoint, and $\build$ and $\build'$ the symmetric buildings associated to $\mathbf{G}(\ff)$ and $\mathbf{G}(\ff')$ respectively.
    Let $\eta \colon \ff \rightarrow \ff'$ a morphism of ordered valued fields and $\rho \colon \mathbf{G}(\ff) \to \mathbf{G}(\ff')$ the induced group homomorphism.
    Then there exists an equivariant morphism of buildings $m \colon \build \rightarrow \build'$.
    
    Additionally, if $\gamma \colon \Lambda \to \Lambda'$ denotes the morphism of the value groups induced by $\eta$, then
    \begin{enumerate}[label=(\alph*)]
        \item \label{item: injective eta and gamma gives injective morphism} if $\gamma$ is injective, then $m$ is injective.
        \item \label{item: surjective eta and gamma gives surjective morphism} if $\eta$ is surjective, then $m$ is surjective.
        \item \label{item: bijective eta and gamma gives bijective morphism} if $\eta$ and $\gamma$ are isomorphisms, then $m$ is an isomorphism.
    \end{enumerate}
\end{thm}

\begin{proof}
Denote by $v \colon \ff^\times \to \Lambda$, $v'\colon (\ff')^\times \to \Lambda'$ the valuations.
    Recall that the buildings $\build$ and $\build'$ are of type $\aa = \aa(\Sigma^\vee, \Lambda, T)$ and $\aa' = \aa(\Sigma^\vee, \Lambda', T')$ with full translation groups $T = \aa$ and $T' = \aa'$, where $\Sigma^\vee$ is the coroot system of $\mathbf{G}(\mathbb{R})$ with basis $\Delta$.
    Let $\gamma \colon \Lambda \to \Lambda'$ be the morphism of ordered abelian groups as in \Cref{lem:field_morphism_equiv}.
    Taking $L = \Id_{\operatorname{Span}_{\mathbb{Q}}(\Sigma^\vee)}$, we obtain a $\mathbb{Q}$-linear map $ \tau  = \Id \otimes \gamma \colon \aa \to \aa'$.
    The identity $\sigma_s = \Id_{\sphwg} \colon \sphwg \to \sphwg$ on the spherical Weyl groups, allows us to define a group homomorphism of the full affine Weyl groups $\sigma \colon  \affwg \to \affwg$ such that $\tau=(\Id,\gamma,\sigma)$ is a morphism of apartments, see \Cref{lem:morphism_apartment}.

    The morphism $\eta \colon \ff \to \ff'$ defines a group homomorphism $\rho \colon \mathbf{G}(\ff) \to \mathbf{G}(\ff')$ defined by $\rho((a_{ij})_{i,j}) = (\eta(a_{ij}))_{i,j}$, since polynomial equations are preserved by $\eta$.
    This extension is compatible with subgroups and algebraic morphisms: the characters $\chi_\alpha \colon \mathbf{S}(\ff) \to \ff^\times$, $\chi_\alpha' \colon \mathbf{S}(\ff') \to (\ff')^\times$ for $\alpha \in \Sigma$ satisfy for instance $\chi_\alpha'(\rho(a)) = \eta(\chi_\alpha(a))$ for all $a \in \mathbf{S}(\ff)$, i.e.\ 
    \[\begin{tikzcd}
	{\mathbf{S}(\ff)} & {\mathbf{S}(\ff')} \\
	{\ff^{\times}} & {(\ff')^{\times}.}
	\arrow["\rho", from=1-1, to=1-2]
	\arrow["\chi_\alpha"', from=1-1, to=2-1]
	\arrow["\chi_\alpha'", from=1-2, to=2-2]
	\arrow["\eta"', from=2-1, to=2-2]
\end{tikzcd}\]

We apply \Cref{thm:baby_morphismofGbuildings} to the group homomorphism $\rho \colon \mathbf{G}(\ff) \to \mathbf{G}(\ff')$ and the morphism of apartments $\tau \colon \aa \to \aa'$ constructed above.
We know that $\mathbf{G}(\ff)$ acts transitively on $\build$ and $\atlas$, and we consider the standard charts $f_0 \colon \aa \to \build$, $f'_0 \colon \aa' \to \build'$ as in \Cref{ex:SymmetricBuilding1}.

Condition \ref{condition:baby_stabilizerofapoint} is satisfied since $\rho(\operatorname{Stab}_{\mathbf{G}(\ff)}(f_0(0))) =\rho( \mathbf{G}(\mathcal O) ) \subseteq \mathbf{G}(\mathcal O') = \operatorname{Stab}_{\mathbf{G}(\ff')}(f_0'(0))$ using \Cref{prop:homogeneous_stabilizers} and $\eta(\mathcal O) \subseteq \mathcal O'$.

By \Cref{prop:homogeneous_stabilizers}, the pointwise stabilizer of $f_{0}(\aa)$ is $A_\ff(\mathcal{O}) M_\ff = A_{f_0,\operatorname{Id}}$ and the stablizer of $f_0'(\aa)$ is $A_{\ff'}(\mathcal{ O}')M_{\ff'}$. Applying $\rho$ gives
\begin{equation}\label{Eqn: rho(A_{f_0,Id})}\tag{I}
    \rho(\operatorname{Stab}_{\mathbf{G}(\ff)}(f_0)) = \rho(A_\ff(\mathcal{O})) \rho(M_\ff) \subseteq A_{\ff'}(\mathcal{O}') M_{\ff'} = \operatorname{Stab}_{\mathbf{G}(\ff')}(f_0'),
\end{equation}
which is condition \ref{condition:baby_stabilizerofchart}.

For $x \in \aa$, let $a\in A_\ff$ such that $a.f_0(0) = f_0(x)$.
By the compatibility condition (\ref{eq:BH_compatible}), we have $-v(\chi_\alpha (a)) = \alpha(x)$.
Now
\begin{align*}
    -v(\chi_\alpha'(\rho(a))) &= -v(\eta(\chi_\alpha(a))) 
    = \gamma(-v(\chi_\alpha(a))) 
    = \gamma(\alpha(x)) 
    = \alpha(\tau(x))
\end{align*}
for all $\alpha \in \Sigma$, so by (\ref{eq:BH_compatible}), $\rho(a).f_0'(0) = f_0'(\tau(x))$ which shows \ref{condition:baby_Afw}.
We obtain an equivariant morphism of buildings $m \colon \build \to \build'$.

We now analyze when $m$ is injective, surjective, or an isomorphism, depending on $\eta$ and $\gamma$. 

To prove \ref{item: injective eta and gamma gives injective morphism}, suppose that $\gamma \colon \Lambda \rightarrow \Lambda'$ is injective.
Since any field morphism is injective, $\rho \colon \mathbf{G}(\ff) \rightarrow \mathbf{G}(\ff')$ is injective.
Now $L = \Id_{\operatorname{Span}_{\qq}(\Sigma^\vee)}$ is injective, and thus, by \ref{thm:baby_monomorphism} of \Cref{thm:baby_morphismofGbuildings}, it remains to check that 
\[
    \rho(\mathrm{Stab}_{\mathbf{G}(\ff)}(f_0))=\mathrm{Stab}_{\mathbf{G}(\ff')}(f_0').
\]
This equality holds precisely when the inclusion in \eqref{Eqn: rho(A_{f_0,Id})} is an equality.
Since $\gamma$ is injective, for every $a\in \ff$, it holds $(-v')(\eta(a)) = \gamma ( (-v)( a) ) = 0$ if and only if $(-v)(a) = 0$. In other words 
\[
    \eta(a) \in \mathcal{O}' \text{ if and only if } a\in \mathcal{O}.
\]
Hence $\eta(\mathcal{O}) = \mathcal{O}'$, and the inclusion in \eqref{Eqn: rho(A_{f_0,Id})} is indeed an equality.
Therefore $m$ is injective.

To prove \ref{item: surjective eta and gamma gives surjective morphism}, suppose that $\eta$ is surjective.
Then $\rho$ is surjective, and by the last part of \Cref{lem:field_morphism_equiv}, the map $\gamma$ is also surjective.
Since $L = \Id_{\operatorname{Span}_{\qq}(\Sigma^\vee)}$, by \Cref{thm:baby_morphismofGbuildings}~\ref{thm:baby_epimorphism}, the morphism $m$ is surjective provided that $\mathbf{G}(\ff')$ acts transitively on $\build'$ and $\atlas'$, which is the case for symmetric buildings, see \Cref{ex:SymmetricBuilding2}. Hence $m$ is surjective.

To prove \ref{item: bijective eta and gamma gives bijective morphism}, assume $\gamma$ and $\eta$ are isomorphisms.
We know that $\mathbf{G}(\ff')$ acts transitively on $\build'$ and $\atlas'$.
Then $\rho$ is a group isomorphism, $L = \Id_{\operatorname{Span}_{\qq}(\Sigma^\vee)}$ is a vector space isomorphism, $\sigma$ is an isomorphism of affine Weyl groups, and thus $\tau$ is an isomorphism of apartments. 
By \Cref{thm:baby_morphismofGbuildings}~\ref{thm:baby_isomorphism}, it remains to verify that the inclusions in \ref{condition:baby_stabilizerofapoint} and \ref{condition:baby_stabilizerofchart} are equalities.
The equality in \ref{condition:baby_stabilizerofapoint} is due to the injectivity of $\gamma$, since then $\eta(\mathcal{O}) = \mathcal{O'}$ (as in the proof of \ref{thm:baby_monomorphism} of this theorem) and $\mathbf{G}(\mathcal O) = \mathbf{G}(\mathcal O')$.
Equality in \ref{condition:baby_stabilizerofchart} follows similarly from the fact that $\rho(A_\ff(\mathcal O)) = A_{\ff'}(\mathcal O ')$ and \eqref{Eqn: rho(A_{f_0,Id})}.
Therefore, by \Cref{thm:baby_morphismofGbuildings}~\ref{thm:baby_isomorphism}, we conclude that $m$ is an isomorphism.
\end{proof}

This allows us to construct more examples of surjective morphisms between buildings.

\begin{example}\label{ex:morphism_puiseux}
    Let $\ff$ be the real closure of $\rr(X,Y)$ endowed with the following order: $f \in \rr(X,Y)$ is positive if $f(e^t,t)$ is positive for $t \in \rr$ large enough, see \Cref{example:OrderedFields}.
    With this order we have that $Y>r$ for all $r \in \rr$, and $X>Y^n$ for every $n \in \nn$.
    We have the order valuation $v=v_\textnormal{lex} \colon \ff^\times \to \Lambda$ as defined in \Cref{ex:OrderValuation}.
    One can check that $-v_\textnormal{lex}$ records for a polynomial in $\rr(X,Y)$ its multi-degree.
    Note that $\Lambda  \cong \qq^2$ is an ordered subgroup of $\rr^2$ endowed with the lexicographic order given by 
    \[ (a,b)\leq \left(a',b'\right) \textnormal{ if and only if } a<a' \textnormal{ or }(a=a' \textnormal{ and } b\leq b').\]
    We also see that for the order on $\ff$, the element $X$ is a big element and we denote by $v'=v_X \colon \ff^\times \to \Lambda' \subset \rr$ the associated valuation, see \Cref{subsection:OrderedFields}.
    Then $-v_X$ records for a polynomial only its degree in $X$, so $\Lambda' \cong \qq$.
    
    The identity is a morphism of ordered valued fields from $(\ff,v_\textnormal{lex}) \to (\ff,v_X)$ since the group homomorphism $\pr_1 \colon\qq^2 \to \qq$ given by projection on the first factor is order-preserving and $\pr_1((-v_\textnormal{lex})(f))=(-v_X)(f)$, see \Cref{lem:field_morphism_equiv}.  
    \[\begin{tikzcd}
	{\ff^\times} & {\ff^\times} \\
	{\qq^2} & {\qq}
	\arrow["\Id", from=1-1, to=1-2]
	\arrow["-v_\textnormal{lex}"', from=1-1, to=2-1]
	\arrow["-v_X", from=1-2, to=2-2]
	\arrow["\pr_1"', from=2-1, to=2-2]
\end{tikzcd}\]
    Furthermore, the identity is surjective.
    By \Cref{thm:corol_FieldExt}~\ref{item: surjective eta and gamma gives surjective morphism}, given any semisimple connected self-adjoint linear algebraic $\qq$-group with reduced root system $\mathbf{G}<\mathrm{SL}_n$, there exists a surjective 
    morphism $\build \to \build'$ between the associated symmetric buildings of types $\aa(\Sigma^\vee,\qq^2,(\qq^2)^n)$ respectively $\aa(\Sigma^\vee,\qq,\qq^n)$, that is induced by the identity $(\ff,v_\textnormal{lex}) \to (\ff,v_X)$.
\end{example}

\subsection{Functoriality under injective group morphisms}
We now turn to the functoriality statements under monomorphisms of algebraic groups (in the sense of \cite[Definition 5.11]{Milne}) for symmetric buildings.
We first need some preliminary considerations about root systems and Weyl groups.
We then prove the case of subgroups and finally discuss the general case.

\subsubsection{Extending Weyl group actions}
\label{subsection:inclusion_Weyl_group}

The goal of this subsection is the proof of \Cref{prop:subroot}, where under certain conditions we can extend a Weyl group action on a vector subspace to the ambient vector space.
The notion of regularity plays an important role in the proof of \Cref{prop:subroot}.
Let $V$ be vector space and $(\Sigma,V^\star)$ a root system.
Recall that for $\alpha \in \Sigma$
\[M_{\alpha} = \{x \in V \colon \alpha(x) =0 \}\]
is called the wall or hyperplane associated to $\alpha$.

\begin{defn}\label{def:regular} 
    A point in $V$ is \emph{regular in $V$} if it does not lie on a hyperplane of $V$ with respect to the root system $\Sigma$.
    If $S \subseteq V$ is a subset, a point in $S$ is \emph{$S$-regular in $V$} if it lies on the smallest amount of hyperplanes possible for a point in $S$.
\end{defn}

This means that every $S$-regular point in $V$ is ``most'' regular in $V$ among all points in $S$.

\begin{defn}
    For a subset $S \subseteq V$ we write
\[ \Sigma_S \coloneqq  \{ \alpha \in \Sigma \colon \alpha(S) = 0 \} \subseteq \Sigma.\] 
If $S$ consists of a single element $S = \{p\}$, we abbreviate $\Sigma_p \coloneqq \Sigma_{\{p\}}$.
\end{defn}

Note that $p$ is regular in $V$ if and only if $\Sigma_p = \emptyset$.
We also always have that if $T \subseteq S \subseteq V$ then $\Sigma_S \subseteq \Sigma_T$.
We will mostly apply the above notions when the subset $S$ is a vector subspace.

Let thus $V'$ be a vector space, and $V < V'$ a vector subspace.
Consider the respective duals $V^\star$, $(V')^{\star}$ of $V$, $V'$.
Let $(\Sigma,V^\star)$, $(\Sigma',(V')^\star)$ be root systems and $W$ and $W'$ denote their respective (spherical) Weyl groups.
We think of $W$ and $W'$ as acting on $V$ and $V'$ respectively.
With this we now have the following characterization of $V$-regular elements in $V'$.

\begin{lemma}\label{lem:regular}
    The following are equivalent for $p\in V$:
    \begin{enumerate}
        \item \label{item: p regular} $p$ is $V$-regular in $V'$,
        \item \label{item: sigma} for every $q \in V$, $|\Sigma_p'| \leq |\Sigma_q'|$,
        \item \label{item: complement}$p \in  V \setminus \bigcup_{\alpha' \in \Sigma' \setminus \Sigma_V'} M_{\alpha'}$,
        \item \label{item: sigma V}$\Sigma_p' = \Sigma_V'$.
    \end{enumerate}
\end{lemma}
\begin{proof}
(\ref{item: p regular}) $\iff$ (\ref{item: sigma}). This is the definition.

(\ref{item: complement}) $\Rightarrow$ (\ref{item: sigma V}). Consider $p\in V \setminus \bigcup_{\alpha' \in \Sigma' \setminus \Sigma_V'} M_{\alpha'}$. Since $p\in V$, $\Sigma_V' \subseteq \Sigma_p'$. If there was some $\alpha' \in \Sigma_p' \setminus \Sigma_V'$, then $\alpha' \in \Sigma' \setminus \Sigma_V'$ and by (\ref{item: complement}), $\alpha'(p)\neq 0$, which contradicts $\alpha' \in \Sigma_p'$.

(\ref{item: sigma V}) $\Rightarrow$ (\ref{item: sigma}). For every $q \in V$ we have $\Sigma_V' \subseteq \Sigma_q'$, so
$
|\Sigma_p'| = |\Sigma_V'| \leq |\Sigma_q'|.
$

(\ref{item: sigma}) $\Rightarrow$ (\ref{item: complement}). We note that the set in (\ref{item: complement}) is non-empty, since the walls $M_{\alpha'}$ are lower-dimensional and $\Sigma'$ is finite.
By the implication ``(\ref{item: complement}) $\Rightarrow$ (\ref{item: sigma V})", all elements $r \in V \setminus \bigcup_{\alpha' \in \Sigma' \setminus \Sigma_V'} M_{\alpha'} $ satisfy $\Sigma_r' = \Sigma_V'$. By assumption, $p$ verifies $|\Sigma_p'| \leq |\Sigma_r'|$, and since $\Sigma_r' = \Sigma_V' \subseteq \Sigma_p'$, we have $\Sigma_p' = \Sigma_V'$. Therefore $p \notin M_{\alpha'}$ for all $\alpha' \in \Sigma' \setminus \Sigma_V' = \Sigma' \setminus \Sigma_p'$.
\end{proof}

Similar as the openness of regular elements, it follows from \Cref{lem:regular}~(\ref{item: complement}), that the set of elements that are $V$-regular in $V'$ is open in $V$ as it is the complement of finitely many lower dimensional subspaces.

\begin{prop}\label{prop:subroot}
   Let $C_0$ and $C_0'$ be fundamental Weyl chambers in $V$ and $V'$ respectively with a common point in $ C_0 \cap C_0'$ that is $V$-regular in $V'$.
   Suppose that
    \begin{align}
        \tag{$\triangle$} \label{condition: subroot_cells} 
        \forall \, w \in W, w' \in W', x \in C_0 \cap C_0' \text{, if } w(x) \in w'(C_0') \text{, then } w(x)=w'(x).
    \end{align}
    Then there exists an injective group homomorphism $\sigma \colon W \to W'$ such that $\sigma(w)|_{V} = w$ for all $w\in W$.
\end{prop}

Note that condition (\ref{condition: subroot_cells}) implies that when $V=V'$ also $\Sigma = \Sigma'$.
The following example illustrates this condition in dimension two.

\begin{example}
    Assume $V'$ is two-dimensional and $\Sigma'$ is a root system of type $A_2$.
    For every one-dimensional subspace $V < V'$, there exist root systems of type $A_1$ in $V^\star$.
    Note that only when $V$ is orthogonal to a hyperplane of $\Sigma'$ is condition (\ref{condition: subroot_cells}) satisfied, see \Cref{fig:A1_in_A2} (left).
    This is for example the case for the root systems associated to $\operatorname{SL}_2 < \operatorname{SL}_3$.
    More examples of root systems where \Cref{prop:subroot} applies can be found in Figures~\ref{fig:intro:Sp4inSL4}, \ref{fig:A2_in_G2}, \ref{fig:A2_in_B3} and \ref{fig:A1xA1}.

    \begin{figure}[h]
        \centering
        \includegraphics[width=\linewidth]{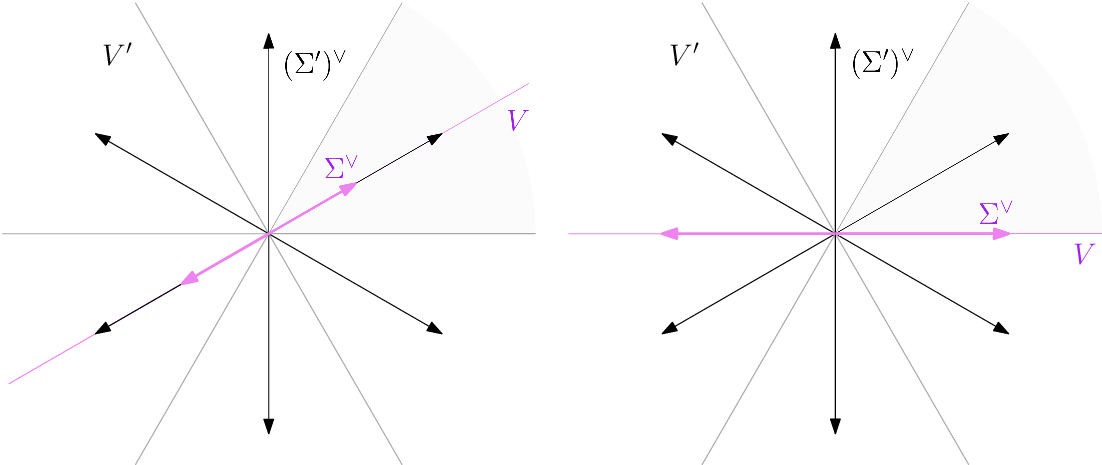} 
        \caption{The purple and black arrows depict the dual roots of $\Sigma$ respectively $\Sigma'$ of type $A_1$ respectively $A_2$.
        The associated Weyl groups $W$ and $W'$ satisfy condition (\ref{condition: subroot_cells}) on the left, but not on the right.
        }
        \label{fig:A1_in_A2}
    \end{figure}
\end{example}

Before we give the proof of \Cref{prop:subroot}, we first indicate how to construct $\sigma \colon W \to W'$ in the special case when there exists some $p \in C_0 \cap C_0'$ that is regular in $V$ and in $V'$, as for example in \Cref{fig:A2_in_B3}.
\begin{figure}[h]
  \centering
  \includegraphics[width=0.5\linewidth]{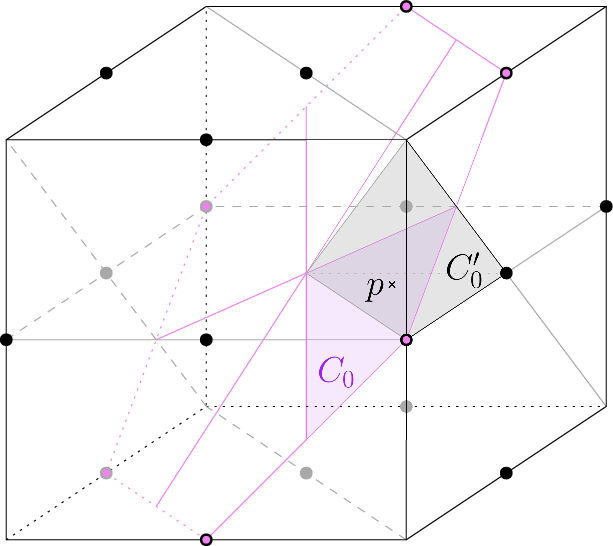} 
  \caption{A coroot system of type $A_2$ lying in type $B_3$ is illustrated by purple and black dots. In this situation there is a point $p \in C_0 \cap C_0'$ that is regular in $V'$.}
  \label{fig:A2_in_B3}
\end{figure}
In this case, for $w\in W$, the point $w(p)$ lies in a unique chamber  $w'(C_0')$ for some $w' \in W'$.
By condition (\ref{condition: subroot_cells}) we have $w'(p) = w(p)$, and since $W$ preserves regularity, $w(p)$ is also regular in $V'$. Since $W'$ acts simply transitively on chambers in $V'$, we conclude that $w'(C_0')$ is the only chamber of $V'$ that contains $w(p)$ and thus $w'$ is the unique element of $W'$ with $w'(p) = w(p)$. 
It is then straightforward to check that $\sigma(w)\coloneqq w'$ defines an injective group homomorphism satisfying $\sigma(w)|_V=w$.

We note that in general, there may not be a $p\in V$ that is regular in $V'$, see \Cref{fig:A1xA1}.
However, there always exists a point $p\in V$ that is $V$-regular in $V'$. Now, there may be multiple chambers in $V'$ containing a non-regular point $p$ respectively $w(p)$, but the following lemma provides a canonical chamber that allows to prove \Cref{prop:subroot} in full generality.

\begin{figure}[h]
  \centering
  \includegraphics[width=0.4\linewidth]{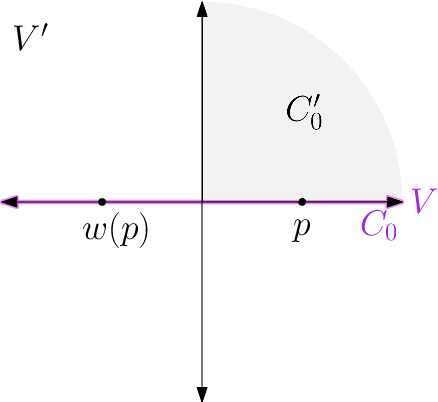} 
  \caption{A coroot system of type $A_1$ lying in type $A_1 \times A_1$ is illustrated in purple and black. The point $p\in C_0$ is not regular in $V'$, but $p$ is $V$-regular in $V'$. There are two chambers containing $w(p)$ and thus two choices for $\sigma(w) \in W'$.
  }
  \label{fig:A1xA1}
\end{figure}

\begin{lemma}\label{lem:regular_unique_chamber}
    Given vector spaces $V<V'$, root systems $(\Sigma,V^\star)$, $(\Sigma',(V')^\star)$ with (spherical) Weyl groups $W$, $W'$ and fundamental Weyl chambers $C_0$, $C_0'$ in $V$, $V'$ with a point $p\in C_0 \cap C_0'$ that is $V$-regular in $V'$. 
    There is a subset $F\subseteq V'$ such that for every $w\in W$ where $w(p)$ is $V$-regular in $V'$,
    \[
    C_w' \coloneqq  F \     \cap \!\!
    \bigcap_{\substack{ \alpha' \in \Sigma' \\ \alpha' (w(p)) >0}} H_{\alpha'}^+
    \]
    is the unique chamber of $V'$ such that $w(p) \in C_w' \subseteq F$.
    Moreover, if $w' \in W'$ satisfies $w'(C_0') = C_w'$ and $w'|_V = w$, then $w'(F) = F$. 
\end{lemma}
\begin{proof}
Let $\Sigma_{>0}'$ be the set of positive roots in $\Sigma'$ with respect to $C_0'$, $H_{\alpha'}^+ \coloneqq \{v' \in V' \colon \alpha'(v') \geq 0\}$ the positive halfspace with respect to $\alpha'\in \Sigma'$ and set
\[
F \coloneqq \bigcap_{\alpha' \in \Sigma_{>0}' \cap \Sigma_V'} H_{\alpha'}^+.
\]
We claim that this $F$ is the desired subset.
Let thus $w \in W$ such that $w(p)$ is $V$-regular in $V'$.
By \Cref{lem:regular}~(\ref{item: sigma V}) and the assumption that $p$ and $w(p)$ are $V$-regular in $V'$, we have $\Sigma_p' = \Sigma_V' = \Sigma_{w(p)}'$, so
\[
\Sigma_{>0}' = (\Sigma_{>0}'\cap \Sigma_V') \cup \{\alpha' \in \Sigma' \colon \alpha'(p)>0 \},
\]
and the fundamental Weyl chamber $C_0'$ can be described as
\[
    C_0' = \bigcap_{\alpha' \in \Sigma_{>0}'}  H_{\alpha'}^+
    = F \ \cap \!\!
    \bigcap_{\substack{ \alpha' \in \Sigma' \\ \alpha' (p) >0}} H_{\alpha'}^+.
\]
In fact, among the chambers containig $p$, $C_0'$ is the unique Weyl chamber such that $C_0' \subseteq F$. 
Now we claim that
\[
C_w' \coloneqq   F \     \cap \!\!
    \bigcap_{\substack{ \alpha' \in \Sigma' \\ \alpha' (w(p)) >0}} H_{\alpha'}^+
\]
is the unique chamber in $V'$ such that $w(p) \in C_w' \subseteq F$. 
We first convince ourselves that $C_w'$ is indeed a chamber of $V'$.
For this, recall that chambers are (closures of) connected components of $V' \setminus \bigcup_{\alpha' \in \Sigma'} M_{\alpha'}$, where the $M_{\alpha'} =  \{ v' \in V' \colon \alpha'(v') = 0\}$ are the walls.
For each wall $M_{\alpha'}$ with $\alpha' \in \Sigma' = \Sigma_V' \cup \{\alpha' \in \Sigma' \colon \alpha'(w(p)) \neq 0 \}$, the definition of $C_w'$ determines on which side $C_w'$ lies, so since $C_w'$ is connected (even convex), $C_w'$ is contained in a chamber.
To make sure it is not smaller than a chamber, it suffices to find a point $q'\in C_w'$ that is regular in $V'$: Let $v'\in C_0'$ be regular in $V'$ and close to $0$, so that $|\alpha'(v')| <\alpha'(w(p))$ for all $\alpha' \in \Sigma' \setminus \Sigma_{V}'$. Then $q' \coloneqq w(p) + v'$ is contained in $C_w'$ and regular in $V'$ because for $\alpha' \in \Sigma_{>0}' \cap \Sigma_V'$ we have
\begin{align*}
    \alpha'(q) &= \alpha'(w(p)) + \alpha'(v) = \alpha'(v) > 0
\end{align*}
and for $\alpha'$ with $\alpha'(w(p)) >0$ we have
\begin{align*}
    \alpha'(q) &= \alpha'(w(p)) + \alpha'(v') > 0
\end{align*}
since $v'$ was chosen close to $0$.
Hence $C_w'$ is a chamber.

Since $\Sigma_{w(p)}' = \Sigma_{V}'$ and $w(p) \in C_w'$, any chamber $C'$ that satisfies $w(p) \in C' \subseteq F$, also satisfies
\[
C' \subseteq \bigcap_{\substack{ \alpha' \in \Sigma' \\ \alpha' (w(p)) >0}} H_{\alpha'}^+,
\]
and is thus contained in $C_w'$, so $C_w'$ is unique.

We now prove the second statement of the lemma.
Recall that $W'$ acts on $\Sigma'$ by $w'(\alpha') \coloneqq \alpha' \circ (w')^{-1}$ for $w'\in W'$ and $\alpha' \in \Sigma'$. 
We remark that 
\begin{align*}
    w'(M_{\alpha'})&= 
    M_{w'(\alpha')}
\end{align*}
for all $w' \in W'$ and $\alpha' \in \Sigma'$. 
If now $w' \in W'$ such that $w'(C_0') = C_w'$ and $w'|_V = w$, we have for all $\alpha' \in \Sigma_V'$ and $v\in V$
\[
    (w'(\alpha'))(v) = \alpha'(w'^{-1}(v)) = \alpha'(w^{-1}(v)) =0,
\]
so $w'(\alpha') \in \Sigma_V'$. We conclude that the hyperplane arrangement $\mathcal{H} \coloneqq \{ M_{\alpha'} \colon \alpha' \in \Sigma_V'\}$ is invariant under $w'$.
Thus $w'$ permutes the connected components of the complement of $\mathcal{H}$.
One such (closure of a) connected component is $F$ for which we know that $C_0' \subseteq F$ and $C_w' \subseteq F$. 
Since $w'(C_0')=C_w'$ and $w'$ preserves connected components, we have $w'(F)= F$.
\end{proof}

We are now in the position to prove \Cref{prop:subroot}.

\begin{proof}[Proof of \Cref{prop:subroot}]
Let $p \in C_0 \cap C_0'$ be $V$-regular in $V'$.  Without loss of generality we may assume that $p$ is $V$-regular: 
by \Cref{lem:regular}~(\ref{item: complement}), the set of points in $V$ that are $V$-regular in $V'$ is given by $\{q \in V \colon \alpha'(q) \neq 0 \text{ for all } \alpha' \in \Sigma' \setminus \Sigma_V' \}$, so if $q\in C_0 \cap C_0'$ is $V$-regular in $V'$, then there exists a convex open neighborhood $U$ of $q$ in $V$, consisting of points that are $V$-regular in $V'$ and contained in $C_0'$.
The intersection $U \cap C_0$ has non-empty interior, because $C_0$ is convex and has non-empty interior, and any point $p$ of the interior is contained in $C_0\cap C_0'$, $V$-regular in $V'$, and regular in $V'$. 

We first prove that for all $w\in W$, $w(p)$ is also $V$-regular in $V'$. 
Recall that $W'$ acts on $\Sigma'$ by $w' (\alpha') = \alpha' \circ (w')^{-1}$ for all $w'\in W'$, $\alpha' \in \Sigma'$. Pick some $w' \in W'$ such that $w(p) \in w'(C_0')$. By Condition (\ref{condition: subroot_cells}), $w(p)=w'(p)$ and so
\begin{align*}
    \Sigma_{w(p)}' &=  \{\alpha' \in \Sigma' \colon \alpha'(w'(p)) = 0 \} 
    = \{ \alpha' \in \Sigma' \colon (w')^{-1}(\alpha')(p) = 0 \} \\
    &= w'\left( \left\{ \alpha' \in \Sigma' \colon \alpha'(p) = 0 \right\} \right)
    = w'(\Sigma_p') = w'(\Sigma_V'),
\end{align*}
where we used \Cref{lem:regular} (4) for the last equality.
Since we always have $\Sigma_V' \subseteq \Sigma_{w(p)}'$, and $w'$ is a permutation of the finite set $\Sigma'$, we have $\Sigma_V' = w'(\Sigma_V') = \Sigma_{w(p)}'$, which means that $w(p)$ is $V$-regular in $V'$ by \Cref{lem:regular}.

Now, we satisfy the conditions of \Cref{lem:regular_unique_chamber} to obtain a subset $F \subseteq V'$ such that for each $w\in W$,
\[
C_w' \coloneqq   F \     \cap \!\!
    \bigcap_{\substack{ \alpha' \in \Sigma' \\ \alpha' (w(p)) >0}} H_{\alpha'}^+.
\]
is the unique chamber in $V'$ such that $ w(p) \in C_w' \subseteq F$.
Since $W'$ acts simply-transitively on Weyl chambers, there exists a unique $w' \in W'$ with $w'(C_0') = C_w'$, and we define
    \[
        \sigma \colon W \to W', \quad \sigma(w)\coloneqq w'.
    \]
Before we prove that $\sigma$ is an injective group homomorphism, let us verify that $\sigma(w)|_{V}=w$ for all $w \in W$, which then implies $\sigma(w)(F) = F$ for all $w \in W$ by the last statement of \Cref{lem:regular_unique_chamber}.

Since $w(p) \in C_w' = w'(C_0')$, condition (\ref{condition: subroot_cells}) implies that $w(p)=w'(p)=\sigma(w)(p)$. By \Cref{lem:regular} (3), the set of elements that are $V$-regular in $V'$ is open, so for $x \in V$ in a small neighborhood around $p$, the element $x$ is also $V$-regular in $V'$.
Thus $\alpha'(x) >0$ if and only if $\alpha'(p)>0$ for all $\alpha' \in \Sigma'$.
Since $w$ is continuous and $w(p)$ is $V$-regular in $V'$, we also have that $\alpha'(w(p))>0$ if and only if $\alpha'(w(x))>0$.
Thus $w(x)\in C_w' = w'(C_0')$ and by condition (\ref{condition: subroot_cells}), $w(x) = \sigma(w)(x)$.
Since $w$ and $\sigma(w)|_V$ are linear maps that agree on an open neighborhood, we have that $w =\sigma(w)|_{V}$.
    
It remains to show that $\sigma \colon W \to W'$ is an injective group homomorphism. 
We use the action of $W'$ on $\Sigma'$ once more to note that
\begin{align*}
    w'(H_{\alpha'}^+)&= \{ w'(v') \in V'\colon \alpha'(v')\geq 0 \}
    = \{ v' \in V' \colon \alpha'((w')^{-1}(v'))\geq 0 \} \\
    & =\{ v' \in V' \colon w'(\alpha')(v')\geq 0 \}
    =H_{w'(\alpha')}^+
\end{align*}
for all $w' \in W'$ and $\alpha' \in \Sigma'$. For $w=\Id \in W$, $C_w'=C_0'$, so $\sigma(w)=\Id \in W'$. If now $w_1$, $w_2 \in W$, then 
\begin{align*}
    \sigma(w_1)\sigma(w_2)(C_0') 
    &= \sigma(w_1)(C_{w_2}')
    = \sigma(w_1)(F) \cap \!\! \!\!
     \bigcap_{\substack{ \alpha' \in \Sigma' \\ \alpha' (w_2(p)) >0}} \!\! \!\!  \sigma(w_1) \left( H_{\alpha'}^+ \right) \\
     &
     = F 
     \cap \!\! \!\! \!\! \!\! 
     \bigcap_{\substack{ \beta' \in \Sigma' \\ \beta' (\sigma(w_1)w_2(p)) >0}} \!\! \!\! \!\! \!\!  H_{\beta'}^+ 
     = F  
     \cap \!\! \!\! \!\! \!\! 
     \bigcap_{\substack{ \beta' \in \Sigma' \\ \beta' (w_1 w_2(p)) >0}} \!\! \!\! \!\! \!\!  H_{\beta'}^+ 
     = C_{w_1 w_2}',
\end{align*}
where we used the substitution $\beta' = \sigma(w_1)( \alpha' )$, $\alpha' = \beta' \circ \sigma(w_1)$ and the fact that $\sigma(w_1)(w_2(p))=w_1 w_2 (p)$ since $w_2(p)\in V$.
By the definition of $\sigma$, this means that $\sigma(w_1 w_2)=\sigma(w_1)\sigma(w_2)$.
Finally $\ker(\sigma) = \{\Id_{V}\}$.
Indeed, all other elements $w\in W\setminus \{\Id_{V}\}$ send $p$ to $w(p)=\sigma(w)(p) \neq p$, since $p$ is regular in $V$.
This concludes the proof that $\sigma \colon W \to W'$ is an injective group homomorphism with $\sigma(w)|_V = w$.
\end{proof}

\subsubsection{Subgroups}
\label{sec:functoriality_subgroups}
In this section, we show that the symmetric buildings constructed from subgroups induce an injective morphism of symmetric buildings.
Examples include $\operatorname{SL}_m < \operatorname{SL}_n$ for $m\leq n$ or $\operatorname{Sp}_{2n} < \operatorname{SL}_{2n}$ (see \Cref{fig:intro:Sp4inSL4}), as well as inclusions $\operatorname{SL}_2 < \mathbf{G}$ or $\operatorname{PGL}_2 < \mathbf{G}$ arising from the Jacobson--Morozov theorem, see \cite[(6.1.3.b.2), (6.2.3.b)]{BruhatTits} and \cite[Section 6.8]{appenzeller2024semialgebraic}.

For the remainder of this section, $\ff$ is a real closed field and $v\colon \ff^\times \to \Lambda$ an order-compatible valuation.
Let $\mathbf{G}<\mathbf{G'}<\operatorname{SL}_n$ be two Zariski-connected semisimple self-adjoint linear algebraic $\qq$-groups with reduced root systems.
Let $\mathbf{S}<\mathbf{S'}$ be maximal $\rr$-split tori of $\mathbf{G}$, $\mathbf{G'}$ consisting only of self-adjoint ($g= g^T$) elements and let $A_\ff, A_\ff'$ be the semialgebraically connected components of the $\ff$-extensions $\mathbf{S}_\ff, \mathbf{S}_\ff'$ of $\mathbf{S}$ respectively $\mathbf{S'}$ that contain the identity.

The goal is to show that the inclusion $\mathbf{G}(\ff)<\mathbf{G'}(\ff)$ induces an injective morphism from the symmetric building $\build$ associated to $\mathbf{G}(\mathbb{F})$ to the symmetric building $\build'$ associated to $\mathbf{G'}(\mathbb{F})$.
Let $\sphwg$, $\sphwg'$ denote the (spherical) Weyl groups of the root systems $\Sigma \subseteq \fraka^\star$, $\Sigma' \subseteq (\fraka')^\star$, where $\fraka$, $\fraka'$ are the Lie algebras of $A_\rr$, $A_\rr'$.
For more detailed definitions see \Cref{ex:SymmetricBuilding1}. 

The following lemmas are used to constuct a morphism of apartments $\aa \to \aa'$ in \Cref{prop:FunctorialitySubgroups_apartments}.
In particular, we verify the assumption (\ref{condition: subroot_cells}) of \Cref{prop:subroot}. We would like to express our gratitude for communication with Anne Parreau which lead to some of the ideas in this section.

\begin{lemma}\label{lem:condition_cells_for_symmetric_building}
    Let $C_0$ and $C_0'$ be fundamental Weyl chambers in $\fraka$, $\fraka'$. 
    For all $w\in \sphwg$ and $w' \in \sphwg'$, if $x\in C_0 \cap C_0'$ satisfies $w(x) \in w'(C_0')$, then $w(x) = w'(x)$.
\end{lemma}
\begin{proof}
Let $w\in \sphwg$, $w' \in \sphwg'$ and $x\in C_0\cap C_0'$ such that $w(x)\in w'(C_0')$. 
    Denote by $\exp \colon \fraka' \to X_\rr'$ the Riemannian exponential map to the real symmetric space $X_\rr'$ and note that $\exp(\fraka) \subseteq X_\rr$. Recall from \cite[Proposition 7.32]{Knapp02} that there exist $k_0 \in K_\rr \subseteq K_\rr'$ and $k_0' \in K_\rr'$ such that $k_0.\exp(H)= \exp(w(H))$ for all $H\in \fraka$ and $k_0'.\exp(H')= \exp(w'(H'))$ for all $H'\in \fraka'$.
    Recall from \Cref{ex:SymmetricBuilding1} that the Cartan decomposition $G_{\mathbb{R}}=K_{\mathbb{R}}A_{\mathbb{R}}K_{\mathbb{R}}$ (postcomposed with $\log \colon A_\rr \to \fraka$) can be used to obtain a Cartan projection $\delta \colon X_\rr' \to C_0'$, $k_1'ak_2'.\Id \mapsto \log(a)$  which is invariant under the action of $K_\rr'$ in the sense that for all $q' \in X_\rr'$, $k' \in K_\rr'$
    \[
        \delta(k'.q') = \delta(q'), \text{ and }  \delta(\exp(y')) = y' \text{ for all $y' \in C_0'$}.
    \]
    We have $(w')^{-1}w(x) \in C_0'$ so that
    \[
    (w')^{-1}w(x)  = \delta( \exp((w')^{-1}w(x) ) ) = \delta((k_0')^{-1}k_0.\exp(x)) = \delta(\exp(x)) = x.
    \]
    Hence $w(x)=w'(x)$, which concludes the proof.
\end{proof}

Recall that the symmetric building $\build$ associated to $\mathbf{G}(\ff)$ has type $\aa = \aa(\Sigma^\vee, \Lambda, T)$ with full translation group $T=\aa$ and the symmetric building $\build'$ associated to $\mathbf{G}'(\ff)$ has type $\aa' = \aa((\Sigma')^\vee, \Lambda, T')$ with full translation group $T'=\aa'$.

\begin{lemma}\label{lem:FunctorialitySubgroups_L}
    The inclusion $\fraka \subseteq \fraka'$ restricts to an inclusion $\operatorname{Span}_{\qq}(\Sigma^\vee) \subseteq \operatorname{Span}_{\qq}((\Sigma')^\vee)$.
\end{lemma}
\begin{proof}
    Since we are in the context of algebraic groups, the coroot lattice $\operatorname{Span}_{\qq}(\Sigma^\vee)$ can be identified with the cocharacter lattice $X^\star(\mathbf{S}) \otimes_{\zz} \qq$ (for the root lattice and the characters, this identification is spelled out in \cite[Section 6.2]{appenzeller2024semialgebraic}).
    Elements in $X^\star(\mathbf{S})$, so-called cocharacters, are algebraic homomorphisms $\mathbb{G}_m \to \mathbf{S}$. 
    Every cocharacter of $\mathbf{S}$ gives rise to a cocharacter of $\mathbf{S}'$ by postcomposing with the inclusion $\mathbf{S} < \mathbf{S}'$, so 
    \[
        \operatorname{Span}_{\qq}(\Sigma^\vee)\cong X^\star(\mathbf{S})\otimes_\zz \qq \subseteq    X^\star(\mathbf{S}')\otimes_\zz \qq\cong\operatorname{Span}_\qq( (\Sigma')^\vee ).\qedhere
    \]
\end{proof}

\begin{prop}\label{prop:FunctorialitySubgroups_apartments}
    There is an injective morphism of apartments $\aa \to \aa'$.
\end{prop}
\begin{proof}
    Let $L \colon \operatorname{Span}_{\qq}(\Sigma^\vee) \to \operatorname{Span}_{\qq}((\Sigma')^\vee)$ be the inclusion from \Cref{lem:FunctorialitySubgroups_L} and $\gamma \colon \Lambda \to \Lambda$ the identity. We choose a fundamental Weyl chamber $C_0$ in $\fraka$ and a point $p\in \fraka$ regular in $\fraka$. Then choose a fundamental Weyl chamber $C_0'$ in $\fraka'$.
    Then $p$ is $\fraka$-regular in $\fraka'$, and by \Cref{lem:condition_cells_for_symmetric_building}, the conditions for \Cref{prop:subroot} are satisfied, so we obtain a group homomorphism $\sigma_s \colon \sphwg \to \sphwg'$ such that $\sigma_s(w)|_{\operatorname{Span}_{\qq}(\Sigma^\vee)} = w$ for all $w\in \sphwg$.
    The diagram
    \[
    \begin{tikzcd}
        \mathbb{A} \arrow[r, "L \otimes \gamma"] \arrow[d, "w"'] & \mathbb{A}' \arrow[d, "\sigma_s(w)"] \\
        \mathbb{A} \arrow[r, "L \otimes \gamma"']                & \mathbb{A}'
    \end{tikzcd}
    \]
    commutes for all $w\in \sphwg$.
    Moreover $T = \aa \subseteq \aa' = T'$, so by \Cref{lem:morphism_apartment}, $\sigma_s$ extends to $\sigma\colon \affwg \to \affwg'$ so that $(L,\gamma,\sigma)$ is a morphism of apartments. Both $L$ and $\gamma$ are injective, so the morphism is injective.
\end{proof}

We now have all the tools to prove \Cref{intro:thm:FunctorialitySubgroups}.

\begin{thm}[\Cref{intro:thm:FunctorialitySubgroups}]
\label{thm:FunctorialitySubgroups}
    Let $\mathbf{G}<\mathbf{G'}<\operatorname{SL}_n$ be two connected semisimple self-adjoint linear algebraic $\qq$-groups and $\mathbf{S}<\mathbf{S'}$ maximal $\rr$-split tori of $\mathbf{G}$, $\mathbf{G'}$ (all of whose elements are self-adjoint) such that the root systems are reduced.
    Let $\build$ and $\build'$ be the associated symmetric buildings.
    If $\mathbf{G}$ is $\rr$-split, then the inclusion $\mathbf{G}(\ff) \subseteq \mathbf{G}'(\ff)$ induces an equivariant injective morphism $\build \to \build'$ of buildings.
\end{thm}
\begin{proof}
    Let $o = [\Id] \in \build$ and $o' = [\Id] \in \build'$ be the base points.
    By \Cref{prop:FunctorialitySubgroups_apartments}, the inclusion $\tau \colon \mathbb{A} \to \mathbb{A}'$ is an injective morphism of apartments.
    We consider the standard charts $f_0 \colon \mathbb{A} \to \build$ and $f_0' \colon \mathbb{A}' \to \build'$ and apply \Cref{thm:baby_morphismofGbuildings}. 

By \Cref{prop:homogeneous_stabilizers}, $\operatorname{Stab}_{\mathbf{G}(\ff)}(o) = \mathbf{G}(\mathcal{O}) \subseteq \mathbf{G}'(\mathcal{O}) = \operatorname{Stab}_{{\mathbf{G}'(\ff)}}(o')$, so condition \ref{condition:baby_stabilizerofapoint} holds. Using that $\mathbf{G}$ is $\ff$-split in \Cref{prop:homogeneous_stabilizers}, we obtain that 
\[
    \operatorname{Stab}_{\mathbf{G}(\ff)}(f_0) = \mathbf{S}(\mathcal{O}) \subseteq \mathbf{S}'(\mathcal{O}) \subseteq \operatorname{Cent}_{\mathbf{G}'(\ff)}(\mathbf{S}'(\ff)) \cap \mathcal{O}^{n\times n} = \operatorname{Stab}_{\mathbf{G}'(\ff)}(f_0'),
    \]
    so condition \ref{condition:baby_stabilizerofchart} holds. For every $x\in \aa$ there exists $a\in A_\ff$ with $a.o = f_0(x)$. Consider $\alpha' \in \Sigma'$.
    Then $\alpha'|_{\operatorname{Span}_{\qq}(\Sigma^\vee)} = \sum_{i=1}^r q_i \alpha_i$ for some $\alpha_i \in \Sigma$, since in the theory of algebraic groups, the restriction of $\alpha'$ is a character of $\mathbf{G}(\ff)$.
    By (\ref{eq:BH_compatible}), we obtain
    \begin{align*}
        (-v)(\chi_{\alpha'}(a)) 
        &= (-v)\left( \prod_{i=1}^r \chi_{\alpha_i}(a)^{q_i}\right) 
        = \sum_{i=1}^r q_i(-v)(\chi_{\alpha_i}(a)) \\
        &= \sum_{i=1}^r q_i \alpha_i(x) 
        = \alpha'(x),
    \end{align*}
    which shows that $a.o' = f_0'(\tau(x))$, which is condition \ref{condition:baby_Afw}.
    Thus $m \coloneqq (\psi,\varphi,\tau)$ is an equivariant morphism of buildings.

    We check injectivity of $m$ directly, as the condition of \Cref{thm:baby_morphismofGbuildings}~\ref{thm:baby_monomorphism} is not satisfied.
    The map $\psi \colon \build \to \build'$ is injective, since if $g,h \in \mathbf{G}(\ff)$ satisfy $g.o'=h.o'$, then $h^{-1}g \in \operatorname{Stab}_{\mathbf{G}'(\ff)}(o')=\mathbf{G}'(\mathcal{O})$, see \Cref{prop:homogeneous_stabilizers}, but also $h^{-1}g \in \mathbf{G}(\ff)$, so $h^{-1}g \in \mathbf{G}(\ff) \cap \mathbf{G}'(\mathcal O) =  \mathbf{G}(\mathcal{O})=\operatorname{Stab}_{{\mathbf{G}(\ff)}}(o)$, so $g.o = h.o$. 
    Similarly, again using \Cref{prop:homogeneous_stabilizers}, the map $\varphi \colon \mathcal{A} \to \mathcal{A}'$ is injective because
    \begin{align*}
        &\operatorname{Stab}_{\mathbf{G}'(\ff)}(f_0') \cap \mathbf{G}(\ff) 
        = \operatorname{Cent}_{\mathbf{G}'(\ff)}(\mathbf{S}'(\ff)) \cap {\mathcal O }^{n\times n} \cap \mathbf{G}(\ff)  \\
        &=  \operatorname{Cent}_{\mathbf{G}(\ff)}(\mathbf{S}'(\ff)) \cap {\mathcal O }^{n\times n}
        \subseteq  \operatorname{Cent}_{\mathbf{G}(\ff)}(\mathbf{S}(\ff)) \cap {\mathcal O }^{n\times n}
        = \operatorname{Stab}_{\mathbf{G}(\ff)}(f_0),
    \end{align*}
    since $\mathbf{S}(\ff) \subseteq \mathbf S' (\ff)$.
    
    Finally, $\tau$ is an injective morphism of apartments, so $m$ is an equivariant injective morphism of buildings.
\end{proof}

\begin{remark}
Most of the assumptions in Theorem \ref{thm:FunctorialitySubgroups} are just there to guarantee the existence of the buildings $\build$ and $\build'$. However, the assumption that $\mathbf{G}$ is $\ff$-split is a technical assumption that implies $A_\ff M_\ff = \operatorname{Cent}_{\mathbf{G}(\ff)}(\mathbf{S}(\ff)) = \mathbf{S}(\ff)$ and allows to define the map $\varphi \colon \atlas \to \atlas'$. The same proof goes through when replacing that condition with $\operatorname{Cent}_{\mathbf{G}(\ff)}(\mathbf{S}(\ff)) \subseteq \operatorname{Cent}_{\mathbf{G}'(\ff)}(\mathbf{S}'(\ff))$.
The existence of a morphism between the associated symmetric buildings might still be true without the assumption that $\mathbf{G}$ is $\ff$-split, but then one has to deal with the anisotropic part of $\mathbf{T}$.
Examples where our proof does not work verbatum are given by $\mathbf{G}=\operatorname{SO}(3)$ (with the building consisting of a single point) in $\mathbf{G}'=\operatorname{SL}_3$, and $\operatorname{SO}(2,3)<\operatorname{SL}_5$.
\end{remark}

\subsubsection{Injective group morphisms}
The goal of this section is to generalize the result on subgroups to general injective morphisms of algebraic groups.
Note that we do not assume in this subsection that the algebraic groups in question are $\ff$-split.

We first show that if $\mathbf{G}$, $\mathbf{G}'$ are isomorphic, then their associated symmetric buildings are isomorphic.
In particular the buildings are independent of the choice of maximal self-adjoint $\rr$-split torus $\mathbf{S}$.
This is again an application of \Cref{thm:morphismofGbuildings}.
Together with the results of the previous subsection we then obtain the following.

\begin{prop} \label{prop:isomorphism_of_group_implies_isomorphism_of_buildings}
    Let $\mathbf{G}< \operatorname{SL}_{n}$ and $\mathbf{G}'< \operatorname{SL}_m$ be semisimple connected self-adjoint linear algebraic $\qq$-groups with reduced root systems and let $\build$, $\build'$ be the associated symmetric buildings.
    If $\rho \colon \mathbf{G} \to \mathbf{G}'$ is an isomorphism of algebraic groups defined over $\qq$ (or more generally over $\rr \cap \ff$), then $\build$ is (equivariantly) isomorphic to $\build'$.
\end{prop}
\begin{proof}
    Let $\mathbf{S}$ and $\mathbf{S}'$ be maximal $\rr$-split tori of $\mathbf{G}$ and $\mathbf{G}'$, all of whose elements are self-adjoint. We note that $\rho(\mathbf{S})$ is again a maximal $\rr$-split torus, but it may not coincide with $\mathbf{S}'$. However, by \cite[Theorem 15.14]{Borellinearalgebraixgroups}, all $\rr$-split (and hence $(\rr\cap \ff)$-split \cite[Theorem 5.17]{appenzeller2024semialgebraic}) tori are conjugated in $\mathbf{G}(\rr\cap \ff)$, so up to postcomposing $\rho$ with this conjugation, we may assume that $\rho(\mathbf{S})=\mathbf{S}'$.
    
    Let $A_\rr$ and $A_\rr'$ be the semialgebraically connected components of $\mathbf{S}(\rr)$ and $\mathbf{S}'(\rr)$ that contain the identity and let $\fraka$ and $\fraka'$ be their respective Lie algebras. 
    The Lie algebra isomorphism $\operatorname{D}_{\Id}\rho \colon \operatorname{Lie}(\mathbf{G}(\rr)) \to \operatorname{Lie}(\mathbf{G}'(\rr))$ satisfies $\operatorname{D}_{\Id}\rho (\fraka) = \fraka'$ and sends root spaces to root spaces, inducing an isomorphism $\Sigma \to \Sigma'$, $\alpha \mapsto \alpha \circ \operatorname{D}_{\Id}\rho^{-1}$.
    This restricts to an isomorphism 
    \[
    L \colon \operatorname{Span}_{\qq}(\Sigma^\vee) \to \operatorname{Span}_\qq((\Sigma')^\vee), \quad  H \mapsto \operatorname{D}_{\Id}\rho (H),
    \]
    and induces an isomorphism 
    \[
    \sigma_s \colon \sphwg \to \sphwg', \quad w \mapsto \operatorname{D}_{\operatorname{Id}}\rho \circ w \circ \operatorname{D}_{\Id}\rho^{-1}.
    \]
    Together with the identity $\gamma=\Id_\Lambda \colon \Lambda \to \Lambda$, \Cref{lem:morphism_apartment} can be used to extend $\sigma_s$ to $\sigma\colon \affwg \to \affwg'$ such that $\tau \coloneqq (L,\gamma, \sigma)$ forms an isomorphism of apartments.

    We now check the conditions of \Cref{thm:baby_morphismofGbuildings}.
    The action of $\mathbf{G}(\ff)$ is transitive on $\build$ and $\atlas$.
    Let $f_0 \colon \aa \to \build$ and $f_0' \colon \aa' \to \build'$ be the standard charts.
    Since $\rho$ and $\rho^{-1}$ are defined over $\qq$, $\rho$ and $\rho^{-1}$ are entrywise defined by polynomials with coefficients in $\qq$ \cite[Lemma 4.1]{appenzeller2024semialgebraic}, so $\rho(\mathbf{G}(\mathcal{O})) = \mathbf{G}'(\mathcal{O})$ and by \Cref{prop:homogeneous_stabilizers}, $\rho(\operatorname{Stab}_{\mathbf{G}(\ff)}(f_0(0))) = \operatorname{Stab}_{\mathbf{G}'(\ff)}(f_0'(0))$ which is condition \ref{condition:baby_stabilizerofapoint}.

Since $\rho(A_\ff) = A_\ff'$, we have $\rho(\operatorname{Cent}_{\mathbf{G}(\ff)}(A_\ff) ) = \operatorname{Cent}_{\mathbf{G}'(\ff)}(A_\ff')$ and by \Cref{prop:homogeneous_stabilizers}, $\rho(\operatorname{Stab}_{\mathbf{G}(\ff)}(f_0) = \operatorname{Stab}_{\mathbf{G}'(\ff)}(f_0')$, which is condition \ref{condition:baby_stabilizerofchart}. 
For every $x = \sum_{\delta\in \Delta} \lambda_\delta \delta^\vee \in \aa$, there is some $a\in A_\ff$ such that $f_0(x) = a.o$ with $(-v)(\chi_\alpha(a)) = \alpha(x)$ for all $\alpha \in \Sigma$.
    By \cite[Lemma 6.1]{appenzeller2024semialgebraic} over $\rr$, for all $H\in \fraka$ and $\alpha \in \Sigma$, we have
    \[
    \chi_\alpha(\exp(H)) = e^{\alpha(H)} = e^{\alpha \circ \operatorname{D}_{\Id}\rho^{-1}  ( \operatorname{D}_{\Id}\rho (H) )  } = \chi_{\alpha \circ  \operatorname{D}_{\Id}\rho^{-1}} (\rho(\exp(H))).
    \]
    By the transfer principle (\Cref{thm_TarskiSeidenberg}), for $\alpha \in \Sigma$ and the corresponding $\alpha' \coloneqq \alpha \circ \operatorname{D}_{\Id} \rho^{-1} \in \Sigma'$, we have 
    \[\chi_\alpha(a) = \chi_{\alpha'}(\rho(a)).\]
    Thus
    \begin{align*}
        (-v)(\chi_{\alpha'}(\rho(a)) ) 
        &= (-v)(\chi_{\alpha}(a) )
        = \alpha(x) = \sum_{\delta \in \Delta} \lambda_\delta \alpha \left( \delta^\vee \right) \\
        & = \sum_{\delta \in \Delta} \lambda_\delta \left(\alpha \circ \operatorname{D}_{\Id}\rho^{-1}\right)(L(\delta^{\vee})) 
        = \alpha'\left( \tau (x) \right),
    \end{align*}
    so $\rho(a).o' = f_0'(\tau (x))$ and condition \ref{condition:baby_Afw} is satisfied. By \Cref{thm:baby_morphismofGbuildings}~\ref{thm:baby_isomorphism} we obtain that $m$ is an isomorphism of buildings $\build \xrightarrow{\sim} \build'$.
\end{proof}

We can now prove our most general result about functoriality for monomorphisms.
A morphism of linear algebraic groups $\rho \colon \mathbf{G} \to \mathbf{G}'$ is called \emph{transposition-invariant} if $\rho(g)^\top = \rho(g^\top)$ for all $g\in \mathbf{G}$.

\begin{thm}[\Cref{intro:thm:InjMorphBuildings}] \label{thm:inj_group_morphism_implies_inj_morphism_buildings}
    Let $\mathbf{G}< \operatorname{SL}_{n}$ and $\mathbf{G}'< \operatorname{SL}_m$ be semisimple connected self-adjoint linear algebraic $\qq$-groups with reduced root systems and let $\build$, $\build'$ be the associated symmetric buildings.
    If $\mathbf{G}$ is $\rr$-split and $\rho \colon \mathbf{G} \to \mathbf{G}'$ is a transposition-invariant monomorphism of algebraic groups defined over $\qq$ (or more generally over $\rr \cap \ff$), then there is an equivariant injective morphism $\build \to \build'$.
\end{thm}
\begin{proof}
    The image $\rho(\mathbf{G})$ of a morphism $\rho\colon \mathbf{G} \to \mathbf{G}'$ of algebraic groups is an algebraic group, and since $\rho$ is a monomorphism, $\mathbf{G}$ and $\rho(\mathbf{G})$ are isomorphic \cite[Corollary 5.22]{Milne}.
    In particular $\rho(\mathbf{G})$ is also semisimple, connected, $\rr$-split and has reduced root system. Since $\rho$ is transposition-invariant, $\rho(\mathbf{G})$ is self-adjoint ($g\in \rho(\mathbf{G})$ implies $g^\top \in \rho(\mathbf{G})$).

    Let $\mathbf{S}<\mathbf{G}$ and $\mathbf{S}'<\mathbf{G}'$ be maximal $\rr$-split tori consisting of self-adjoint elements.
    Now $\rho(\mathbf{S})$ is also an $\rr$-split torus, and it lies in a maximal $\rr$-split torus, which is conjugated by some element $g \in \mathbf{G'}(\rr \cap \ff)$ to $\mathbf{S}'$.
    In fact, by \cite[Theorem 6.51]{Knapp02}, even $g\in \mathbf{K}'(\rr \cap \ff) = \mathbf{G}'(\rr \cap \ff) \cap \operatorname{SO}_m(\rr \cap \ff)$ (here we used that $\rho$ has self-adjoint image, to make sure that $\rho(\fraka) \subseteq \frakp'$ for the Cartan decomposition $\operatorname{Lie}(\mathbf{G}'(\rr)) = \frakp' \oplus \frakk'$ associated to the standard Cartan involution $X \mapsto -X^\top$).
    If $c_g$ denotes that conjugation, $c_g(\rho(\mathbf{G}))$ is isomorphic to $\mathbf{G}$, so $c_g(\rho(\mathbf{G}))$ is semisimple, connected, $\rr$-split and has reduced root system.
    Since $g \in \mathbf{K}(\rr \cap \ff)$, $c_g(\rho(\mathbf{G}))$ is self-adjoint.
    Moreover $c_g(\rho(\mathbf{S})) \subseteq \mathbf{S}'$ is a $\rr$-split torus (maximal in $c_g(\rho(\mathbf{G}))$) and all its elements are self-adjoint (since those of $\mathbf{S}$ are).    
    Denote by $\rho(B)$ the building associated to $\rho(\mathbf{G})$ and by $c_g(\rho(\build))$ the building associated to $c_g(\rho(\mathbf{G}))$.
    Then by \Cref{prop:isomorphism_of_group_implies_isomorphism_of_buildings}, we get equivariant morphisms $\build \cong \rho(\build) \cong c_g(\rho(\build))$ and by \Cref{thm:FunctorialitySubgroups} we obtain an equivariant injective morphism
    $\build  \cong c_g(\rho(\build)) \to \build'$.
\end{proof}

\appendix
\section{Basics from real algebraic geometry}
\label{appendix}

\subsection{Ordered abelian groups and ordered fields}
\label{subsection:OrderedFields}
An \emph{ordered field} $\ff$ is a field with a total order that is compatible with the field operations.
An ordered field $\ff$ is called \emph{non-Archimedean} if there exists $x \in \ff$ with $x>n$ for all $n\in \nn$.
We ask any valuation $v$ on an ordered field $\ff$ to be \emph{order-compatible}, meaning that for all $x,y \in \ff$ with $0 < x \leq y$ we have $v(x) \geq v(y)$, i.e.\ $-v$ restricted to $\ff_{>0}$ is order-preserving. 
In this case, the field $\ff$ is necessarily non-Archimedean.

Two elements $x, y$ of an ordered abelian group $\Lambda$ (resp.\ an ordered field $\ff$) are in the same \emph{Archimedean class} if there exists $n \in \nn$ such that $|x|<n|y|$ and $|y|<n|x|$, where $|x|\coloneqq \max\{x,-x\}$. The set of Archimedean classes is called the \emph{rank} of $\Lambda$ (resp.\ $\ff$) and forms an ordered abelian group $\operatorname{rk}(\Lambda)$ (resp.\ $\operatorname{rk}(\ff)$), where addition and the order are induced from the ones in $\Lambda$ (resp.\ $\ff$). 
We have the following important theorem about ordered abelian groups.

\begin{thm}[Hahn's embedding theorem \cite{Gravett_OrderedAbelianGroups}]\label{thm:hahn}
    Every ordered abelian group $\Lambda$ can be viewed as a subgroup of the abelian additive group
\[
\mathfrak{R} \coloneqq \{ (x_i)_{i\in I} \in \rr^{I} \colon \operatorname{supp}((x_i)_{i\in I}) \text{ is a well-ordered subset of }I  \}
\]
where $I \coloneqq \operatorname{rk}(\Lambda)$, endowed with a lexicographical order.
\end{thm}

The following is a general construction to define order-compatible valuations on ordered fields.

\begin{example}[Order valuation]\label{ex:OrderValuation}
    Let $\ff$ be an ordered field and $\Lambda = \operatorname{rk}(\ff)$ the ordered abelian group of Archimedean classes. The map $ 0 \neq x \mapsto -[x] \in \Lambda$ is an order-compatible valuation called the \emph{order valuation}.
    Note that $\Lambda \neq \{0\}$ if and only if $\ff$ is non-Archimedean.
\end{example}

An ordered field can admit many order-compatible valuations, where the order valuation is in some sense the ``finest'' one, meaning that if $v$ is any other valuation it factors through the order valuation.
For example, if $\ff$ has a \emph{big element} $b$, i.e.\ for all $x \in \ff$ there exists $n\in \nn$ with $x<b^n$, then one can define an order-compatible rank-$1$ valuation $v_b \colon \ff^\times \to \rr$ by setting
\[ v_b(x) \coloneqq -\inf\left\{\frac{p}{q} \in \qq \colon x^q<b^p\right\},\]
mimicking the definition of the standard logarithm.
Any order-compatible rank-$1$ valuation on $\ff$ is a positive scalar multiple of $v_b$ for some big element $b\in\ff$.
Note that there are ordered fields that do not admit big elements, such as the hyperreals.

An ordered field is \emph{real closed} if every positive element is a square and every odd degree polynomial has a root.
Note that every ordered field has a \emph{real closure}, that means an algebraic field extension that is real closed and whose order extends the original one \cite[\S 1.3]{BochnakCosteRoy_RealAlgebraicGeometry}.

\begin{example}
\label{example:OrderedFields}
The real numbers $\rr$ and the real algebraic numbers $\overline{\qq}^r$ are both real closed.
The field of \emph{real Puiseux series} is the set of expressions
\[
\rr(X)^\wedge \coloneqq \Bigl\{  \sum_{k=k_0}^{\infty} c_k X^{k/m} \, \Bigm| \, k_0 \in \zz, \, m \in \nn \setminus\{0\} , \, c_k \in \rr \Bigr\},
\]
together with formal addition and multiplication.
An element $\sum_{k=k_0}^{\infty} c_k X^{k/m}$ is positive if $c_{k_0} > 0$.
With this order $\rr(X)^\wedge$ is real closed, see e.g.\ \cite[Theorem 2.91]{BasuPollackRoy_AlgorithmsRAG}.
The real closure of $\qq$ is $\overline{\qq}^r$.
The real closure of $\rr(X)$ (together with the order $X>0$ but $X<\lambda$ for all $\lambda \in \rr_{>0}$) is the field of real Puiseux series that are algebraic over $\rr(X)$.
\end{example}

Real closed fields play a crucial role in real algebraic geometry.

\subsection{Semialgebraic geometry}
We summarize general definitions and results from real algebraic geometry and set up notation.
We refer the reader to \cite{BochnakCosteRoy_RealAlgebraicGeometry}, in particular Chapters 1, 2 and 5, for more details and proofs.
The main objects of study in real algebraic geometry are semialgebraic sets.
From now on let $\ff$ be a real closed field.

\begin{defn}
\label{dfn_SemiAlgSet}
A subset $\mathcal{B} \subseteq \ff^n$ is a \emph{basic semialgebraic set}, if there exists a polynomial $f \in \ff[X_1,\ldots,X_n]$ such that
\[ \mathcal{B} = \mathcal{B}(f)= \{ x \in \ff^n \colon f(x)>0\}.\]
A subset $X \subseteq \ff^n$ is \emph{semialgebraic} if it is a Boolean combination of basic semialgebraic sets, i.e.\ $X$ is obtained by taking finite unions and intersections of basic semialgebraic sets and their complements.
Let $X \subseteq \ff^n$ and $Y \subseteq \ff^m$ be two semialgebraic sets.
A map $f \colon X \to Y$ is called \emph{semialgebraic} if its graph $\textrm{Graph}(f) \subseteq X \times Y$ is semialgebraic in $\ff^{n+m}$.
\end{defn}
Algebraic sets are semialgebraic and any polynomial or rational map is semialgebraic.

Note that if $\ff\neq \rr$, then $\ff$ is totally-disconnected in the order topology on $\ff$.
However we have the following notion of connectedness for semialgebraic sets.

\begin{defn} \label{dfn_SemiAlgConn}
A semialgebraic set $X \subseteq \ff^n$ is \emph{semialgebraically connected} if it cannot be written as the disjoint union of two non-empty semialgebraic subsets of $\ff^n$ both of which are closed in $X$.
\end{defn}

\begin{thm}[{\cite[Theorem 2.4.5]{BochnakCosteRoy_RealAlgebraicGeometry}}]
\label{thm_RConnSemiAlgConn}
A semialgebraic set of $\rr^n$ is connected if and only if it is semialgebraically connected.
Every semialgebraic set of $\rr^n$ has a finite number of connected components, which are semialgebraic.
\end{thm}

From now on, denote by $\kk$ a real closed extension of $\ff$.

\begin{defn}
\label{dfn_ExtSemiAlgSets}
Let $X \subseteq \ff^n$ be a semialgebraic set given as
\[ X = \bigcup_{i=1}^s\bigcap_{j=1}^{r_i} \{x \in \ff^n \colon f_{ij}(x) \ast_{ij} 0 \},\]
with $f_{ij} \in \ff[X_1,\ldots,X_n]$ and $\ast_{ij}$ is either $<$ or $=$ for $i = 1, \ldots,s$ and $j=1,\ldots,r_i$.
The \emph{$\kk$-extension $X_{\kk}$} of $X$ is the set given by the same Boolean combination of sign conditions as $X$, more precisely
\[X_{\kk} = \bigcup_{i=1}^s\bigcap_{j=1}^{r_i} \{x \in \kk^n \colon f_{ij}(x) \ast_{ij} 0 \}.\]
\end{defn}

Note that $X_{\kk}$ is semialgebraic and depends only on the set $X$, and not on the Boolean combination describing it, see \cite[Proposition 5.1.1]{BochnakCosteRoy_RealAlgebraicGeometry}.
The proof of this is based on the Tarski--Seidenberg transfer principle.

\begin{thm}[Tarski--Seidenberg transfer principle, {\cite[Theorem 5.2.1]{BochnakCosteRoy_RealAlgebraicGeometry}}]
\label{thm_TarskiSeidenberg}
Let $X \subseteq \ff^{n+1}$ be a semialgebraic set.
Denote the projection $\pr \colon \ff^{n+1} \to \ff^n$ onto the first $n$ coordinates by $\pr$.
Then $\pr(X) \subseteq \ff^n$ is semialgebraic.
Furthermore, if $\kk$ is a real closed extension of $\ff$, and $\pr_{\kk} \colon \kk^{n+1} \to \kk$ is the projection on the first $n$ coordinates, then
\[\pr_{\kk}(X_{\kk}) =\pr(X)_{\kk}.\]
\end{thm}

Finally, we have the following relation between extension of semialgebraic sets and semialgebraically connected components.

\begin{thm}[{\cite[Proposition 5.3.6 (ii)]{BochnakCosteRoy_RealAlgebraicGeometry}}]
\label{thm_ExtConnComp}
Let $X \subseteq \ff^n$ be semialgebraic.
Then $X$ is semialgebraically connected if and only if $X_{\kk}$ is semialgebraically connected. 
More generally, if $C_1, \ldots, C_m$ are the semialgebraically connected components of $X$, then $(C_1)_{\kk}, \ldots, (C_m)_{\kk}$ are the semialgebraically connected components of $X_{\kk}$.
\end{thm}

\bibliographystyle{amsalpha}
\bibliography{main}

\end{document}